\documentclass[11pt]{article}
\usepackage[margin=1in]{geometry}
\usepackage{amsmath,amsthm}
\usepackage{xcolor}
\usepackage[colorlinks,linkcolor=red!50!black,citecolor=green!40!black,urlcolor=magenta!60!black]{hyperref}
\usepackage[capitalize,nameinlink]{cleveref}

\theoremstyle{plain}
\newtheorem{theorem}{Theorem}[section]
\newtheorem{lemma}{Lemma}[section]
\newtheorem{proposition}{Proposition}[section]
\newtheorem{corollary}{Corollary}[section]
\theoremstyle{definition}
\newtheorem{definition}{Definition}[section]
\newtheorem{assumption}{Assumption}[section]
\newtheorem{example}{Example}[section]
\theoremstyle{remark}
\newtheorem{remark}{Remark}[section]

\newcommand{\email}[1]{\href{mailto:#1}{\nolinkurl{#1}}}
\newcommand{\headers}[2]{}
\newenvironment{keywords}{\par\medskip\noindent\textbf{Keywords:}\ }{\par}
\newenvironment{MSCcodes}{\par\noindent\textbf{MSC codes:}\ }{\par}

\makeatletter
\let\arx@newtheorem\newtheorem
\newcommand{\arx@gobblethm}[2]{\@ifnextchar[{\arx@gobbleopt}{}}
\def\arx@gobbleopt[#1]{}
\let\newtheorem\arx@gobblethm
\let\renewtheorem\arx@gobblethm
\let\arx@theoremstyle\theoremstyle
\let\theoremstyle\@gobble
\providecommand{\theoremheaderfont}[1]{}
\providecommand{\theorembodyfont}[1]{}
\providecommand{\theoremseparator}[1]{}
\providecommand{\theoremsymbol}[1]{}
\usepackage[T1]{fontenc}
\usepackage[utf8]{inputenc}
\usepackage{amsfonts,amssymb}
\usepackage[numbers]{natbib}
\usepackage{algorithm,algorithmic}
\usepackage{etoolbox}
\makeatletter
\g@addto@macro\fs@ruled{\def\@fs@pre{\begingroup\color{black}\hrule height.8pt depth0pt \kern2pt}\def\@fs@post{\kern2pt\hrule\relax\endgroup}}
\makeatother
\AtBeginEnvironment{algorithmic}{%
  \setlength{\abovedisplayskip}{3pt}%
  \setlength{\belowdisplayskip}{3pt}%
  \setlength{\abovedisplayshortskip}{0pt}%
  \setlength{\belowdisplayshortskip}{3pt}%
}
\usepackage{array}
\usepackage{graphicx}
\usepackage{xcolor}
\usepackage{url}
\usepackage[all]{hypcap}
\usepackage{booktabs, multirow, setspace}
\usepackage{tikz}
\usetikzlibrary{arrows.meta, positioning, fit, shapes.misc}
\usepackage{enumitem}
\usepackage{subcaption}
\usepackage{nicefrac, xspace, mathrsfs, bm}
\usepackage{nameref}
\usepackage{stfloats}
\usepackage{textcomp}
\usepackage{microtype}

\ifpdf
  \DeclareGraphicsExtensions{.eps,.pdf,.png,.jpg}
\else
  \DeclareGraphicsExtensions{.eps}
\fi

\setlist[enumerate]{leftmargin=.5in}
\setlist[itemize]{leftmargin=.5in}

\newcommand{\RR}{\mathbb{R}}

\newcommand{\EE}{\mathbb{E}}

\newcommand{\NN}{\mathbb{N}}

\newcommand{\bA}{\bm{A}}
\newcommand{\bB}{\bm{B}}

\newcommand{\bx}{\bm{x}}
\newcommand{\bb}{\bm{b}}
\newcommand{\bc}{\bm{c}}

\newcommand{\bu}{\bm{u}}
\newcommand{\bz}{\bm{z}}
\newcommand{\ba}{\bm{a}}
\newcommand{\by}{\bm{y}}
\newcommand{\be}{\bm{e}}

\newcommand{\bZ}{\bm{Z}}
\newcommand{\bY}{\bm{Y}}

\newcommand{\bLmbd}{\bm{\Lambda}}
\newcommand{\bg}{\bm{g}}
\newcommand{\bd}{\bm{d}}

\newcommand{\bv}{\bm{v}}
\newcommand{\bw}{\bm{w}}

\DeclareMathOperator*{\argmin}{\mathrm{argmin}}

\newcommand{\calL}{\mathcal{L}}
\newcommand{\calG}{\mathcal{G}}
\newcommand{\calS}{\mathcal{S}}
\newcommand{\calC}{\mathcal{C}}
\newcommand{\calJ}{\mathcal{J}}
\newcommand{\calD}{\mathcal{D}}
\newcommand{\calQ}{\mathcal{Q}}
\newcommand{\calO}{\mathcal{O}}

\newcommand{\sumn}{\sum_{i=1}^n}

\newcommand{\sumT}{\sum\limits_{k=1}^{T}}

\newcommand{\blue}[1]{{\color{blue}#1}}

\theoremstyle{plain}
\theoremheaderfont{\normalfont\sffamily}
\theorembodyfont{\normalfont\itshape}
\theoremseparator{.}
\theoremsymbol{}
\renewtheorem{lemma}{Lemma}[section]
\renewtheorem{proposition}{Proposition}[section]
\renewtheorem{definition}{Definition}[section]
\renewtheorem{corollary}{Corollary}[section]
\newtheorem{assumption}{Assumption}[section]
\newtheorem{example}{Example}[section]
\theoremheaderfont{\normalfont\itshape}
\theorembodyfont{\normalfont}
\newtheorem{remark}{Remark}[section]

\numberwithin{equation}{section}

\crefname{subsection}{Section}{Sections}
\Crefname{subsection}{Section}{Sections}

\title{DualHash: Stochastic Primal-Dual \blue{Algorithms} for \blue{W-Type Regularized} Deep Hashing}

\author{Luxuan Li\thanks{School of Mathematical Sciences, Beihang University, Beijing, 100191, China (\email{luxuanli@buaa.edu.cn}).}
\and Xiao Wang\thanks{School of Computer Science and Engineering, Sun Yat-Sen University, Guangzhou, 510006, China (\email{wangx936@mail.sysu.edu.cn}).}
\and Chunfeng Cui\thanks{School of Mathematical Sciences, Beihang University, Beijing, 100191, China, and State Key Laboratory of Virtual Reality Technology and Systems, Beihang University, Beijing, 100191, China (\email{chunfengcui@buaa.edu.cn}).}}

\headers{DualHash: Stochastic Primal-Dual Algorithms}{Luxuan Li, Xiao Wang, and Chunfeng Cui}

\makeatletter
\let\newtheorem\arx@newtheorem
\let\theoremstyle\arx@theoremstyle
\makeatother
\definecolor{blue}{rgb}{0,0,0}

\usepackage{caption}
\usepackage{float}
\allowdisplaybreaks[2]
\numberwithin{figure}{section}
\numberwithin{table}{section}
\numberwithin{algorithm}{section}

\begin{document}

\maketitle
\begin{abstract}
{Deep hashing encodes high-dimensional image features into compact 
binary codes, enabling efficient similarity search in large-scale
content-based image retrieval, a problem akin to semantic-preserving
lossy compression under extreme bit-rate constraints.} A fundamental challenge in deep hashing stems
from the discrete nature of quantization in generating the codes.
W-type regularizations, such as $\bz\mapsto\lVert|\bz|-\mathbf{1}\rVert_1$, are widely adopted and empirically effective at pushing network outputs toward binary values.
However, existing methods often directly optimize these regularizations without convergence guarantees.
While proximal gradient methods offer a promising solution, the coupling between W-type regularizers and neural network outputs results in composite forms that generally lack closed-form proximal solutions.
In this paper, using Fenchel duality, we reformulate the nonconvex W-type regularization term through its convex conjugate, yielding a min--max problem with tractable proximal updates.
We show that this min--max problem captures the stationarity of the original deep hashing objective up to the quantization error.
Building on this problem, we present a stochastic primal-dual hashing algorithm, referred to as DualHash, with oracle complexity guarantees.
We derive two algorithm instances: a momentum-based version with $\mathcal{O}(\varepsilon^{-4})$ oracle complexity and an improved $\mathcal{O}(\varepsilon^{-3})$ version using variance reduction, both for $\varepsilon$-KKT points of the min--max problem, which yield near-stationarity guarantees for the original objective.
Experiments on three standard datasets demonstrate competitive retrieval accuracy and reduced quantization error, validating the practical value of DualHash for image retrieval problems.
Our implementation of DualHash is available at \url{https://github.com/RailLee/DualHash}.
\end{abstract}
\begin{keywords}
nonconvex nonsmooth optimization, primal-dual method, deep hashing,
image retrieval, stochastic optimization, oracle complexity.
\end{keywords} 

\begin{MSCcodes}
90C26, 90C46, 65K05, 68T07, 68P20
\end{MSCcodes}

\section{Introduction}
\label{Section: Introduction}
Hashing is a pivotal technique in large-scale image retrieval,
which encodes image content into compact binary codes (e.g., $\pm 1$)
for efficient storage and accurate similarity search.
\label{loc:q8a}Mathematically, this amounts to \blue{carrying out} a projection onto the binary hypercube $\{-1,+1\}^d$ that preserves semantic
neighborhood structures\blue{. Such a process can be interpreted as a nonconvex discrete optimization problem akin to lossy image compression, where the data fidelity criterion is based on}
pairwise similarity rather than pixel-level reconstruction.\label{loc:q8b}
\label{loc:q9a}Deep hashing leverages neural networks to learn such projections and
\blue{is able to reduce storage by a factor of $512$} while maintaining 
competitive retrieval accuracy~\cite{luoSurveyDeepHashing2023}\blue{; for example,
a $1024$-dimensional \texttt{float32} feature vector shrinks to a $64$-bit
binary code}.\label{loc:q9b}
The mainstream approaches fall into two categories: sign-based
methods and continuous relaxation.
\label{loc:q27}A fundamental challenge, however, is that the $\operatorname{sgn}$
function has zero gradient almost everywhere, rendering standard
first-order methods ineffective~\cite{Cao2017HashNetDL}.
Continuous relaxation mitigates this but inevitably introduces
\emph{quantization error}\blue{\footnote{\blue{A precise definition is given in~\eqref{eq:quantization-error}.}}}---the discrepancy between optimized
continuous values and the final binary codes~\cite{weissSpectralHashing2008}.
\begin{figure}[!t]
\centering
\resizebox{0.72\columnwidth}{!}{
\begin{tikzpicture}[x=0.9cm,y=0.9cm,>=Stealth,line cap=round,line join=round]
    % (a)
    \begin{scope}[shift={(0,0)}]
        \draw[-{Stealth},thin] (-2.8,0) -- (2.8,0) node[below right] {\blue{$z$}};
        \draw[-{Stealth},thin] (0,-0.7) -- (0,2.9);
        \draw[line width=0.9pt] (-2.6,2.6) -- (0,0) -- (2.6,2.6);
        \node at (0,-1.15) {$\text{(a)}\ h(\bz)=\|\bz\|_1$};
    \end{scope}
    % (b)
    \begin{scope}[shift={(7.8,0)}]
        \draw[-{Stealth},thin] (-2.9,0) -- (2.9,0) node[below right] {\blue{$z$}};
        \draw[-{Stealth},thin] (0,-0.7) -- (0,2.9);
        \draw[line width=0.9pt] (-2.6,1.9) -- (-1,0) -- (0,2.1) -- (1,0) -- (2.6,1.9);
        \node at (-1,-0.42) {$-1$};
        \node at (1,-0.42) {$+1$};
        \node at (0,-1.15) {$\text{(b)}\ h(\bz)=\||\bz|-\blue{\mathbf{1}}\|_1$};
    \end{scope}
\end{tikzpicture}}
\caption{The \blue{one-dimensional} V-type (left) versus W-type (right) \blue{regularizers}.}
\label{Fig: the W and V regularizations}
\vspace{-3ex}
\end{figure}
To mitigate this issue, researchers have developed specialized regularization techniques~\cite{caoDeepCauchyHashing2018,li2017deep,liuDeepSupervisedHashing}. \label{loc:wtypea}Among them, W-type \blue{regularizers}, i.e., \blue{componentwise penalties that vanish exactly at $\pm1$, such as $\bz\mapsto\lVert|\bz|-\mathbf{1}\rVert_1$~\cite{liuDeepSupervisedHashing,wang2023deep} and $\bz\mapsto\lVert|\bz|-\mathbf{1}\rVert_2^2$~\cite{wang2016DTSH,he2025mambahash}}, have proven effective.\label{loc:wtypeb}
These regularizers impose sharp penalties to encourage variables towards binary values (e.g., $-1$ or $1$). Mathematically, this approach naturally formulates deep hashing as a nonconvex composite optimization problem\footnote{A formal definition of this problem is provided in \Cref{Section: Problem Reformulation}.}:
\begin{align}
\label{Eq: the general nonlinear composite problem in intro}
\min_{\bx \in \RR^{\blue{q}}} \quad \blue{J_0(\bx):=}\frac{1}{n} \sumn f_i (\bx) + \sumn h(\blue{\calD_i}(\bx)),
\tag{\blue{P0}}
\end{align}where \(f_i\) and \(\blue{\calD_i}\) are continuously differentiable and \(h\) is a proper and \blue{lower semicontinuous} (l.s.c.) regularizer. Classic examples include the
\(\ell_1\) norm to introduce sparsity and the \(\ell_2\) norm weight decay to prevent overfitting. In practice, researchers typically apply stochastic (sub)gradient methods (SGD) directly to optimize the objective \eqref{Eq: the general nonlinear composite problem in intro}, and deep hashing follows this paradigm as well. This seemingly effective choice, however, frequently leads to suboptimal performance. \label{loc:q11a}The regularizer is often nonsmooth\blue{; for such objectives, classical subgradient methods may exhibit} slow convergence~\blue{\cite{nesterov2005smooth}} \blue{or oscillatory behavior~\cite{bolte2024longterm}}.\label{loc:q11b}

The proximal operator is a powerful tool for handling nonsmooth
regularizers, with widespread applications in signal/image processing
and computer vision, including
\label{par:intro-imaging}inverse imaging problems~\blue{\cite{ding2020ldct_pfbs,chambolleFirstOrderPrimalDualAlgorithm2011,repetti2021variable,shumaylovweakly,li2023spherical}},\label{loc:q12b}
low-rank and robust matrix factorization~\cite{chen2025convergence},
and structured neural network training~\cite{bian2026stochastic}. For the special case $\blue{\calD_i}(\bx) = \bx$, stochastic proximal gradient methods (SPGD) and their variance-reduced or momentum-based variants have been extensively studied~\cite{xu2023momentum}\blue{, provided $h$ has a tractable proximal operator (e.g., $\ell_1$ soft-thresholding). For a linear mapping $\calD_i(\bx)=\bA_i\bx$, proximal-based Newton-type and ADMM methods are available even when the outer function is a discontinuous step function such as $\operatorname{sgn}$~\cite{zhou2021quadratic,wang2022svm}.} Although the W-type regularizer \blue{$h$ can admit closed-form proximal operators}~\cite{bai2019proxquant} (\Cref{Fig: the W and V regularizations}), \label{loc:q13c}\blue{its composition with the network output $\calD_i$, i.e., $h\circ\calD_i$,} renders the proximal subproblem intractable, making SPGD inapplicable to Problem~\eqref{Eq: the general nonlinear composite problem in intro}. Moreover, unlike standard deep learning tasks, where regularization is typically imposed on model parameters, deep hashing requires regularization directly on the network output to enforce binary codes. This makes the nonsmooth composite form \(\blue{h\circ\calD_i}\) intrinsic to deep hashing and therefore necessary to handle explicitly.

\leavevmode\label{loc:q14a}For nonsmooth compositional optimization, including Problem~\eqref{Eq: the general nonlinear composite problem in intro}, significant research has been conducted when \(h\circ \blue{\calD}\) is either convex~\cite{wang2024alexr} or weakly convex~\cite{davis2019,Drusvyatskiy2019,wang2024,hu23}. For instance, \(h\circ \calD\) is weakly convex when \(h\) is convex and \(\calD\) is smooth~\cite{davis2019,Drusvyatskiy2019,wang2024}, or when \(h\) is \blue{weakly convex and} component-wise nondecreasing\blue{,} and \(\calD\) is weakly convex~\cite{hu23}. However, when $h$ itself is nonconvex, as in our setting, $h \circ \calD$ is not necessarily weakly convex, rendering these methods inapplicable.\label{loc:q14b} For the case where $h$ is nonconvex, \cite{Bolte2018} \label{loc:q15a}reformulates the problem as \blue{one of} constrained optimization\label{loc:q15b} by imposing redundant variables and establishes asymptotic convergence via a Lagrangian-based approach, while Yao et al.~\cite{yao2024note} and Zhang et al.~\cite{zhang2025complexity} establish non-asymptotic iteration complexities for finding approximate stationary points. \label{par:intro-deterministic}However, all consider only the \emph{deterministic} setting, whereas ours addresses the \emph{stochastic} finite-sum formulation that naturally arises in large-scale deep hashing.\label{loc:q16b}

\subsection{Contributions}
\blue{To handle the nonsmooth composition $h\circ\calD_i$ in Problem~\eqref{Eq: the general nonlinear composite problem in intro} in the stochastic setting,} \label{loc:q17a}we reformulate deep hashing models with W-type regularization as a finite-sum two-block optimization problem \eqref{Prop: block-wise regularized problem}, decoupling the challenging composite term \blue{$h\circ\calD$}.\label{loc:q17b} Using Fenchel duality, we transform the nonconvex W-type regularization into its convex conjugate $h^*$\blue{, which leads to a min--max problem. This approach is widely used in imaging sciences~\cite{chambolleFirstOrderPrimalDualAlgorithm2011,esserZhangChan2010,clasonAccelerationGlobalConvergence2019}, where the min--max problem is equivalent to the original one when $h$ is convex. For nonconvex $h$, such as W-type regularizers, however, this equivalence no longer holds. This raises our first question:}
\begin{quote}
  \blue{\textbf{(RQ1)} \emph{To what extent does the min--max problem capture the stationarity of Problem~\eqref{Eq: the general nonlinear composite problem in intro} with nonconvex W-type regularizers?}}
\end{quote}

\blue{We answer \textbf{(RQ1)} by showing that the min--max problem captures the stationarity of Problem~\eqref{Eq: the general nonlinear composite problem in intro} up to the quantization error. Specifically, under unified conditions for W-type regularizers, we derive their conjugates and biconjugates in closed form, show that a KKT point is stationary for Problem~\eqref{Prop: block-wise regularized problem} if and only if $h$ and $h^{**}$ coincide at it, and bound the resulting structural error by the quantization error. This reduction motivates our second question:}
\begin{quote}
  \blue{\textbf{(RQ2)} \emph{Can a stochastic first-order method find an $\varepsilon$-KKT point of the min--max problem with non-asymptotic oracle complexity?}}
\end{quote}

\blue{We answer \textbf{(RQ2)} affirmatively.} We propose \textit{DualHash}, a stochastic primal-dual algorithm for deep hashing, whose dual update admits closed-form solutions for common W-type regularizers, with two instances: a momentum-based implementation (\textbf{DualHash-StoM}) and a variance-reduced implementation (\textbf{DualHash-StoRM}). We establish an $\mathcal{O}(\varepsilon^{-4})$ oracle complexity for DualHash-StoM and, under the mean-square smoothness condition, an improved $\mathcal{O}(\varepsilon^{-3})$ complexity for DualHash-StoRM, matching the best-known lower bound for stochastic smooth optimization~\cite{arjevani2023lower}. \blue{Combined with our answer to \textbf{(RQ1)}, these complexity bounds guarantee that the algorithm output is near-stationary for Problem~\eqref{Eq: the general nonlinear composite problem in intro} up to the quantization error.} Table~\ref{Table: summary of deep hashing algorithms} summarizes the comparison with existing deep hashing methods.

Experiments on three widely-used datasets demonstrate the practical effectiveness of DualHash. In particular, the proposed method achieves consistently lower quantization errors across different bit lengths compared to baselines, validating the advantages of our primal-dual framework for image retrieval in both single-label and multi-label settings.
\begin{table}[!t]
    \centering
    \caption{\blue{Comparison of representative deep hashing methods in terms of regularizer, optimization strategy, and oracle complexity.}}
    \label{Table: summary of deep hashing algorithms}
    \setlength{\tabcolsep}{3pt}
    \renewcommand{\arraystretch}{0.85}
    \resizebox{0.8\linewidth}{!}{
\begin{tabular}{ccccc}
\toprule
\textbf{Type} & \textbf{Method} & \textbf{\blue{W-type Regularizer}} & \textbf{Optimization} & \textbf{Complexity} \\
\midrule
    \multirow{3}{*}{\centering Discrete}
    & DDSH~\cite{liDeepSupervisedHashing} & \multirow{3}{*}{\centering \blue{$\ell_2^2$}} & BQP & \multirow{9}{*}{\centering \textemdash} \\
    & DSDH~\cite{li2017deep} & & $\operatorname{sgn}$ update & \\
    & DPSH~\cite{liFeatureLearningBased2016} & & DCC & \\
    \cmidrule(lr){1-4}
    \multirow{6}{*}{\centering Continuous}
    & DSH~\cite{liuDeepSupervisedHashing} & \blue{$\ell_1$} & \multirow{3}{*}{\centering SGDM} & \\
    & DTSH~\cite{wang2016DTSH} & \blue{$\ell_2^2$} & & \\
    & DHN~\cite{zhuDeepHashingNetwork2016} & \blue{log-cosh} & & \\
    \cmidrule(lr){2-4}
    & MDSHC~\cite{wang2023deep} & \blue{$\ell_1$} & RMSProp & \\
    \cmidrule(lr){2-4}
    & HashNet~\cite{Cao2017HashNetDL} & \multirow{2}{*}{\centering \blue{\textemdash}} & SGDM & \\
    & OrthoHash~\cite{hoeOneLossAll2021} & & Adam & \\
    \cmidrule(lr){2-5}
    & \textbf{DualHash-StoM} & \multirow{2}{*}{\blue{$\ell_1$}} & \multirow{2}{*}{Sto Primal-Dual} & $\mathcal{O}(\varepsilon^{-4})$ \\
    & \textbf{DualHash-StoRM} & & & $\mathcal{O}(\varepsilon^{-3})$ \\
    \bottomrule
    \end{tabular}
    }
    \parbox{\linewidth}{\scriptsize
    \textbf{Notes}:
    DCC stands for Discrete Cyclic Coordinate \blue{Descent}, while BQP denotes Binary Quadratic Programming. The \textbf{Complexity} column reports the \blue{oracle complexity}, with `\textemdash' indicating \blue{no corresponding complexity guarantee}.}
     \vspace{-2ex}
\end{table}

\subsection{Related Work}
\label{Section: related work}
\paragraph{Optimization in Deep Hashing} 
Besides W-type regularizers,
alternative strategies either introduce auxiliary hash codes with a 
smoother surrogate $\|\cdot\|_2^2$~\cite{liDeepSupervisedHashing,
li2017deep,liFeatureLearningBased2016}, or refine the discretization 
objective via central similarity targets~\cite{yuanCentralSimilarityQuantization2020}.
\label{loc:q19a}\blue{Recent architectures retain squared W-type penalties: MambaHash~\cite{he2025mambahash}
and CRH~\cite{yin2026crh} use $\bz\mapsto\lVert|\bz|-\mathbf{1}\rVert_2^2$
to reduce quantization error.}\label{loc:q19b}
\label{loc:q18a}Nevertheless, theoretically grounded methods for nonsmooth W-type regularization
composed with nonlinear network mappings \blue{remain limited}.\label{loc:q18b}

\paragraph{Primal-Dual and Stochastic Splitting Methods} \blue{Working with the convex conjugate $h^*$ to reformulate composite problems as min--max problems is a standard strategy in imaging sciences~\cite{chambolleFirstOrderPrimalDualAlgorithm2011,esserZhangChan2010,chouzenouxAdjointMismatch}, with extensions to nonlinear saddle-point models~\cite{gao2023alternative}.} Stochastic counterparts, notably SPDHG~\cite{chambolleSPDHG2018} \label{loc:q20a}\blue{and stochastic majorization-minimization subspace methods~\cite{chouzenouxStochasticMM2017}}, handle structured finite-sum objectives \blue{with convex $h$}, while nonconvex extensions establish convergence to critical points under the Kurdyka--{\L}ojasiewicz (KL) condition \blue{for convex or weakly convex $h$}~\cite{clasonAccelerationGlobalConvergence2019,shumaylovweakly}.

\paragraph{Stochastic Nonconvex Constrained Optimization}
From a constrained optimization perspective, Problem~\eqref{Eq: the general nonlinear composite problem in intro} \blue{can be reformulated via a variable lifting technique~\cite{wang2025augmentedrnn,jin2025multicomposite} as a nonlinearly constrained optimization problem of the form
\begin{equation}
  \label{Eq: the constrained reformulation}
  \min_{\bw}\; f(\bw) + h(\bw)
  \quad \text{s.t.}\quad
  \bc(\bw) = 0.
\end{equation}}Several recent works have established oracle complexity results for first-order stochastic augmented-Lagrangian and penalty methods for problems of this type\blue{~\cite{jin2022stochastic,shi2025momentum,lu2024variancereducedfirstordermethodsdeterministically}. Specifically,} Jin and Wang~\cite{jin2022stochastic} achieved $\mathcal{O}(\varepsilon^{-6})$ (or $\mathcal{O}(\varepsilon^{-5})$ from a feasible initial point) for locating an $\varepsilon$-KKT point, while incorporating recursive momentum~\cite{cutkoskyMomentumBasedVarianceReduction2019} improves the complexity to $\mathcal{O}(\varepsilon^{-4})$ and the nearly optimal $\mathcal{O}(\varepsilon^{-3})$~\cite{shi2025momentum,lu2024variancereducedfirstordermethodsdeterministically}. \blue{However, these guarantees rely on the convexity of $h$ and thus do not directly apply to Problem~\eqref{Eq: the general nonlinear composite problem in intro}, where $h$ is a nonconvex W-type regularizer.}

\smallskip
\noindent\textbf{Notations.}\quad Throughout this paper, we use $a$, $\ba$, and $\bA$ to denote scalars, vectors, and matrices, respectively.
\blue{All vectors are column vectors. For a matrix \(\bA\in\RR^{n\times d}\), we write \(\bA=[\ba_1^\top;\dots;\ba_{n}^\top]\), where \(\ba_i\in\RR^d\) and \(\ba_i^\top\) is the $i$-th row of \(\bA\).}
\blue{$\bm{1}$ denotes an all-ones vector or matrix of the appropriate size}, and $i,j$ serve as indices. For matrices \(\bA,\bA^{\prime}\in\RR^{n\times d}$, we use the Frobenius inner product \(\langle \bA,\bA^{\prime}\rangle\) and the induced Frobenius norm \(\|\bA\|:=\sqrt{\langle\bA,\bA\rangle}\). Unless specified otherwise, $\|\cdot\|$ denotes the Euclidean norm for vectors and the Frobenius norm for matrices. $|\cdot|$ denotes element-wise absolute value.
\blue{For a matrix \(\bA=[A_{ij}]\), $\|\bA\|_1:=\sum_{i,j}|A_{ij}|$ denotes the \blue{componentwise} \(\ell_1\) norm.} \label{loc:hnotation}\blue{With a slight abuse of notation, the functions $h$, $h^*$, and $h^{**}$ are used for both vectors and matrices.} For a closed set \(\calC\subseteq\RR^q$, the distance from \(\bx$ to \(\calC$ is $\text{dist}(\bx,\calC):=\min_{\by\in\calC}\|\bx-\by\|$.
\smallskip

\section{Problem Formulation and Preliminaries}
\label{Section: Problem Reformulation}
\subsection{{Preliminaries}}
\blue{An} extended real-valued function $r:\RR^{q}\to\blue{[-\infty,+\infty]}$ is called
\textit{proper} if \blue{$r(\bx)>-\infty$ for all $\bx\in\RR^{q}$ and} its domain
$\operatorname{dom}(r):=\{\bx\in\RR^{q}\mid r(\bx)<+\infty\}$ is nonempty; \blue{it is called
\textit{lower semicontinuous} if its epigraph
$\operatorname{epi}(r):=\{(\bx,t)\in\RR^{q+1}\mid r(\bx)\le t\}$ is closed.}
We write
$\by\xrightarrow{r}\bx$ to denote $\by\to\bx$ and $r(\by)\to r(\bx)$. The \textit{(limiting) subdifferential}~\cite[Definition~8.3]{Rockafellar1998VariationalA} of a function $r:\RR^{q}\to[-\infty,+\infty]$ at a point
$\bx$ where $r(\bx)$ is finite is given by
\begin{equation}\label{Eq: limiting subdifferential}
\begin{aligned}
\partial r(\bx):={}&\Bigl\{\bv\in\RR^{q}:\ \exists\,\bx^{k}\xrightarrow{r}\bx,\ \bv^{k}\to\bv\ \text{with}\\
&\quad \liminf_{\substack{\by\to\bx^{k}\\ \by\neq\bx^{k}}}
\frac{r(\by)-r(\bx^{k})-\langle\bv^{k},\by-\bx^{k}\rangle}{\|\by-\bx^{k}\|}
\ge0\quad\text{for all } k\Bigr\}.
\end{aligned}
\end{equation}When $r$ is \blue{continuously differentiable} or convex, the above subdifferential coincides with the
\blue{classical gradient} or convex subdifferential of $r$
\cite[Exercise~8.8(b) and Proposition~8.12]{Rockafellar1998VariationalA}.
Moreover, for any continuously differentiable $g$,
$\partial(r+g)(\bx)=\partial r(\bx)+\nabla g(\bx)$
\cite[Exercise~8.8(c)]{Rockafellar1998VariationalA}; and if
$r(\bx_{1},\bx_{2})=r_{1}(\bx_{1})+r_{2}(\bx_{2})$ with $r_{1},r_{2}$ proper and lower
semicontinuous, then $\blue{\partial r(\bx_{1},\bx_{2})=\partial r_{1}(\bx_{1})\times\partial r_{2}(\bx_{2})}$
\cite[Proposition~10.5]{Rockafellar1998VariationalA}.
\label{loc:conj}For a proper function $r:\RR^{l}\to(-\infty,+\infty]$, its conjugate is $r^*(\by):=\sup_{\bx}\{\langle\by,\bx\rangle-r(\bx)\}$ and \blue{its biconjugate is $r^{**}:=(r^*)^*$.} \blue{If $r$ has an affine minorant, then $r^{**}$ is the largest lower semicontinuous convex function below $r$~\cite[Theorem~11.1]{Rockafellar1998VariationalA}.}
For a proper, \blue{lower semicontinuous, and bounded-below} function $r$, the \textit{proximal mapping} of $r$, denoted by $\operatorname{prox}_{\tau r}(\by)$, is defined as
\begin{align}
    % \label{Def: the definition of proximal operator}
    \operatorname{prox}_{\tau r}(\by) := \argmin\limits_{\bx \in \operatorname{dom}(r)}\left\{ r(\bx) + \frac{1}{2\tau}\|\bx - \by\|^{2}\right\}, \notag
\end{align}
where $\tau$ is a positive scalar parameter. The proximal mapping admits the following
properties.
\begin{proposition}\label{prop:rockafellar}
\textnormal{\cite{Rockafellar1998VariationalA}}
    Let $r : \mathbb{R}^{q}\rightarrow \mathbb{R}\cup \{+\infty\}$ be proper, \blue{lower semicontinuous, and bounded below}. Then, the following {statements hold}.
    \begin{enumerate}[label=(\roman*)]
        \item The set $\operatorname{prox}_{\tau r}(\bx)$ is nonempty and compact
            for any $\bx \in \mathbb{R}^{q}$ and $\tau > 0$.

        \item If $r$ is convex, then $\operatorname{prox}_{\tau r}(\bx)$
            contains exactly one value for any $\bx \in \mathbb{R}^{q}$ and
            $\tau > 0$.
    \end{enumerate}
\end{proposition}

\label{loc:q25}\blue{A differentiable} function \blue{$g$}: $\RR^q \rightarrow \RR$ is said to be \blue{$L_g$}-smooth if $\| \blue{\nabla g}(\bx) - \blue{\nabla g}(\bx') \| \leq \blue{L_g} \| \bx - \bx'\|$ for all $\bx,\, \bx' \in \RR^q$.
A mapping $\mathcal{M} : \blue{\Omega} \to \RR^l$ is said to be $C$-Lipschitz continuous if $\|\mathcal{M}(\bx) - \mathcal{M}(\bx')\| \leq C\|\bx - \bx'\|$ for all $\bx, \bx' \in \blue{\Omega}$. \label{loc:q26}For an \blue{$L_g$}-smooth function, the descent lemma \blue{\cite{Bertsekas1997}} states that
\begin{align}
\label{Eq: the descent lemma}
    \blue{g}(\bx) \leq \blue{g}(\bx') + \langle \blue{\nabla g}(\bx'), \bx - \bx' \rangle + \frac{\blue{L_g}}{2} \| \bx - \bx'\|^2.
\end{align}

\subsection{W-type Regularizer: Definition and Properties}
\label{Section: W-regularizer}

\begingroup\color{blue}
To reduce the quantization error introduced by continuous relaxation, penalties with W-shaped graphs, vanishing exactly at
$\pm1$, are added to drive the continuous outputs toward $\{-1,1\}^d$. Such penalties
are standard in deep hashing~\cite{liuDeepSupervisedHashing,yuanCentralSimilarityQuantization2020,he2025mambahash,
yin2026crh}. Their formulas differ, but our analysis depends only on their common
shape, which we formalize next.
\begin{definition}[W-type regularizer]\label{def:wtype}
Let \(\phi:\mathbb{R}\to[0,+\infty)\). For \(\bz\in\mathbb{R}^{d}\) and \(\bZ\in\mathbb{R}^{n\times d}\), define \(h(\bz):=\sum_{j=1}^{d}\phi(z_j)\) and \(h(\bZ):=\sum_{i=1}^{n}\sum_{j=1}^{d}\phi(z_{ij})\). We call \(h\) a \emph{W-type regularizer} if \(\phi\) satisfies:
\begin{enumerate}[label=\textup{(W\arabic*)}]
  \item\label{W1} \(\phi\) is even and continuous, and \(\phi(t)=0\) if and only if \(t\in\{-1,1\}\);
  \item\label{W2} \(\phi\) is convex on \([1,+\infty)\);
  \item\label{W3} \(\phi\) decreases at a nonvanishing rate on \([0,1)\): for every \(\theta\in(0,1)\) there exists \(c_{\theta}>0\) such that \(\phi(s)-\phi(t)\ge c_{\theta}(t-s)\) for all \(0\le s<t\le\theta\).
\end{enumerate}
\end{definition}
Condition~\ref{W1} fixes the W shape and, together with \ref{W2}, makes $\phi$ nondecreasing on $[1,+\infty)$, which determines the convex envelope and the conjugate below; \ref{W3} excludes flat pieces in $(-1,1)$, so that $\phi$ has no stationary point there, which yields the stationarity characterization in \Cref{prop:transfer}. We list representative W-type regularizers from the deep hashing literature together with their conjugates in \Cref{tab:wtype}. Its first entry, $h(\bz)=\lambda\bigl\||\bz|-\mathbf{1}\bigr\|_1$, is called the \emph{$\ell_1$ W-regularizer} and is illustrated in \Cref{fig:w-conj-prox}.
\begin{table}[!t]
\centering
\color{blue}
\captionsetup{labelfont={bf,color=blue},labelsep=period,textfont={color=blue}}
\caption{Representative W-type regularizers $h(\bz)=\sum_{j=1}^{d}\phi(z_j)$.}
\label{tab:wtype}
\setlength{\tabcolsep}{4pt}
\renewcommand{\arraystretch}{1.4}
\resizebox{\columnwidth}{!}{
\begin{tabular}{@{}ccccc>{\centering\arraybackslash}m{2.2cm}@{}}
\toprule
\(\phi(t)\) & \(\partial\phi(0)\) & \(\phi^*(y)\) & \(\operatorname{dom}\phi^*\) & \(\operatorname{prox}_{\tau^{-1}\phi^*}(v)\) & Relation to \(Q_1\)\\
\midrule
{\renewcommand{\arraystretch}{1}\begin{tabular}[c]{@{}c@{}}\(\lambda\bigl||t|-1\bigr|\)\\\cite{liuDeepSupervisedHashing,wang2023deep,bai2019proxquant}\end{tabular}}
 & \(\{\pm\lambda\}\) & \(|y|\) & \([-\lambda,\lambda]\)
 & \(\operatorname{clip}(\operatorname{sgn}(v)(|v|-\tau^{-1})_+)\) & $h(\mathcal{D}(\bx))=\lambda Q_1(\bx)$\\
\addlinespace[3pt]
{\renewcommand{\arraystretch}{1}\begin{tabular}[c]{@{}c@{}}\(\lambda\bigl||t|-1\bigr|^2\)\\\cite{wang2016DTSH,he2025mambahash,semanticHashCenters2025,yin2026crh,bai2019proxquant}\end{tabular}}
 & \(\{\pm2\lambda\}\) & \(|y|+\frac{y^2}{4\lambda}\) & \(\mathbb{R}\)
 & \(\operatorname{sgn}(v)\frac{(|v|-\tau^{-1})_+}{1+\tau^{-1}/(2\lambda)}\)
 & {\scriptsize $h(\mathcal{D}(\bx))\le\lambda\max\{1,C_{\mathcal{D}}-1\}Q_1(\bx)$}\\
\midrule
{\renewcommand{\arraystretch}{1}\begin{tabular}[c]{@{}c@{}}\(\lambda\log\cosh(|t|-1)\)\\\cite{zhuDeepHashingNetwork2016,yuanCentralSimilarityQuantization2020}\end{tabular}}
 & \(\{\pm\lambda\tanh1\}\)
 & \(|y|+\lambda\psi(|y|/\lambda)\)
 & \([-\lambda,\lambda]\) & \textemdash & $h(\mathcal{D}(\bx))\le\lambda Q_1(\bx)$\\
\bottomrule
\end{tabular}}
\parbox{\linewidth}{\scriptsize
\textbf{Notes}: \(\lambda>0\) and \(\psi(u):=\frac12[(1+u)\log(1+u)+(1-u)\log(1-u)]\). \blue{$\operatorname{clip}$ denotes projection onto $[-\lambda,\lambda]$.} In the proximal-mapping column, `\textemdash' indicates that no closed-form expression is available.}
\end{table}
\begin{remark}
{The structural results in \Cref{Section: Methodology} hold for every W-type regularizer, whereas the algorithms, the complexity analysis, and the experiments use the $\ell_1$ W-regularizer, which has a \emph{compact dual domain}, i.e., $\operatorname{dom}\phi^*=[-\lambda,\lambda]$.}
\end{remark}

For general W-type regularizers, the convex envelope and the conjugate admit the following explicit description. The proof is provided in \Cref{app:wtype-conjugate-proof}.
\begin{lemma}\label{lem:wtype-conj}
Let \(\phi:\mathbb{R}\to[0,+\infty)\) satisfy \ref{W1}--\ref{W2}. Then it holds that
\begin{enumerate}[label=\textup{(\alph*)}]
\item \(\phi^{**}(t)=\phi(\max\{|t|,1\})\) for all \(t\in\mathbb{R}\); in particular, \(\phi^{**}(t)=\phi(t)\) if and only if \(|t|\ge1\);
\item \(\phi^*\) is proper, lower semicontinuous, convex, and even; \(\phi^*(y)\ge|y|\) for all \(y\); and \(\operatorname{dom}\phi^*\) has nonempty interior.
\end{enumerate}
Consequently, \(h^*(\bLmbd)=\sum_{i,j}\phi^*(\Lambda_{ij})\), \(h^{**}(\bB)=\sum_{i,j}\phi^{**}(b_{ij})\), and \(h(\bB)=h^{**}(\bB)\) if and only if \(|b_{ij}|\ge1\) for all \(i,j\).
\end{lemma}

\endgroup

\begin{figure}[!t]
\centering
\resizebox{0.98\columnwidth}{!}{
\begin{tikzpicture}[x=0.9cm,y=0.9cm,>=Stealth,line cap=round,line join=round]
    % (a) W-type function
    \begin{scope}[shift={(0,0)}]
        \draw[-{Stealth},thin] (-2.9,0) -- (2.9,0) node[below right] {$t$};
        \draw[-{Stealth},thin] (0,-0.65) -- (0,1.55);
        \draw[line width=0.9pt] (-2.6,1.6) -- (-1,0) -- (0,1) -- (1,0) -- (2.6,1.6);

        % break-point labels
        \node at (-1,-0.42) {$-1$};
        \node at ( 1,-0.42) {$1$};
        \node[left] at (0,1) {$\lambda$};

        \node at (0,-2) {$\text{(a)}\ \phi(t)=\lambda\bigl||t|-1\bigr|$};
    \end{scope}

    % (b) Conjugate function phi^*(y)
    \begin{scope}[shift={(8.0,0)}]
        \draw[-{Stealth},thin] (-2.9,0) -- (2.9,0) node[below right] {$y$};
        \draw[-{Stealth},thin] (0,-0.65) -- (0,1.55);
        \draw[line width=0.9pt] (-1,1) -- (0,0) -- (1,1);
        \draw[densely dotted,thin] (-1,0) -- (-1,1.5);
        \draw[densely dotted,thin] ( 1,0) -- ( 1,1.5);
        \fill (-1,1) circle (0.07);
        \fill ( 1,1) circle (0.07);

        % break-point labels
        \node at (-1,-0.42) {$-\lambda$};
        \node at ( 1,-0.42) {$\lambda$};

        \node[left] at (0,1) {$\lambda$};

        \node at (0,-2) {$\text{(b)}\ \phi^*(y)$};
    \end{scope}

    % (c) Proximal mapping of phi^*(v)
    \begin{scope}[shift={(16.0,0)}]
        \draw[-{Stealth},thin] (-2.9,0) -- (2.9,0) node[below right] {$v$};
        \draw[-{Stealth},thin] (0,-1.35) -- (0,1.35);

        \draw[line width=0.9pt] (-2.8,-1) -- (-2,-1) -- (-1,0) -- (1,0) -- (2,1) -- (2.8,1);

        \draw[densely dotted,thin] (-2,0)--(-2,-1);
        \draw[densely dotted,thin] (-1,0)--(-1,-1);
        \draw[densely dotted,thin] ( 1,0)--( 1, 1);
        \draw[densely dotted,thin] ( 2,0)--( 2, 1);
        \draw[densely dotted,thin] (-2,-1)--(0,-1);
        \draw[densely dotted,thin] (0,1)--(2,1);

        \node[above] at (-2,0.10) {$-(\lambda+\tau^{-1})$};
        \node at (-1,-0.42) {$-\tau^{-1}$};
        \node at ( 1,-0.42) {$\tau^{-1}$};
        \node at ( 2,-0.42) {$\lambda+\tau^{-1}$};
        \node[left] at (0, 1) {$\lambda$};
        \node[left] at (0,-1) {$-\lambda$};

        \node at (0,-2) {$\text{(c)}\ \operatorname{prox}_{\tau^{-1}\phi^*}(v)$};
    \end{scope}
\end{tikzpicture}
}
\caption{The $\ell_1$ W-regularizer, drawn with $\lambda=\tau^{-1}=1$: (a) the scalar W-type regularizer $\phi$, (b) its conjugate $\phi^*$, finite on $[-\lambda,\lambda]$, and (c) the proximal mapping of $\phi^*$.}
\label{fig:w-conj-prox}
\vspace{-2ex}
\end{figure}

\subsection{W-type Regularized Deep Hashing Model}
\label{Sec:W-type Regularized Deep Hashing Model}
With the W-type regularizer \(h\) \blue{defined in \Cref{def:wtype}}, the model is now trained to produce the \(i\)-th continuous network output \(\blue{\calD_i}(\bx)\), and binary codes are obtained post-training via \(\operatorname{sgn}(\blue{\calD_i}(\bx))\in\{-1,1\}^d.\) Consequently, we formulate the following core optimization problem for W-type regularized deep hashing methods: % [2026-10-25] removed ``which serves as a continuous surrogate of Problem (direct form)''
\begin{align}
\label{Eq: the primal deep hashing problem}
\min_{\bx \in \RR^{\blue{q}}} \quad \blue{J_0(\bx):=}\frac{1}{n}\sum_{i=1}^n f_i(\bx) + \sum_{i=1}^n h(\blue{\calD_i}(\bx)),
\tag{\blue{P0}}
\end{align}where
\blue{\(f_i(\bx):=\ell_{s,i}(\mathcal{D}(\bx))+\frac{\mu}{2}\|\bx\|^2\) with \(\mu>0\), and \(\ell_{s,i}:\RR^{n\times d}\to\RR\) is a similarity-preserving loss such as the pairwise cross-entropy loss~\cite{zhuDeepHashingNetwork2016}.} \blue{Following~\cite{bai2019proxquant}, we define the $\ell_1$ quantization error in deep hashing problems as follows:
\begin{equation}
\label{eq:quantization-error}
Q_1(\bx):=\bigl\|\mathcal{D}(\bx)-\operatorname{sgn}(\mathcal{D}(\bx))\bigr\|_1
=\sum_{i=1}^{n}\sum_{j=1}^{d}\bigl||\mathcal{D}_{ij}(\bx)|-1\bigr|,
\end{equation}{where the equality follows componentwise from $\min_{b\in\{-1,1\}}|z-b|=\bigl||z|-1\bigr|$. The last column of \Cref{tab:wtype} relates each regularizer to $Q_1$: the $\ell_1$ W-regularizer gives $h(\mathcal{D}(\bx))=\lambda Q_1(\bx)$, whereas the others give only upper bounds of the form $h(\mathcal{D}(\bx))\le c\,Q_1(\bx)$.}}

We then employ variable splitting by introducing auxiliary variables \(\bb_i\in\RR^d\) for \(i=1,\dots,n\), stacked as \(\bB=[\bb_1^\top;\ldots;\bb_n^\top]\in\RR^{n\times d}\), subject to \(\bb_i=\calD_i(\bx)\). This yields the equivalent nonlinearly constrained problem:
\begin{align}
    \label{Eq: the constrained splitting problem}
    \min_{\bx \in \RR^{q},\ \bB \in \RR^{n\times d}}\quad \frac{1}{n}\sum_{i=1}^{n} f_i(\bx)+\sum_{i=1}^{n}h(\bb_i)\quad \text{s.t.}\quad \bb_i=\blue{\calD_i}(\bx),\ \forall i=1,\dots,n.
\end{align}
We next relax the constraints \(\bb_i=\calD_i(\bx)\) in \eqref{Eq: the constrained splitting problem} via the quadratic penalty~\cite[Sec.~17.1]{nocedalWright2006}
\begin{align}
    P(\bx,\bB):= \frac{\gamma}{2n}\sum_{i=1}^{n}\|\calD_i(\bx)-\bb_i\|^2=\frac{m}{2}\|\mathcal{D}(\bx)-\bB\|^2, \notag
\end{align}
where \(\gamma>0\) is a penalty parameter and \blue{\(m:=\gamma/n\)}. Defining \(f(\bx):=\frac{1}{n}\sum_{i=1}^{n}f_i(\bx)\) and \(F(\bx,\bB):=f(\bx)+P(\bx,\bB)\), and noting that \(h(\bB)=\sum_{i=1}^{n}h(\bb_i)\), we obtain the two-block finite-sum optimization problem:
\begin{align}
    \label{Prop: block-wise regularized problem}
    \min_{\bx \in \RR^{q},\,\bB \in \RR^{n\times d}}\ J(\bx,\bB):=F(\bx,\bB)+h(\bB). \tag{P}
\end{align}
Here \(F\) admits the finite-sum form
\begin{align}
    F(\bx,\bB)= \frac{1}{n}\sum_{i=1}^{n}F_i(\bx,\bb_i),\qquad F_i(\bx,\bb_i):= f_i(\bx)+\frac{\gamma}{2}\|\calD_i(\bx)-\bb_i\|^2. \notag
\end{align}
This reformulation decomposes the original problem \eqref{Eq: the primal deep hashing problem} into simpler structural blocks. Building upon Problem~\eqref{Prop: block-wise regularized problem}, we develop a stochastic primal-dual algorithm that addresses the lack of tractable proximal mappings in the nonsmooth composite form and the absence of convergence guarantees in existing deep hashing methods.
\subsection{Basic Assumptions}
We impose the following basic assumptions on Problem \eqref{Eq: the primal deep hashing problem}.
\begin{assumption}
\label{Ass: basic regularity}
\leavevmode
\begin{enumerate}[label=\textup{(\roman*)},leftmargin=2em]
    \item The average loss is \blue{bounded below: $\underline{\ell}:=\inf_{\bx}\bar{\ell}_{s}(\mathcal{D}(\bx))>-\infty$} where \(\bar{\ell}_s:=n^{-1}\sum_{i=1}^{n}\ell_{s,i}\).
    \item \blue{The function $f$ is $L_f$-smooth on $\RR^q$}, and $\nabla\mathcal{D}$ is $L_{\nabla\mathcal{D}}$-Lipschitz on $\RR^q$, with $L_f,L_{\nabla\mathcal{D}}>0$. There exist constants \blue{$C_{\mathcal{D}},G_{\mathcal{D}}>0$ such that}
    \[
       \blue{ \|\mathcal{D}(\bx)\|\le C_{\mathcal{D}},\qquad
        \|\nabla\mathcal{D}(\bx)\|\le G_{\mathcal{D}},\qquad
        \forall\,\bx\in\RR^q.}
    \]
\end{enumerate}
\end{assumption}
\blue{For example, for the pairwise cross-entropy loss, $\ell_{s,i}\ge0$ for all $i$, so Assumption~\ref{Ass: basic regularity}(i) holds with $\underline{\ell}\ge0$.}
\begin{remark}\label{rem:coercivity}
\color{blue}
Assumption~\ref{Ass: basic regularity} and $\mu>0$ make $J$ coercive, and thus Problem~\eqref{Prop: block-wise regularized problem} has an optimal solution. Specifically, since $h\ge0$ and $\|\mathcal{D}(\bx)\|\le C_{\mathcal{D}}$, we have
\[
J(\bx,\bB)\ge\underline{\ell}-\frac{m}{2}C_{\mathcal{D}}^2+\frac{\mu}{2}\|\bx\|^2+\frac{m}{4}\|\bB\|^2.
\]
\end{remark}

\section{Primal--Dual Formulation and Optimality Guarantees}
\label{Section: Methodology}
\subsection{\textcolor{blue}{Min--Max Reformulation and KKT Conditions}}
\label{Sec: first-order optimality conditions}
\blue{Since $h$ in Problem~\eqref{Prop: block-wise regularized problem} is nonconvex whereas its conjugate $h^*$ is convex, we} rewrite Problem~\eqref{Prop: block-wise regularized problem} in the equivalent constrained form
\begin{align}
\min_{\bx \in \RR^{\blue{q}}, \bB \in \RR^{n\times d}, \bY \in \RR^{n\times d}} &\quad F(\bx, \bB) + h(\bY) \notag \\
\text{s.t. } &\quad \bY = \bB. \notag
\end{align}
\begingroup\color{blue}
Attaching the dual multiplier $\bLmbd\in\RR^{n\times d}$ to $\bY=\bB$ and minimizing the resulting Lagrangian $F(\bx,\bB)+h(\bY)+\langle\bB-\bY,\bLmbd\rangle$ over $\bY$ via the definition of $h^*$ yields
\[
\calL(\bx,\bB,\bLmbd):=F(\bx,\bB)+\langle\bB,\bLmbd\rangle-h^*(\bLmbd).
\]
Maximizing $\calL$ over $\bLmbd$ and using $h^{**}=(h^*)^*$ then leads to the nonconvex-linear min--max problem~\cite{xu2026rada} targeted by our algorithms:
\begin{equation}\label{Eq: minmax problem}
\min_{\bx\in\RR^q,\,\bB\in\RR^{n\times d}}\max_{\bLmbd\in\RR^{n\times d}}\ \calL(\bx,\bB,\bLmbd)
= \min_{\bx\in\RR^q,\,\bB\in\RR^{n\times d}}\ F(\bx,\bB)+h^{**}(\bB). \tag{NC-L}
\end{equation}

In general, a saddle point of $\calL$ may fail to exist, and even checking whether a point is a local minimizer of the nonconvex problem \eqref{Prop: block-wise regularized problem} is NP-hard for nonconvex quadratic programs~\cite{Murty1987NP}. We therefore seek points satisfying the first-order necessary optimality condition for the min--max problem~\eqref{Eq: minmax problem}. Following~\cite{chambolleFirstOrderPrimalDualAlgorithm2011,
clasonAccelerationGlobalConvergence2019,
botBohm2023alternating}, we call
$(\tilde{\bx},\tilde{\bB},\tilde{\bLmbd})$ a
\emph{KKT point} of $\calL$ if
$\bm{0}\in\mathcal{K}(\tilde{\bx},\tilde{\bB},\tilde{\bLmbd})$,
where
\begin{equation}\label{Eq: KKT mapping}
\mathcal{K}(\bx,\bB,\bLmbd):=
\begin{pmatrix}
\nabla_{\bx}F(\bx,\bB)\\
\nabla_{\bB}F(\bx,\bB)+\bLmbd\\
\bB-\partial h^*(\bLmbd)
\end{pmatrix}
=
\begin{pmatrix}
\nabla_{\bx}\calL(\bx,\bB,\bLmbd)\\
\nabla_{\bB}\calL(\bx,\bB,\bLmbd)\\
-\partial_{\bLmbd}(-\calL)(\bx,\bB,\bLmbd)
\end{pmatrix}.
\end{equation}
Equivalently,
\begin{equation}\label{Eq: KKT conditions}
\nabla_{\bx}F(\tilde{\bx},\tilde{\bB})=\bm{0},
\qquad
\nabla_{\bB}F(\tilde{\bx},\tilde{\bB})
+\tilde{\bLmbd}=\bm{0},
\qquad
\tilde{\bB}\in\partial h^*(\tilde{\bLmbd}).
\end{equation}
Notice that KKT points of $\calL$ correspond one-to-one to stationary points of the min--max problem~\eqref{Eq: minmax problem}, i.e.,
\begin{equation}\label{Eq: saddle-KKT-convexified}
\nabla_{\bx}F(\tilde{\bx},\tilde{\bB})=\bm{0},
\qquad
\nabla_{\bB}F(\tilde{\bx},\tilde{\bB})+\tilde{\bLmbd}=\bm{0},
\qquad
\tilde{\bLmbd}\in\partial h^{**}(\tilde{\bB}),
\end{equation}which motivates the characterization of when a KKT point of $\calL$ is also stationary for Problem~\eqref{Prop: block-wise regularized problem}. In the following, we denote by $\operatorname{crit}\calL:=\{(\bx, \bB, \bLmbd) \in \RR^{q} \times \RR^{n\times d} \times \RR^{n\times d}: \bm{0} \in \mathcal K(\bx, \bB, \bLmbd)\}$ the set of KKT points of $\calL$.
\begin{remark}\label{rem:crit-nonempty}
\blue{Under Assumption~\ref{Ass: basic regularity} and $\mu>0$, $\operatorname{crit}\calL$ is nonempty. Specifically, since $h^{**}\ge0$, $F+h^{**}$ is coercive by Remark~\ref{rem:coercivity} and thus has a minimizer $(\bar\bx,\bar\bB)$, for which $(\bar\bx,\bar\bB,-\nabla_{\bB}F(\bar\bx,\bar\bB))$ satisfies~\eqref{Eq: saddle-KKT-convexified}.}
\end{remark}
\endgroup

\subsection{Exact Stationarity Relations for Problem~\eqref{Prop: block-wise regularized problem}}
\label{Sec: original optimality relations}
\begingroup\color{blue}
In this subsection, we relate KKT points of $\calL$ to Problem~\eqref{Prop: block-wise regularized problem}, first through the duality gap and then through stationarity. To examine the duality gap associated with $\calL$, we consider the dual problem of Problem~\eqref{Prop: block-wise regularized problem}:
\begin{equation}\label{Prop: the dual problem}
\max_{\bLmbd\in\RR^{n\times d}}\ \min_{\bx\in\RR^{q},\,\bB\in\RR^{n\times d}}\ \calL(\bx,\bB,\bLmbd).\tag{D}
\end{equation}
Let $p_{\min}$, $l_{\min}$, and $d_{\max}$ denote the optimal values of Problem~\eqref{Prop: block-wise regularized problem}, Problem~\eqref{Eq: minmax problem}, and Problem~\eqref{Prop: the dual problem}, respectively. By weak duality and $h^{**}\le h$, we have $d_{\max}\le l_{\min}\le p_{\min}$.
Recall that strong duality holds for Problems~\eqref{Prop: block-wise regularized problem} and~\eqref{Prop: the dual problem} if their duality gap $p_{\min}-d_{\max}$ is zero and both optimal values are attained. Here, zero duality gap and strong duality are equivalent, since both problems admit optimal solutions by the coercivity of $J$ and the bound $\phi^*(y)\ge|y|$ in Lemma~\ref{lem:wtype-conj}(b).
We first prove that any primal--dual optimal solution $(\bx^*,\bB^*,\bLmbd^*)$ with zero duality gap is a KKT point of $\calL$. The proof is provided in \Cref{app:zero-gap-proof}.
\begin{proposition}\label{prop:zerogap}
Suppose that $(\bx^*,\bB^*)$ and $\bLmbd^*$ are optimal solutions of Problems~\eqref{Prop: block-wise regularized problem} and~\eqref{Prop: the dual problem}, respectively, with zero duality gap. Then $(\bx^*,\bB^*,\bLmbd^*)\in\operatorname{crit}\calL$\blue{, i.e., $\bm{0}\in\mathcal{K}(\bx^*,\bB^*,\bLmbd^*)$}.
\end{proposition}

While \Cref{prop:zerogap} requires a zero duality gap, this gap can be positive even when $F$ is convex, due to replacing $h$ with $h^{**}$. Specifically, the duality gap splits as
\[
\blue{p_{\min}-d_{\max}=(p_{\min}-l_{\min})+(l_{\min}-d_{\max}),}
\]
where the first term is the loss caused by replacing $h$ with $h^{**}$, and the second term is caused by the nonconvexity of $F$. \Cref{ex:gap} illustrates this with a convex $F$, for which $\blue{l_{\min}}=\blue{d_{\max}}$.
\begin{example}\label{ex:gap}
Consider the problem
\[
\blue{p_{\min}}=\min_{b\in\RR}\Bigl\{\tfrac12 b^2+\bigl||b|-1\bigr|\Bigr\}.
\]
Its optimal value is $\blue{p_{\min}}=\tfrac12$, attained at $b^*=\pm1$. The dual value is
\[
\blue{d_{\max}}=\sup_{|\Lambda|\le1}\Bigl\{-\tfrac12\Lambda^2-|\Lambda|\Bigr\}=0,
\]
attained at $\Lambda^*=0$.
Hence $\blue{p_{\min}}-\blue{d_{\max}}=\tfrac12>0$, and no zero-gap primal--dual pair exists.
\end{example}

Since a zero duality gap cannot be guaranteed in general, we next characterize when a KKT point of $\calL$ is stationary for Problem~\eqref{Prop: block-wise regularized problem}. The proof is provided in \Cref{app:transfer-proof}.
\begin{proposition}\label{prop:transfer}
Let \(h\) be a W-type regularizer, and let $(\tilde{\bx},\tilde{\bB},\tilde{\bLmbd})$ be a KKT point of $\calL$. The following statements are equivalent:
\begin{enumerate}[label=\textup{(\alph*)}]
\item $(\tilde{\bx},\tilde{\bB})$ is a stationary point of Problem~\eqref{Prop: block-wise regularized problem}, i.e.,
$\nabla_{\bx}F(\tilde{\bx},\tilde{\bB})=\bm{0}$ and
$\bm{0}\in\nabla_{\bB}F(\tilde{\bx},\tilde{\bB})+\partial h(\tilde{\bB})$;
\item $h(\tilde{\bB})+h^*(\tilde{\bLmbd})=\langle\tilde{\bLmbd},\tilde{\bB}\rangle$;
\item $|\tilde{b}_{ij}|\ge1$ for all $i,j$.
\end{enumerate}
\end{proposition}
\endgroup

\subsection{From KKT Residuals to Problem~\eqref{Eq: the primal deep hashing problem}}
\label{Sec: original objective guarantees}
\begingroup\color{blue}
Proposition~\ref{prop:transfer} concerns exact KKT points, whereas the algorithms return an approximate KKT point $(\bx,\bB,\bLmbd)$ of $\calL$, whose accuracy is typically measured by the KKT residual
\begin{align}
    \label{Eq: kkt residual}
r_{\mathcal{K}}:=\operatorname{dist}\bigl(\bm{0},\mathcal{K}(\bx,\bB,\bLmbd)\bigr).
\end{align} \blue{Transferring such a point to $J_0$, however, incurs two errors. The optimization error is controlled by $r_{\mathcal{K}}$ and can be driven to zero by the algorithm. The structural error stems from replacing $h$ with $h^{**}$ since the min--max problem~\eqref{Eq: minmax problem} does not penalize non-binary outputs in $[-1,1]^{n\times d}$. In this subsection, we quantify both errors and show that the structural error, which is intrinsic to the model, is controlled by the quantization error~\eqref{eq:quantization-error}, which gives this empirical quantity a theoretical interpretation.}
% Consequently, the stationarity measure of $J_0$ at a point $\widehat{\bx}$ near $\bx$ is bounded by an optimization error, controlled by $r_{\mathcal{K}}$, and a structural error, controlled by $h(\mathcal{D}(\bx))$; for the $\ell_1$ W-regularizer, the latter equals $\lambda Q_1(\bx)$ with $Q_1$ defined in~\eqref{eq:quantization-error}.
% \[
% \operatorname{dist}\bigl(\bm{0},\partial J_0(\widehat{\bx})\bigr)
% \le
% \underbrace{\calO\bigl(r_{\mathcal{K}}\bigr)}_{\text{optimization error}}
% +
% \underbrace{\calO\bigl(\sqrt{h(\mathcal{D}(\bx))}\bigr)}_{\text{structural error}}.
% \]

Since $J_0$ is nonsmooth and nonconvex, our goal for Problem~\eqref{Eq: the primal deep hashing problem} is to find a point satisfying the following nearby-point criterion, related to near-approximate stationarity~\cite[Definition~2.5]{tian2022stationarity}.
\begin{definition}[Near-stationarity]\label{def:near-stationarity}
Given $\delta,\varepsilon\ge0$, a point $\bx$ is called \emph{$(\delta,\varepsilon)$-near-stationary} for Problem~\eqref{Eq: the primal deep hashing problem} if there exists $\widehat{\bx}$ such that
\begin{equation}\label{eq:near-stationarity}
\|\widehat{\bx}-\bx\|\le\delta,
\qquad
\operatorname{dist}\bigl(\bm{0},\partial J_0(\widehat{\bx})\bigr)\le\varepsilon,
\end{equation}
where $\partial$ denotes the limiting subdifferential defined in~\eqref{Eq: limiting subdifferential}.
\end{definition}

The following theorem, proved in \Cref{app:original-certificate-proof}, bounds the near-stationarity of $J_0$ explicitly.
\begin{theorem}\label{prop:original-certificate}
Under Assumption~\ref{Ass: basic regularity}(ii), let $h$ be a W-type regularizer defined in \Cref{def:wtype}, let $\kappa:=(1+m^2)/(4m)$, and let $L>0$ satisfy $L\ge L_f+\|\bLmbd\|L_{\nabla\mathcal{D}}$. Then, for any $(\bx,\bB,\bLmbd)$ satisfying $r_{\mathcal{K}}<\infty$, there exists $\widehat{\bx}\in\RR^q$ such that
\begin{align*}
\|\widehat{\bx}-\bx\|
&\le\sqrt{\frac{2}{L}}
\Bigl(\sqrt{\kappa}\,r_{\mathcal{K}}
+\sqrt{h(\mathcal{D}(\bx))}\Bigr),\\
\operatorname{dist}\bigl(\bm{0},\partial J_0(\widehat{\bx})\bigr)
&\le
\underbrace{
\Bigl(\sqrt{1+G_{\mathcal{D}}^2}
+2\sqrt{2\kappa L}\Bigr)r_{\mathcal{K}}
}_{\text{optimization error}}
+\underbrace{
2\sqrt{2L\,h(\mathcal{D}(\bx))}
}_{\text{structural error}}.
\end{align*}
If $\operatorname{dom}\phi^*\subseteq[-\lambda,\lambda]$, then $\|\bLmbd\|\le\lambda\sqrt{nd}$, and one can take $L=L_0:=L_f+\lambda\sqrt{nd}\,L_{\nabla\mathcal{D}}$, independent of $\bLmbd$.
\end{theorem}

For the regularizers in \Cref{tab:wtype}, the structural error in \Cref{prop:original-certificate} satisfies
\[
2\sqrt{2L\,h(\mathcal{D}(\bx))}\ 
\begin{cases}
\le 2\sqrt{2Lc\,Q_1(\bx)}, & \text{for each row},\\
= 2\sqrt{2L\lambda\,Q_1(\bx)}, & \text{for the $\ell_1$ W-regularizer},
\end{cases}
\]
where $c$ is the constant in the last column of \Cref{tab:wtype}.

{The theorem shows that the quantization error controls the structural component of \blue{the gap between the min--max problem and the original objective}.} Specifically, the optimization error vanishes as $r_{\mathcal{K}}\to0$, whereas the structural error, determined by the continuous outputs $\mathcal{D}(\bx)$, need not vanish as $r_{\mathcal{K}}\to0$. Hence, driving both $r_{\mathcal{K}}$ and $Q_1(\bx)$ to zero yields near-stationary points of the original problem.

% \begin{remark}[Roles of $\gamma$ and $\lambda$]The penalty parameter $\gamma$ affects only the optimization error, and can therefore be fixed, without the limit $\gamma\to\infty$ required by penalty methods. A larger regularization parameter $\lambda$ may accelerate binarization during training, but it also increases the structural error through its dependence on $\lambda$ and $L_0$; designing W-type regularizers or output layers that drive $\mathcal{D}(\bx)$ toward $\pm1$ without relying on a large $\lambda$ is left for future work.
% \end{remark}
\endgroup
% \end{remark}

\section{DualHash: Stochastic Primal-Dual Hashing Algorithms}
\label{Sec: DualHash framework}
\subsection{Algorithmic Framework}
To solve the min--max problem~\eqref{Eq: minmax problem}, we now present the \textbf{DualHash} framework, which updates $\bx$, $\bB$, and $\bLmbd$ alternately. In the dual space, the convexity and separability of $h^*$ \blue{established in \Cref{lem:wtype-conj}} ensure that its proximal mapping is single-valued, nonexpansive, and decomposes componentwise as
\begin{equation}
\label{Eq: the proximal mapping of conjugate function}
[\operatorname{prox}_{\tau^{-1} h^*}(\bLmbd)]_{ij}
=\operatorname{prox}_{\tau^{-1}\phi^*}(\Lambda_{ij}),
\end{equation}
\blue{and \Cref{tab:wtype} records closed-form expressions when they are available.} This motivates us to perform a proximal step to update $\bLmbd$.

First, we update $\bx$ by a stochastic first-order step, as is standard in training neural networks. At iteration $k$, \blue{the $\bx$-update takes the form
\begin{equation}
    \label{Eq: the generic update of x}
    \bx^{k+1} = \by^k - \eta_k \bg^k,
\end{equation}
where $\by^k$ is the base point of the step and $\bg^k$ is a stochastic estimator of $\nabla_{\bx}F$ at a query point.}
The framework accommodates different stochastic estimators $\bg^k$, including momentum-based and \blue{variance-reduced} ones, which may improve both the complexity guarantees and the practical performance; \blue{the choices of $\by^k$ and $\bg^k$ for each implementation are specified in the following two subsections.} Second, $\bB$ is updated by a gradient descent step on $\calL$, \blue{using the closed-form expression of $\nabla_{\bB}F$, i.e., $\nabla_{\bB}F(\bx,\bB)=m(\bB-\mathcal{D}(\bx))$}. Finally, following~\cite{chambolleFirstOrderPrimalDualAlgorithm2011}, the dual variable $\bLmbd$ is updated by a proximal step on $h^*$ at the extrapolated point $2\bB^{k+1}-\bB^k$:
\begin{align}
\bB^{k+1} &= \bB^k - \tau \nabla_{\bB}\calL(\bx^{k+1},\bB^k,\bLmbd^k) \notag\\
&= \bB^k - \tau\big(\nabla_{\bB}F(\bx^{k+1},\bB^k)+\bLmbd^k\big), \label{Eq: the update of B}\\
\bLmbd^{k+1} &= \operatorname*{arg\,min}_{\bLmbd}\Big\{h^*(\bLmbd)-\langle\bLmbd,2\bB^{k+1}-\bB^k\rangle+\frac{\tau}{2}\|\bLmbd-\bLmbd^k\|^2\Big\} \notag\\
&= \operatorname{prox}_{\blue{\tau^{-1}}h^*}\big(\bLmbd^k+\blue{\tau^{-1}}(2\bB^{k+1}-\bB^k)\big). \label{Eq: the update of lambda}
\end{align}

Since computing the full gradient $\nabla_{\bx} F(\bx,\bB)$ is expensive when the number of training samples $n$ is large, as is typical in deep learning, both instances replace it with a mini-batch estimator. We next present two instances of the stochastic update~\eqref{Eq: the generic update of x}, which differ in the momentum technique and in the point at which the stochastic gradient is queried. Throughout, $i\in\{1,\ldots,n\}$ indexes training samples and $k$ indexes iterations. At iteration $k$, we draw $\calJ_k$ uniformly from all subsets of $\{1,\ldots,n\}$ of size $b_k$, \blue{independently of the past. For any point $(\bx,\bB)$ determined before $\calJ_k$ is drawn, define}
\begin{equation}
\label{Eq: minibatch oracle}
\blue{\tilde{\nabla}_{\bx}F(\bx,\bB;\calJ_k):=\frac{1}{b_k}\sum_{i\in\calJ_k}\nabla_{\bx}F_i(\bx,\bb_i).}
\end{equation}
\subsection{DualHash with Momentum}
\label{Sec: DualHash StoM}
\begin{algorithm}[t]
\color{black}
\caption{DualHash-StoM}
\label{alg:our_algorithm_sgdm}
\begin{algorithmic}[1]
\STATE \textbf{Initialize:} \blue{$\bx^0=\bx^1\in\RR^q$, $\bB^1=\operatorname{prox}_{m^{-1}h^{**}}(\mathcal{D}(\bx^1))$, $\bLmbd^1=m(\mathcal{D}(\bx^1)-\bB^1)$, $\bB^0=\bB^1$, and $\bLmbd^0=\bLmbd^1$}; parameters $\{\alpha_k\},\{\beta_k\}\subset[0,1)$, $\{\eta_k\}$, $\tau$, and $T$.
\FOR{\blue{$k = 1, \ldots, T$}}
  \STATE \blue{Compute $\by^k$ and $\bz^k$ by \eqref{Eq: extrapolated points in SGDM}}
  \STATE \blue{Draw $\calJ_k$ and compute $\bg_m^k$ by \eqref{Eq: the stochastic approximate in SGDM}}
  \STATE Update $\bx^{k+1}$ by
  \[
  \bx^{k+1} = \by^k - \eta_k \blue{\bg_m^k}
  \]
  \STATE Update $\bB^{k+1}$ by
  \[
  \bB^{k+1} = \bB^k - \tau \nabla_{\bB}\calL(\bx^{k+1}, \bB^k, \bLmbd^k)
  \]
  \STATE Update $\bLmbd^{k+1}$ by
  \[
  \bLmbd^{k+1} = \operatorname*{arg\,min}_{\bLmbd} \Big\{ h^*(\bLmbd) - \langle \bLmbd, 2\bB^{k+1} - \bB^k \rangle + \frac{\tau}{2} \| \bLmbd - \bLmbd^k\|^2 \Big\}
  \]
\ENDFOR
\STATE \textbf{\blue{Output:}} \blue{$(\bx^R,\bB^R,\bLmbd^R)$}, where $R$ is \blue{drawn uniformly from $\{1,\ldots,T\}$}.
\end{algorithmic}
\end{algorithm}
We first present a momentum-based instance. At iteration $k$, we form two extrapolated points
\begin{equation}
\label{Eq: extrapolated points in SGDM}
\by^k=\bx^k+\alpha_k(\bx^k-\bx^{k-1}),\qquad
\bz^k=\bx^k+\beta_k(\bx^k-\bx^{k-1}),
\end{equation}
where $\by^k$ is the base point of the step and $\bz^k$ is the query point of the oracle. \blue{Since the initialization of \Cref{alg:our_algorithm_sgdm} sets $\bx^0=\bx^1$, \eqref{Eq: extrapolated points in SGDM} gives $\by^1=\bz^1=\bx^1$, so the first iteration is well defined.} Since $(\bz^k,\bB^k)$ is determined before $\calJ_k$ is drawn, \blue{calling \eqref{Eq: minibatch oracle} at this point gives an unbiased estimator of $\nabla_{\bx}F(\bz^k,\bB^k)$:}
\begin{equation}
\label{Eq: the stochastic approximate in SGDM}
\blue{\bg_m^k} := \tilde{\nabla}_{\bx}F(\bz^k,\bB^k;\calJ_k).
\end{equation}
\blue{Taking $\bg^k=\bg_m^k$ in \eqref{Eq: the generic update of x} yields the $\bx$-update}
\begin{equation}
\label{Eq: the update of x in SGDM}
\bx^{k+1} = \by^k - \eta_k \blue{\bg_m^k}.
\end{equation}
The overall algorithm, DualHash-StoM, is summarized in \Cref{alg:our_algorithm_sgdm}.

\begin{remark}
{The update \eqref{Eq: the update of x in SGDM} contains Nesterov's scheme ($\alpha_k=\beta_k$), the heavy-ball method ($\beta_k=0$), and SGD ($\alpha_k=\beta_k=0$) as special cases. The analysis in \Cref{Section: Convergence analysis} requires constant $0<\alpha<(1-\nu)/2$ for some $\nu\in(0,1)$ and $0<\beta<1$, and therefore covers only part of this family.} 
\end{remark}
\subsection{DualHash with Recursive Momentum}
\label{Sec: DualHash StoRM}
\begin{algorithm}[t]
\color{black}
\caption{DualHash-StoRM}
\label{alg:our_algorithm_storm}
\begin{algorithmic}[1]
\STATE \textbf{Initialize:} \blue{$\bx^0=\bx^1\in\RR^q$, $\bB^1=\operatorname{prox}_{m^{-1}h^{**}}(\mathcal{D}(\bx^1))$, $\bLmbd^1=m(\mathcal{D}(\bx^1)-\bB^1)$, $\bB^0=\bB^1$, and $\bLmbd^0=\bLmbd^1$}; parameters $\{\rho_k\}\subset(0,1)$, $\{\eta_k\}$, $\tau$, and $T$.
\FOR{$k = 1, \ldots, T$}
  \STATE Draw $\calJ_k$ and compute \blue{$\bg_s^k$} by \eqref{Eq: the stochastic approximate in STORM}
  \STATE Update $\bx^{k+1}$ by
  \[
  \bx^{k+1} = \bx^k - \eta_k \blue{\bg_s^k}
  \]
  \STATE Update $\bB^{k+1}$ by
  \[
  \bB^{k+1} = \bB^k - \tau \nabla_{\bB}\calL(\bx^{k+1}, \bB^k, \bLmbd^k)
  \]
  \STATE Update $\bLmbd^{k+1}$ by
  \[
  \bLmbd^{k+1} = \operatorname*{arg\,min}_{\bLmbd} \Big\{ h^*(\bLmbd) - \langle \bLmbd, 2\bB^{k+1} - \bB^k \rangle + \frac{\tau}{2} \| \bLmbd - \bLmbd^k\|^2 \Big\}
  \]
\ENDFOR
\STATE \textbf{\blue{Output:}} \blue{$(\bx^R,\bB^R,\bLmbd^R)$}, where $R$ is \blue{drawn uniformly from $\{1,\ldots,T\}$}.
\end{algorithmic}
\end{algorithm}

Alternatively, we consider the StoRM estimator~\cite{cutkoskyMomentumBasedVarianceReduction2019} for the $\bx$-update, which achieves variance reduction without requiring full-gradient or large-batch gradient computations \blue{after initialization}. We construct a mini-batch recursive momentum estimator \blue{$\bg_s^k$ of $\nabla_{\bx} F(\bx^k, \bB^k)$, calling the oracle \eqref{Eq: minibatch oracle} with the same mini-batch $\calJ_k$ at two query points:}
\begin{equation}
\label{Eq: the stochastic approximate in STORM}
\blue{\bg_s^k} :=
\begin{cases}
\tilde{\nabla}_{\bx} F(\bx^1, \bB^1; \calJ_1), &\text{if }  k = 1, \\
(1 - \rho_{k-1}) \big( \blue{\bg_s^{k-1}} - \tilde{\nabla}_{\bx} F(\bx^{k-1}, \bB^{k-1};\blue{\calJ_k})\big)+\tilde{\nabla}_{\bx} F(\bx^k, \bB^k; \calJ_k), & \text{if } k > 1.
\end{cases}
\end{equation}
\blue{Taking $\by^k=\bx^k$ and $\bg^k=\bg_s^k$ in \eqref{Eq: the generic update of x} yields the $\bx$-update}
\begin{equation}
\label{Eq: the update of x in STORM}
    \bx^{k+1} = \bx^k - \eta_k \blue{\bg_s^k}.
\end{equation}
The overall algorithm, DualHash-StoRM, is summarized in \Cref{alg:our_algorithm_storm}.

\section{Complexity Analysis}
\label{Section: Convergence analysis}
We now analyze the complexity of the proposed algorithms in terms
of the number of stochastic first-order oracle (SFO) calls
required to find an $\varepsilon$-KKT point of $\calL$.
Given $\varepsilon>0$, a random point $(\bx,\bB,\bLmbd)$
\blue{is called an \emph{$\varepsilon$-KKT point} if}
\begin{equation}
\label{equ:esp}
\mathbb{E}[\operatorname{dist}^{2}(\bm{0},\blue{\mathcal{K}}(\bx,\bB,\bLmbd))]
\leq\varepsilon^2.
\end{equation}
\blue{Let $T$ denote the maximum number of iterations. Under the mini-batch sampling scheme
introduced in Section~\ref{Sec: DualHash framework}, let $(\Omega,\mathcal F,\mathbb P)$
be the underlying probability space. For $k\geq1$, define
$\mathcal F_k:=\sigma(\mathcal J_1,\ldots,\mathcal J_{k-1})$ with $\mathcal F_1$ trivial, as
the filtration generated by the past mini-batches. We write
$\mathbb E_k[\cdot]:=\mathbb E[\cdot\mid\mathcal F_k]$ for the conditional expectation
given the past and $\mathbb E[\cdot]$ for the total expectation. All equalities and
inequalities involving $\mathbb E_k[\cdot]$ are understood to hold almost surely.} 

For the complexity analysis, we retain Assumption~\ref{Ass: basic regularity} from
Section~\ref{Section: Problem Reformulation}. \blue{Throughout this section, assume that $h$ is the $\ell_1$ W-regularizer or, more generally, a W-type regularizer with $\operatorname{dom}\phi^*\subseteq[-\lambda,\lambda]$.} \blue{Let $L_F:=L_f+\lambda\sqrt{nd}\,L_{\nabla\mathcal{D}}+m(2L_{\nabla\mathcal{D}}C_{\mathcal{D}}+G_{\mathcal{D}}^2+1)$. By \Cref{app:smoothness-discussion}, under the prescribed initialization and step sizes, $F$ is $L_F$-smooth along the iterates and $\|\bB^k\|\le R_B:=C_{\mathcal{D}}+\lambda\sqrt{nd}/m$ for all $k\ge1$.} We then impose the following stochastic-oracle assumption on Problem \eqref{Prop: block-wise regularized problem}.
\begin{assumption}
\label{Ass: main assumptions}
\leavevmode
\blue{There exists $\sigma^2\geq0$ such that, for all} $\bx\in\RR^{\blue{q}}$ and $\bB\in\RR^{n\times d}$ \blue{with $\|\bB\|\le R_B$},
    \begin{align}
        \blue{\frac{1}{n}\sum_{i=1}^{n}
        \left\|\nabla_{\bx} F_i(\bx, \bb_i)
        - \nabla_{\bx} F(\bx, \bB)\right\|^2
        \leq \sigma^2}.\notag
    \end{align}
\end{assumption}

Our analysis aims to bound $\mathbb{E}[\blue{\operatorname{dist}^{2}(\bm{0},\mathcal{K}(\bx,\bB,\bLmbd))}]$. 
\blue{The presence of dual variables $\bLmbd$ makes it difficult to establish a descent property for the Lagrangian $\calL$.}\label{loc:q40} To this end, we construct suitable Lyapunov functions defined in \eqref{Eq: the Lyapunov function in SGDM} and \eqref{Eq: the Lyapunov function in STORM} to obtain the approximate descent property. We first analyze the dual dynamics\blue{: \Cref{Lemma: dual boundedness} bounds} $\|\bLmbd^{k+1}-\bLmbd^k\|^2$ in terms of the primal residuals, \blue{which enables} the Lyapunov construction.

Next, we analyze how the Lagrangian function $\calL$ changes
during the updates of primal and dual variables in one iteration. The total change can be decomposed as:
\begin{align}
    \label{Eq: One-iteration progress decomposition} & \calL(\bx^{k+1}, \bB^{k+1}, \bLmbd^{k+1}) - \calL(\bx^{k}, \bB^{k}, \bLmbd^{k}) \notag                                                                                                                                                                                       \\
    =                                                & \,\underbrace{\calL(\bx^{k+1}, \bB^{k+1}, \bLmbd^{k+1}) - \calL(\bx^{k+1}, \bB^k, \bLmbd^k)}_{\text{one-iteration progress on $\bB$ and $\bLmbd$}}+ \underbrace{\calL(\bx^{k+1}, \bB^{k}, \bLmbd^{k}) - \calL(\bx^{k}, \bB^k, \bLmbd^k)}_{\text{one-iteration progress on $\bx$}}.
\end{align}
Since the updates of $\bB$ and $\bLmbd$ are common to Algorithms~\ref{alg:our_algorithm_sgdm} and~\ref{alg:our_algorithm_storm}, \blue{\Cref{Lemma: One-iteration progress of B and Lambda}} bounds their one-step change. For notational simplicity, we define
\[
    \Delta_{\bx}^{k}= \frac{\|\bx^{k}- \bx^{k-1}\|^{2}}{2},\quad \Delta_{\bm{B}}^{k}
    = \frac{\|\bm{B}^{k}- \bm{B}^{k-1}\|^{2}}{2}.
\]

We now analyze the stochastic $\bx$-updates, which \blue{introduce} algorithm-specific challenges:
as \eqref{Eq: dual boundedness} and \blue{\eqref{item: B and Lambda}} show, randomness in the $\bx$-updates propagates to the $\bB$- and $\bLmbd$-updates, necessitating careful parameter tuning and the construction of Lyapunov functions. In the following, we present the key lemmas and complexity results of both algorithms;\label{loc:roadmap-ref} \blue{the analysis roadmap is summarized in \Cref{fig:complexity-roadmap}, which traces how the common lemmas in \Cref{app:complexity-auxiliary-results} and the algorithm-specific lemmas in \Cref{app:stom-complexity,app:storm-complexity} lead to \Cref{Theorem: the oracle complexity in SGDM,Theorem: the oracle complexity in STORM}.}
\subsection{Oracle Complexity of DualHash-StoM}
\label{Section: complexity of StoM}
\blue{We now present the oracle complexity of DualHash-StoM for finding an $\varepsilon$-KKT point, with the detailed analysis deferred to \Cref{app:stom-complexity}.} \blue{We first set the parameters of \Cref{alg:our_algorithm_sgdm} as}\label{loc:stom-param-a}
\begin{equation}
    \label{Eq: the parameters conditions of dualhash-sgdm in suppl}
    \eta_{k}= \frac{\eta}{L_{F}},\quad \alpha_{k}= \alpha,\quad \beta_{k}= \beta,\quad \blue{0<\tau<\frac{\sqrt{30}-2}{13L_F}},\quad \blue{\eta=\frac{(1 - 2\alpha - \nu)/2}{\frac{\bar{L}_{F}}{L_{F}} + \frac{(1 + \nu)\beta^{2}}{(1 - \alpha)\alpha}}},
\end{equation}
where \blue{$\bar{L}_F:=2L_F+K_3+K_4$ with $K_1,\ldots,K_4$ given in \Cref{Lemma: One-iteration progress of B and Lambda}, and} \blue{$\nu\in(0,1)$, $\alpha\in(0,\frac{1-\nu}{2})$, $\beta\in(0,1)$ are constants independent of $T$.}\label{loc:stom-param-b} We \blue{then} introduce the following Lyapunov function and sequence to establish an in-expectation approximate sufficient descent property for \Cref{alg:our_algorithm_sgdm}:\label{loc:psim}
\begin{align}
    \Psi_m(\mathcal{Q}^{k}) &= \mathcal{L}(\bx^{k}, \bm{B}^{k}, \bLmbd^{k})
        - C_{1}\Delta_{\bB}^{k+1} + C_{2}\Delta_{\bB}^{k}
        - C_{3}\Delta_{\bm{x}}^{k+1} + C_{4}\Delta_{\bm{x}}^{k},
        \label{Eq: the Lyapunov function in SGDM} \\
    \mathcal{Q}^{k} &:= \{\bx^{k}, \bB^{k}, \bm{\Lambda}^{k},
        \bB^{k+1}, \bB^{k-1}, \bx^{k+1}, \bx^{k-1}\}, \notag
\end{align}where the positive coefficients $C_1, C_2, C_3, C_4$ (independent of $k$) are chosen as:
\begin{align}
    \label{Eq: the coefficients of the Lyapunov function in SGDM}
    \begin{cases}
        C_{1} &= K_{1},\\
        C_{2} &= \frac{1}{2}(K_{2} - K_{1} + K_{3}),\\
        C_{3} &= K_{4},\\
        C_{4} &= \frac{\bar{L}_F + \frac{2L_F \beta^2}{\alpha}}{1 - 2\alpha - \nu}.
    \end{cases}
\end{align}\blue{By \Cref{Lemma: lower bound of Psi_m}, $\Psi_m$ is bounded below by some $\Psi_m^*>-\infty$ independent of $T$; for a deterministic initial point, define $\Delta_1:=\calL(\bx^1,\bB^1,\bLmbd^1)-\Psi_m^*$.}
The construction yields the following in-expectation approximate sufficient
descent property of \blue{$\{\Psi_m(\mathcal{Q}^{k})\}$}.
\begin{lemma}[Approximate sufficient descent of $\Psi_m$]
    \label{Lemma: approximate sufficient descent property in SGDM} Under Assumptions~\ref{Ass: basic regularity} and~\ref{Ass: main assumptions} and the parameter conditions \eqref{Eq: the parameters conditions of dualhash-sgdm in suppl}, {for all $k \geq 1$},
    \begin{align}
        \label{Eq: the descent property of Lyapunov func in SGDM}\EE [\Psi_m(\mathcal{Q}^{k+1})] - \EE[\Psi_m(\mathcal{Q}^{k})] \leq -C_{\bB}\EE[\Delta_{\bB}^{k+1}+ \Delta_{\bB}^{k}] -C_{\bx}\EE [\Delta_{\bm{x}}^{k+1}+ \Delta_{\bm{x}}^{k}] + \frac{\sigma^{2}}{b_{k}L_{F}},
    \end{align}
    where the descent coefficients are given by
    \begin{align}
        \label{Eq: the descent coefficients of potential function in SGDM}\begin{cases}C_{\bB}&= C_{2}- K_{3}= \frac{K_2 - K_3 -K_1}{2}>0, \\ C_{\bx}&= \nu C_{4}>0.\end{cases}
    \end{align}
\end{lemma}
\blue{The proof is provided in \Cref{app:stom-descent-proof}}, and the positivity of the coefficients in
\eqref{Eq: the coefficients of the Lyapunov function in SGDM} and the descent
constants in \eqref{Eq: the descent coefficients of potential function in SGDM}
\blue{are} established in \Cref{lem:lyapunov-coef-positivity}.

We now present the convergence rate of \Cref{alg:our_algorithm_sgdm} under the parameter conditions \eqref{Eq: the parameters conditions of dualhash-sgdm in suppl}, which guides the selection of the batch size $b_k$.
\begin{lemma}[Convergence rate of DualHash-StoM]
\label{lemma: convergence rate sgdm} Under the conditions of \Cref{Lemma:
    approximate sufficient descent property in SGDM}, for any $T \geq 1$,
    \begin{align}
        \label{Eq: the convergence rate of sgdm}\blue{\frac{1}{T}\sum_{k=1}^{T}\mathbb{E}[\mathrm{dist}^{2}(\bm{0}, \mathcal K(\bx^{k}, \bB^{k}, \bLmbd^{k}))]} \leq \frac{L_{\Delta}\Delta_{1}}{T}+ \blue{\frac{C_{\sigma}\sigma^{2}}{T}\sum_{k=1}^{T}\frac{1}{b_{k}}},
    \end{align}\blue{where $L_{\Delta}:=\frac{\max\{S_1,S_2,S_3,S_4\}}{\min\{C_{\bB},C_{\bx}\}}$ and $C_{\sigma}:=\frac{L_{\Delta}}{L_F}+3$ are constants independent of $T$, with $S_1,\ldots,S_4$ defined in \eqref{Eq: the coefficients of stationarity error bound in SGDM}.}
\end{lemma}
\blue{To achieve the desired oracle complexity, we choose the batch size according to \eqref{Eq: the convergence rate of sgdm} as}\label{loc:stom-batch-a}
\begin{align}
\label{Eq: the parameters conditions of dualhash-sgdm}
    \blue{b_k=b=\lceil c_bT\rceil\in\{1,\ldots,n\},}
\end{align}
\blue{where $c_b>0$ is a constant independent of $T$, and $1\le\lceil c_bT\rceil\le n$.}\label{loc:stom-batch-b} The complexity result is established in \Cref{Theorem: the oracle complexity in SGDM}, with the proof in \Cref{app:stom-oracle-proof}.
\begin{theorem}[Oracle complexity of DualHash-StoM]
    \label{Theorem: the oracle complexity in SGDM}
Under \blue{Assumptions~\ref{Ass: basic regularity} and}~\ref{Ass: main assumptions}, let \(\{\bx^k, \bB^k, \bLmbd^k\}_{k \in \NN}\) be the iterate sequence from \Cref{alg:our_algorithm_sgdm} with the parameters satisfying \blue{\eqref{Eq: the parameters conditions of dualhash-sgdm in suppl} and} \eqref{Eq: the parameters conditions of dualhash-sgdm}. Then, \blue{for $R$ drawn uniformly from $\{1,\dots,T\}$ independently of the mini-batches,} \( (\bx^R, \bB^R, \bLmbd^R )\) satisfies \eqref{equ:esp} \blue{with}
\begin{align}
    \label{Eq: the iterate number upper bound of algorithm SGDM}
        T \blue{\geq} \left\lceil \frac{L_{\Delta} \Delta_1+ C_{\sigma}\frac{\sigma^2}{c_b}}{\varepsilon^2} \right\rceil,
\end{align}
\blue{where $L_{\Delta}$ and $C_{\sigma}$ are given in \Cref{lemma: convergence rate sgdm}.} \blue{Consequently}, the total number of SFO calls is of order \(\calO(\varepsilon^{-4})\).
\end{theorem}
\subsection{Oracle Complexity of DualHash-StoRM}
\label{Section: complexity of StoRM}
\blue{We now present the oracle complexity of DualHash-StoRM for finding an $\varepsilon$-KKT point, with the detailed analysis deferred to \Cref{app:storm-complexity}.} Before proceeding, we introduce an additional assumption, also known as the mean-squared smoothness
assumption, widely used in variance reduction methods~\cite{arjevani2023lower,fang2018spider,tran2020hybridvrsgd,driggs2021stochastic}.
\begin{assumption}
\label{Ass: the mean squared smoothness of F}
There exists $L_{F} > 0$ such that \blue{for all $\bx, \bx' \in \RR^{q}$ and $\bB, \bB' \in \RR^{n\times d}$} \blue{with $\|\bB\|,\|\bB'\|\le R_B$}, it holds that
\begin{align}
    \label{eq:mss}
    \blue{\frac{1}{n}\sum_{i=1}^{n}}
    \left\| \nabla_{\bx} F_{i}(\bx,\bb_i) - \nabla_{\bx} F_{i}(\bx',\bb_i')\right\|^{2}
    \leq L^{2}_F \left[ \|\bx - \bx'\|^{2} + \|\bB - \bB'\|^{2} \right].
\end{align}
\end{assumption}\blue{By Jensen's inequality, \eqref{eq:mss} implies that $\nabla_{\bx}F$ is $L_F$-Lipschitz in $(\bx,\bB)$. The converse need not hold.} \blue{With a slight abuse of notation, we also use $L_F$ for this constant.}

\blue{With a slight abuse of notation, $\be^k$ denotes the algorithm-specific gradient error, and $\mathcal{Q}^k$ the corresponding iterate tuple.} Let us denote the gradient error \[\be^{k} = \blue{\bg_s^{k}} - \nabla_{\bx}F(\bx^{k}, \bB^{k})
,\]where $\blue{\bg_s^k}$ is defined in \eqref{Eq: the stochastic approximate in STORM}. \blue{We first set the parameters of \Cref{alg:our_algorithm_storm} as}\label{loc:storm-param-a}
\begin{align}
    &\eta_k \blue{\equiv\eta_T}:= \frac{\eta}{\tilde{L}_{F} T^{1/3}},\quad
    \rho_k \blue{\equiv\rho_T}:= \frac{8\rho\eta^2}{T^{2/3}},\quad
    b_k = \begin{cases} b_1, & k = 1, \\ b, & k \geq 2, \end{cases}
    \notag\\
    &\blue{0<\tau<\frac{\sqrt{30}-2}{13L_F}},\quad
    \blue{\rho\ge\frac{2L_F^{2}(1+G)}{b\tilde{L}_F^{2}}},\quad
    \blue{0<\eta\le\min\Bigl\{\frac12,\frac{1}{4\sqrt{\rho}}\Bigr\}},
    \label{Eq: the parameters conditions in dualhash-storm}
\end{align} 
where \blue{$\tilde{L}_F:=L_F+K_3+K_4$,} $\eta$ and $\rho$ are positive constants independent of $T$, \blue{$b\in\{1,\ldots,n\}$ is an integer, and $G:=1753\,G_{\mathcal{D}}^{2}$.}\label{loc:storm-param-b} We \blue{then} introduce the following Lyapunov function of \Cref{alg:our_algorithm_storm}:
% and sequence specific to \Cref{alg:our_algorithm_storm}:
\begin{align}
    \Psi_r(\mathcal{Q}^{k}) &= \calL(\bx^{k}, \bm{B}^{k}, \bLmbd^{k})
    - C_{1}\Delta_{\bB}^{k+1}+ C_{2}\Delta_{\bB}^{k}- C_{3}\Delta_{\bm{x}}^{k+1},
    \label{Eq: the Lyapunov function in STORM}
    % \\
    % \{ \calQ^{k}\}_{k \in \NN} &= \{\bx^{k}, \bB^{k}, \bm{\Lambda}^{k},
    % \bB^{k+1}, \bB^{k-1}, \bx^{k+1}\}_{k \in \NN},
    % \label{Eq: the Lyapunov sequence in STORM}
\end{align}
where $C_1$, $C_2$, and $C_3$, which are independent of $k$, are chosen as in
\eqref{Eq: the coefficients of the Lyapunov function in SGDM}. 
\blue{Similarly, \Cref{Lemma: lower bound of W} shows that \blue{the augmented potential $\EE[\Psi_r(\calQ^k)]+\frac{\eta_T}{\rho_T}\EE[\|\be^k\|^2]$} is bounded below by some $\Psi_r^{*}>-\infty$ independent of $T$; define $\tilde{\Delta}_{1}:=\calL(\bx^{1},\bB^{1},\bLmbd^{1})-\Psi_r^{*}$.}

In the next lemma, we establish an \blue{in-expectation} approximate sufficient descent property
of $\Psi_r$, with the proof provided in \blue{\Cref{app:storm-descent-proof}}.
\begin{lemma}[Approximate sufficient descent of $\Psi_r$]
\label{Lemma: approximate sufficient descent property storm}
    \blue{Under \Cref{Ass: basic regularity}(ii),} with \blue{$0<\tau<\frac{\sqrt{30}-2}{13L_F}$ and $\eta_{k}^{-1}- \tilde{L}_{F}>0$}, for all $k \geq 1$,
    \begin{align}
        \blue{\EE}[\Psi_r(\mathcal{Q}^{k+1})] - \blue{\EE}[\Psi_r(\mathcal{Q}^{k})]
        &\leq -C_{\bB}\blue{\EE}[\Delta_{\bB}^{k+1}+\Delta_{\bB}^{k}]-C_{\bx}^{k}\blue{\EE}[\Delta_{\bx}^{k+1}]+\frac{\eta_k}{2}\blue{\EE}[\|\be^{k}\|^{2}],
        \label{Eq: the descent property of Lyapunov func in storm}
    \end{align}
    where the descent coefficients are given by:
    \begin{align}
        % \label{Eq: the descent coefficients of potential function}
        \begin{cases}C_{\bB}&= C_{2}- K_{3}= \frac{K_2 - K_3 -K_1}{2}>0, \\ C_{\bx}^{k}&=\eta_{k}^{-1}- \tilde{L}_{F} \blue{> 0}.\end{cases} \notag
    \end{align}
\end{lemma}

We now analyze the convergence rate of \Cref{alg:our_algorithm_storm}\blue{, which determines the choice of the initial batch size $b_1$}.
\begin{lemma}[Convergence rate of DualHash-StoRM]
    \label{Lemma: convergence rate of storm} Under \blue{Assumptions~\ref{Ass: basic regularity}, \ref{Ass: main assumptions}, and~\ref{Ass: the mean squared smoothness of F}} and the parameter conditions \eqref{Eq: the parameters conditions in dualhash-storm}, for any $T \geq 1$,
    \begin{align}
    \label{Eq: the convergence rate of storm}\blue{\frac{1}{T}\sum_{k=1}^{T}\EE[\mathrm{dist}^{2}(\bm{0}, \mathcal K(\bx^{k}, \bB^{k}, \bm{\Lambda}^{k}))]} \leq \frac{\tilde{L}_{\Delta}\tilde{\Delta}_{1}}{T^{2/3}}+\blue{\frac{\tilde{C}_{\sigma}\sigma^{2}}{4\rho\eta^{2}b_1T^{1/3}}}+\frac{32\rho\eta^{2}\tilde{C}_{\sigma}\sigma^{2}}{b\,T^{2/3}},
    \end{align}
    where
    \blue{\[
    \tilde{L}_{\Delta}:=\max\Bigl\{\frac{S_1}{C_{\bB}},\frac{S_2}{C_{\bB}},\frac{48\tilde{L}_F}{\eta}+\frac{16\eta L_F^{2}(1+3\tau^{2})}{\tilde{L}_F}\Bigr\},\qquad \tilde{C}_{\sigma}:=\frac{\eta\tilde{L}_{\Delta}}{2\tilde{L}_F}
    \]}
    are constants independent of $T$, with $S_1,S_2$ defined in \eqref{Eq: the coefficients of stationarity error bound in SGDM}.
\end{lemma}

\blue{As noted in~\cite{xu2023momentum}, a moderately large initial batch size improves the complexity order at a negligible cost.} \label{loc:storm-b1a}\blue{To achieve the desired oracle complexity, we choose}
\begin{align}
\label{Eq: the initial batch size of dualhash-storm}
    b_1=\lceil c_bT^{1/3}\rceil\in\{1,\ldots,n\},
\end{align}
\blue{where $c_b>0$ is a constant independent of $T$.}\label{loc:storm-b1b} \blue{The resulting complexity is given in \Cref{Theorem: the oracle complexity in STORM}, whose proof appears in \Cref{app:storm-oracle-proof}.}
\begin{theorem}[Oracle complexity of DualHash-StoRM]
\label{Theorem: the oracle complexity in STORM}
Under \blue{Assumptions~\ref{Ass: basic regularity}, \ref{Ass: main assumptions}, and~\ref{Ass: the mean squared smoothness of F}}, let \(\{\bx^k, \bB^k, \bLmbd^k\}_{k \in \NN}\) be the iterate sequence from \Cref{alg:our_algorithm_storm} with the parameters satisfying \eqref{Eq: the parameters conditions in dualhash-storm} \blue{and} \eqref{Eq: the initial batch size of dualhash-storm}\blue{.} Then, \blue{for $R$ drawn uniformly from $\{1,\dots,T\}$ independently of the mini-batches,} \( (\bx^R, \bB^R, \bLmbd^R )\) satisfies \eqref{equ:esp} whenever
\begin{align}
    \label{Eq: the iterate number upper bound of algorithm VR}
        T \blue{\geq} \left\lceil \frac{\Bigl(\tilde{L}_{\Delta}\tilde{\Delta}_{1}+\tilde{C}_{\sigma}\bigl(\frac{\sigma^{2}}{4\rho\eta^{2}c_b}+\frac{32\rho\eta^{2}\sigma^{2}}{b}\bigr)\Bigr)^{3/2}}{\varepsilon^3} \right\rceil,
\end{align}
with \blue{$\tilde{L}_{\Delta}$ and $\tilde{C}_{\sigma}$ from \Cref{Lemma: convergence rate of storm}}. \blue{Consequently,} the total number of SFO calls is of order \(\calO(\varepsilon^{-3})\).
\end{theorem}
\begin{remark}\label{rem:complexity}
\begin{enumerate}[label=(\roman*), leftmargin=2em]
\item The results in Theorems \ref{Theorem: the oracle complexity in SGDM} and \ref{Theorem: the oracle complexity in STORM} \blue{establish stochastic first-order complexity bounds for the proposed multi-block primal-dual scheme, in a setting} where existing work has largely focused on one-block~\cite{xu2023momentum,shi2025momentum} or deterministic \blue{methods}~\cite{ADMMNolinear2024}.
    \item  \Cref{alg:our_algorithm_sgdm} \blue{under the parameter conditions \eqref{Eq: the parameters conditions of dualhash-sgdm in suppl} and} \eqref{Eq: the parameters conditions of dualhash-sgdm} attains both oracle and sample \(\mathcal{O}(\varepsilon^{-4})\) complexities, which matches that of stochastic first-order methods (SFOMs) for one-block stochastic optimization~\cite{Ghadimi2013MinibatchSA}. Meanwhile, \Cref{alg:our_algorithm_storm} with the refined constant parameters from \eqref{Eq: the parameters conditions in dualhash-storm} attains both oracle and sample $\mathcal{O}(\varepsilon^{-3})$ complexities.
\end{enumerate}
\end{remark}

\subsection{Oracle Complexity Guarantee for Problem~\eqref{Eq: the primal deep hashing problem}}
\label{Sec: guarantees original objective}
\blue{We now present, with \blue{proof in \Cref{app:original-objective}}, the oracle complexity guarantee of DualHash for the original Problem~\eqref{Eq: the primal deep hashing problem}, which follows from combining the KKT-residual bounds in \Cref{Theorem: the oracle complexity in SGDM,Theorem: the oracle complexity in STORM} with \Cref{prop:original-certificate}.}

\begingroup\color{blue}
\begin{corollary}\label{cor:original-stationarity}
Under the assumptions of \Cref{Theorem: the oracle complexity in SGDM} (resp.\ \Cref{Theorem: the oracle complexity in STORM}), \blue{let $(\bx^R,\bB^R,\bLmbd^R)$ be the output of \Cref{alg:our_algorithm_sgdm} (resp.\ \Cref{alg:our_algorithm_storm}), obtained with} $\calO(\varepsilon^{-4})$ (resp.\ $\calO(\varepsilon^{-3})$) SFO calls. Then $\bx^R$ is $(\delta_R,\varepsilon_R)$-near-stationary for \blue{Problem~\eqref{Eq: the primal deep hashing problem}}, where \blue{$\delta_R$ and $\varepsilon_R$ are the bounds of \Cref{prop:original-certificate} at $(\bx^R,\bB^R,\bLmbd^R)$ with $L=L_0$, and satisfy
\begin{align*}
\EE[\delta_R] &\le\sqrt{\frac{2}{L_0}}\Bigl(\sqrt{\kappa}\,\varepsilon+\sqrt{\EE[h(\mathcal{D}(\bx^R))]}\Bigr),\\
\EE[\varepsilon_R] &\le\Bigl(\sqrt{1+G_{\mathcal{D}}^2}+2\sqrt{2\kappa L_0}\Bigr)\varepsilon+2\sqrt{2L_0\,\EE[h(\mathcal{D}(\bx^R))]},
\end{align*}
with $\kappa$ and $L_0$ given in \Cref{prop:original-certificate}.}
\end{corollary}
\endgroup

\blue{\Cref{cor:original-stationarity} makes \Cref{prop:original-certificate} concrete for the algorithm output: the optimization error is reduced to $\calO(\varepsilon)$ at the oracle complexities of \Cref{Theorem: the oracle complexity in SGDM,Theorem: the oracle complexity in STORM}, whereas the structural error is determined by the quantization error of the output, which is inherent to continuous relaxation. In particular, for the $\ell_1$ W-regularizer, $\bx^R$ is \blue{$(\delta_R,\varepsilon_R)$-near-stationary for Problem~\eqref{Eq: the primal deep hashing problem}, with $\EE[\delta_R]=\calO(\varepsilon)$ and $\EE[\varepsilon_R]=\calO(\varepsilon)$,} whenever $\lambda\,\EE[Q_1(\bx^R)]=\calO(\varepsilon^2)$, \blue{i.e., when the expected quantization error is sufficiently small.}}
\begingroup
\setlength{\textfloatsep}{10pt plus 2pt minus 2pt}
\setlength{\floatsep}{8pt plus 2pt minus 2pt}
\setlength{\intextsep}{10pt plus 2pt minus 2pt}
\captionsetup{skip=4pt}
\renewcommand{\topfraction}{0.95}
\renewcommand{\bottomfraction}{0.85}
\renewcommand{\textfraction}{0.05}
\section{Numerical Experiments}
\label{Section: Numerical experiments} 
We conduct extensive experiments to evaluate the effectiveness of DualHash against several state-of-the-art hashing methods on three
standard image retrieval datasets. All experiments use PyTorch 1.12.1 with CUDA 11.3 on an NVIDIA V100 GPU platform.
\subsection{Experimental Setup} 
\label{Subsection: Experimental Setup}
We use three standard datasets: single-label \textbf{CIFAR-10} (10 classes; 1,000/500/500 images per class for training/validation/testing), multi-label \textbf{NUS-WIDE} (21 most frequent concepts; 700/300/300 images per class), and single-label \textbf{ImageNet-100}~\cite{Cao2017HashNetDL} (100 categories; 100/30/50 images per category). Our method is compatible with different similarity measurement paradigms, including \blue{pointwise~\cite{shenSupervisedDiscreteHashing2015}, pairwise~\cite{liuDeepSupervisedHashing}, and tripletwise~\cite{wang2016DTSH}} methods. In our experiments, we focus on deep supervised pairwise hashing using the pairwise cross-entropy loss and the $\ell_1$ W-regularizer; see \Cref{app:experiments} for more details. We compare against eight baselines: SDH~\cite{shenSupervisedDiscreteHashing2015}, four deep hashing methods with \blue{W-type regularizers} (DHN~\cite{zhuDeepHashingNetwork2016}, DSH~\cite{liuDeepSupervisedHashing}, DTSH~\cite{wang2016DTSH}, MDSHC~\cite{wang2023deep}), and three methods addressing quantization error (HashNet~\cite{Cao2017HashNetDL}, DSDH~\cite{li2017deep}, OrthoHash~\cite{hoeOneLossAll2021}), reporting mAP, precision curves (P@topK), and precision within Hamming radius 2 (P@r=2)~\cite{liuDeepSupervisedHashing}. All deep methods employ AlexNet (default) or ResNet-50 as the backbone with ELU activation to satisfy the smoothness assumption (\Cref{Ass: basic regularity}(ii)). During training, we monitor validation performance in real time to observe convergence behaviors and adjust hyperparameters accordingly.\label{loc:q61}

\label{Paragraph: Hyperparameter Settings}All models are trained for at most 300 epochs with early stopping based on validation mAP, using a mini-batch size of 256 for CIFAR-10 and 128 for NUS-WIDE and ImageNet-100. For DualHash-StoM, $\eta_k$ is selected from $\{10^{-4}, 10^{-3}, 10^{-2}, 10^{-1}\}$ with step decay, and $\alpha_k = \beta_k = 0.905$. For DualHash-StoRM, $\eta_k$ is chosen from $\{5 \times 10^{-3}, 5 \times 10^{-2}, 5 \times 10^{-1}\}$ with $\mathcal{O}(T^{-1/3})$ scaling, and $\rho_k \in \{0.01, 0.05, 0.1, 0.2, 0.7\}$. Across all datasets, we set $\tau = 10^{-2}$, $\lambda = 5 \times 10^{-2}$, and $\gamma = 3$. In our practical implementation, the $\bm{\Lambda}$-update uses step size $0.1\tau$. 
% All results represent averages over 10 independent runs. 
% \vspace{3pt}
\subsection{Convergence Validation}
\label{Subsection: Convergence validation and performance efficiency analysis}
\begin{table}[!tbp]
    \centering
    \caption{The mAP results of DualHash and baselines with different numbers of bits on CIFAR-10 and NUS-WIDE datasets.}
    \label{Table: baselines mAP}
    \small
    \setlength{\tabcolsep}{3pt}
    % \begin{tabular}{1\textwidth}{@{\extracolsep{\fill}}ccccccccc@{}}
    \begin{tabular*}{\linewidth}{@{\extracolsep{\fill}}lcccccccc@{}}
        \toprule
        \multirow{2}{*}{Methods} & \multicolumn{4}{c}{CIFAR-10 (mAP@All)} & \multicolumn{4}{c}{NUS-WIDE (mAP@5000)} \\
        \cmidrule(lr){2-5} \cmidrule(lr){6-9}
         & 16 bits & 32 bits & 48 bits & 64 bits & 16 bits & 32 bits & 48 bits & 64 bits \\
        \midrule
        SDH~\cite{shenSupervisedDiscreteHashing2015} & 0.4254 & 0.4575 & 0.4751 & 0.4855 & 0.4756\textsuperscript{*} & 0.5545\textsuperscript{*} & 0.5786\textsuperscript{*} & 0.5812\textsuperscript{*} \\
        SDH-CNN & 0.6665 & 0.6501 & 0.6484 & 0.6260 & 0.4872 & 0.6529 & 0.6451 & 0.5239 \\
        DSH~\cite{liuDeepSupervisedHashing} & 0.6610 & 0.7530 & 0.7817 & 0.8010 & 0.5195 & 0.6629 & 0.6617 & 0.6634 \\
        DTSH~\cite{wang2016DTSH}  & 0.7661 & 0.7501 & 0.7651 & 0.7744 & 0.6590 & 0.6845 & 0.7130 & 0.7371 \\ 
        DHN~\cite{zhuDeepHashingNetwork2016} & 0.6568 & 0.7608 & 0.7838 & 0.8051 & 0.5562 & 0.6427 & 0.6808 & 0.7042 \\
        HashNet~\cite{Cao2017HashNetDL} & 0.3354 & 0.6511 & 0.7732 & 0.8396 & 0.5879 & 0.6769 & 0.7221 & 0.7249 \\
        DSDH~\cite{li2017deep} & 0.2732 & 0.3648 & 0.4222 & 0.4941 & 0.5153 & 0.5953 & 0.6338 & 0.6619 \\
        OrthoHash~\cite{hoeOneLossAll2021} & 0.8070 & 0.8059 & 0.8387 & 0.8355 & \textbf{0.6653} & 0.6912 & 0.7083 & 0.7176 \\
        MDSHC~\cite{wang2023deep} & 0.8142 & 0.8238 & 0.8350 & 0.8262 & 0.5902 & 0.6452 & 0.6081 & 0.6493 \\
         \cmidrule(lr){2-9} 
        \textbf{DualHash-StoRM} & 0.8037 & 0.8051 & 0.8168 & 0.8345 & 0.6485 & 0.6802 & 0.6951 & 0.6982 \\
        \textbf{DualHash-StoM} & \textbf{0.8215} & \textbf{0.8481} & \textbf{0.8534} & \textbf{0.8539} & 0.6339 & \textbf{0.7002} & \textbf{0.7248} & \textbf{0.7448} \\
        \bottomrule
    \end{tabular*}
    \par\vspace{3pt}
    \parbox{\linewidth}{\footnotesize
    \textsuperscript{*}SDH results on NUS-WIDE are cited from HashNet and reported for reference under a similar protocol.
    \textbf{Bold} values indicate the best results.
    }
\end{table}
\begin{figure}[!tp]
    % \vspace{-1mm}
    \centering
    \includegraphics[width=0.95\linewidth]{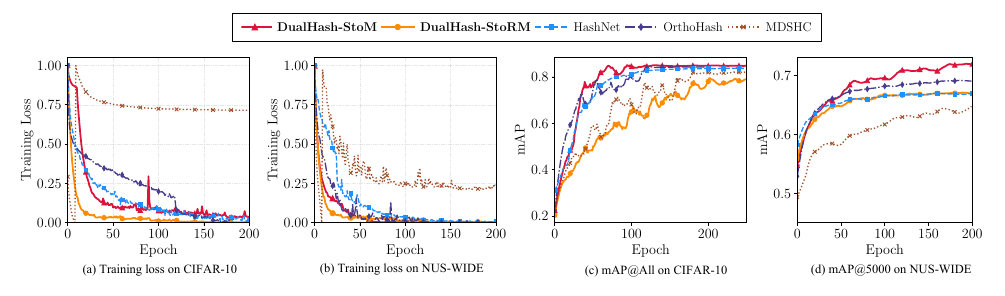}
    \caption{Training loss and the mAP results of DualHash and baselines for 64-bit codes on CIFAR-10 and NUS-WIDE datasets.}
    \label{fig:training_loss_map_2datasets}
\end{figure}
\begin{figure}[!tp]
    % \vspace{-1mm}
    \centering
    \includegraphics[width=0.95\linewidth]{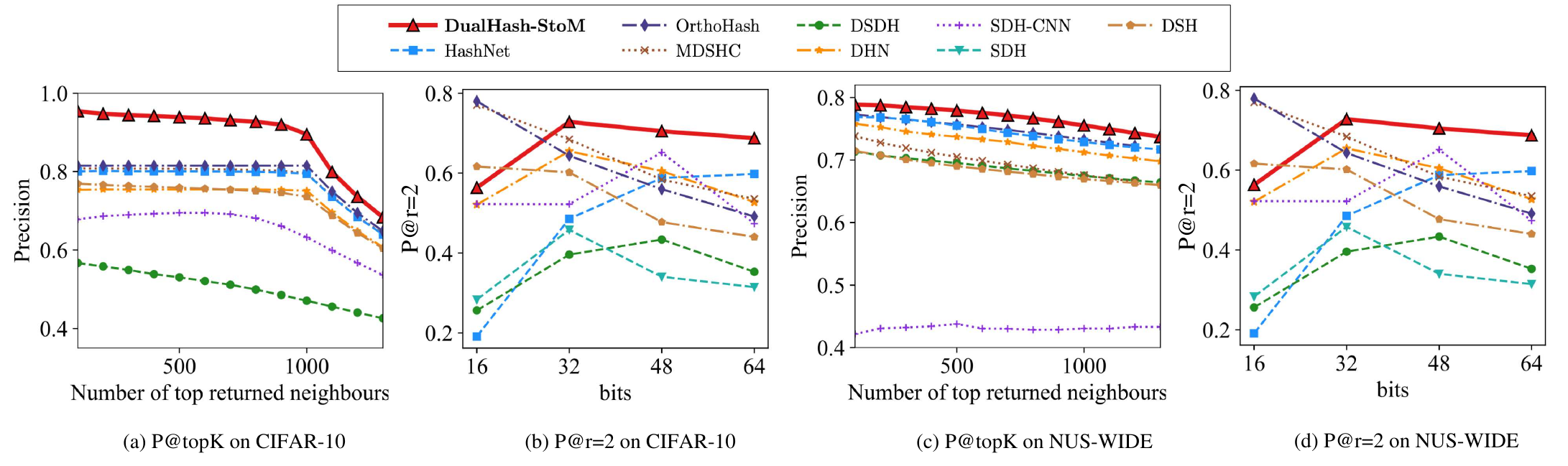}
    \caption{Retrieval performance of DualHash and baselines on CIFAR-10 and NUS-WIDE datasets. (a), (c): P@topK for 64-bit codes; (b), (d): P@r=2 with different numbers of bits.}
 \label{fig:retrieval_2datasets}
\end{figure}
\begin{table}[!tbp]
    \centering
    \caption{The mAP results of DualHash and baselines with different numbers of bits on CIFAR-10 and NUS-WIDE datasets with the ResNet-50 backbone.}
    \label{Table: performance_comparison_resnet50_small}
    \small
    \setlength{\tabcolsep}{3pt}
    \begin{tabular*}{\linewidth}{@{\extracolsep{\fill}}lcccccccc@{}}
        \toprule
         & \multicolumn{4}{c}{CIFAR-10 (mAP@All)} & \multicolumn{4}{c}{NUS-WIDE (mAP@5000)} \\
        \cmidrule(lr){2-5} \cmidrule(lr){6-9}
         Method (ResNet-50) & 16 bits & 32 bits & 48 bits & 64 bits & 16 bits & 32 bits & 48 bits & 64 bits \\
        \midrule
        DSH~\cite{liuDeepSupervisedHashing} & 0.7403 & 0.8116 & 0.8217 & 0.8239 & 0.5559 & 0.6952 & 0.6981 & 0.7165 \\
        OrthoHash~\cite{hoeOneLossAll2021} & \textbf{0.9032} & 0.9182 & 0.9202 & 0.9224 & \textbf{0.7448} & 0.7600 & 0.7569 & 0.7863 \\
        MDSHC~\cite{wang2023deep} & 0.8980 & 0.9007 & 0.9063 & 0.9227 & 0.6428 & 0.7002 & 0.7248 & 0.7758 \\
        \cmidrule(lr){2-9}
        DualHash-StoM & 0.8559 & \textbf{0.9202} & \textbf{0.9392} & \textbf{0.9416} & 0.6538 & \textbf{0.7783} & \textbf{0.7721} & \textbf{0.8016} \\
        \bottomrule
    \end{tabular*}
    \par\vspace{3pt}
    \parbox{\linewidth}{\footnotesize
    \textbf{Bold} values indicate the best results.
    }
\end{table}
We provide empirical evidence that is consistent with our theoretical results (Theorems \ref{Theorem: the oracle complexity in SGDM} and \ref{Theorem: the oracle complexity in STORM}) and compare with HashNet, OrthoHash, and MDSHC with 64-bit codes on CIFAR-10 and NUS-WIDE. 
\Cref{fig:training_loss_map_2datasets} shows the training loss and validation mAP curves over epochs.
As shown in \Cref{fig:training_loss_map_2datasets} (a)-(b), DualHash-StoRM achieves the fastest convergence and lowest loss values on CIFAR-10 and NUS-WIDE, requiring only approximately 25 epochs on CIFAR-10. This aligns with our theoretical outcome: the variance reduction mechanism accelerates convergence. Though slower than DualHash-StoRM, DualHash-StoM still outperforms other baselines (approximately 75 epochs on CIFAR-10). 
Meanwhile, \Cref{fig:training_loss_map_2datasets} (c)-(d) show that although DualHash-StoRM converges faster during training, DualHash-StoM ultimately exhibits superior retrieval performance compared to other baselines. 
% \vspace{3pt}
\subsection{Comparison of Retrieval Performance}
\label{Subsection: Comparison of retrieval performance}
\begin{table}[!tbp]
% \vspace{2ex}
    \centering
    \caption{The mAP results of DualHash and baselines with different numbers of bits on the ImageNet-100 dataset with the ResNet-50 backbone.}
    \label{tab:performance_imagenet_resnet50}
    \small
    \begin{tabular}{lcccc}
     \toprule
 & \multicolumn{4}{c}{ImageNet-100 (mAP@1000)} \\
\cmidrule(lr){2-5}
Method (ResNet-50) & 16 bits & 32 bits & 48 bits & 64 bits \\
\midrule
DSH~\cite{liuDeepSupervisedHashing} & 0.7179 & 0.7448 & 0.7335 & 0.7585 \\
OrthoHash~\cite{hoeOneLossAll2021} & 0.8040 & \textbf{0.8292} & 0.8510 & 0.8636 \\
MDSHC~\cite{wang2023deep} & \textbf{0.8047} & 0.8131 & 0.8561 & 0.8657 \\
 \cmidrule(lr){2-5}
DualHash-StoM & 0.7402 & 0.7705 & \textbf{0.8662} & \textbf{0.8714} \\
        \bottomrule
    \end{tabular}
    \par\vspace{3pt}
    \parbox{\linewidth}{\footnotesize
    \textbf{Bold} values indicate the best results.
    }
\end{table}
We report retrieval performance on CIFAR-10 and NUS-WIDE datasets, as shown in Tables \ref{Table: baselines mAP}
% , \ref{Table: performance_comparison_resnet50_small}, \ref{tab:performance_imagenet_resnet50} 
and \ref{Table: performance_comparison_resnet50_small}. 
In terms of mAP, DualHash methods consistently outperform most baselines, including recent advanced methods like OrthoHash and MDSHC.
On CIFAR-10, DualHash-StoM achieves 0.8539 mAP with 64-bit codes, outperforming OrthoHash (0.8355). On the NUS-WIDE dataset, it reaches 0.7448 mAP, significantly superior to MDSHC (0.6493).
Although slightly less accurate than the StoM variant, DualHash-StoRM remains comparable to other baselines, providing a trade-off between accuracy and efficiency. 

To further demonstrate the scalability of our approach, we evaluate our method and three baselines using a \textbf{ResNet-50 backbone}. As shown in \Cref{Table: performance_comparison_resnet50_small}, this modern architecture brings substantial performance gains in terms of mAP for all methods across CIFAR-10 and NUS-WIDE. Our DualHash-StoM, for example, achieves mAP improvements of 10.3\% on CIFAR-10 (0.9416 vs 0.8539) and 7.6\% on NUS-WIDE (0.8016 vs 0.7448) at 64 bits. Nevertheless, our method consistently maintains peak performance at 32 bits or more, outperforming all baseline methods. On the more complex ImageNet-100 dataset, as reported in \Cref{tab:performance_imagenet_resnet50}, DualHash-StoM reaches 86.62\% and 87.14\% mAP at 48 and 64 bits, outperforming all baselines.

For P@topK (\Cref{fig:retrieval_2datasets} (a),(c)), DualHash maintains 0.9 precision at lower top-k values on CIFAR-10, significantly outperforming other methods. This indicates that our method is suitable for precision-oriented image retrieval systems.
For P@r=2 (\Cref{fig:retrieval_2datasets} (b),(d)), DualHash-StoM consistently maintains peak performance at 32 bits and above, with 32 bits being optimal across CIFAR-10 and NUS-WIDE. These results align with the fact that as bit length increases, the Hamming space becomes sparse and few data points fall within the Hamming ball with radius 2 \cite{norouzi2012fast}.
% \vspace{3pt} 

\subsection{Quantization Error Analysis}
\label{Subsection: Further analysis}
\begin{figure}[!tp]
    \vspace{-1ex}
    \centering
    \includegraphics[width=\linewidth]{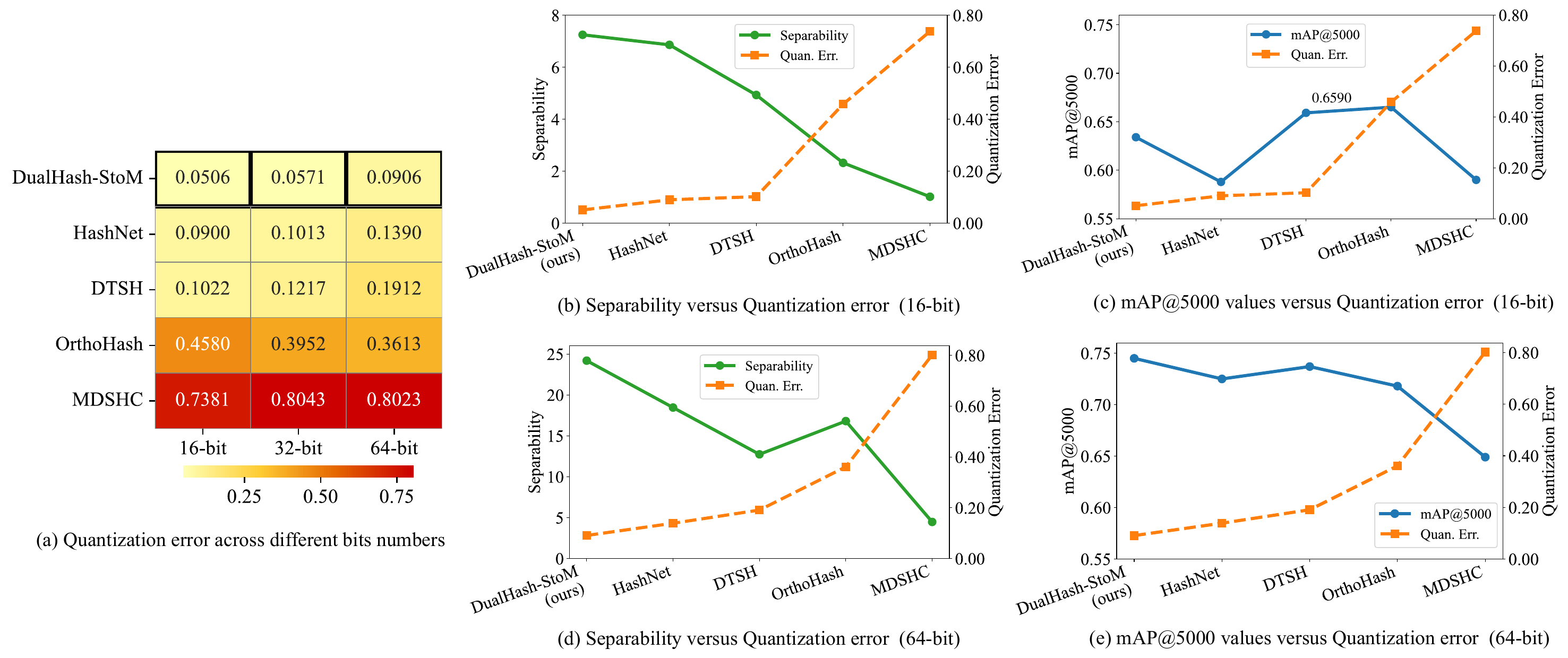}
    \vspace{-2ex}
    \caption{Quantization error and its impact on retrieval performance across different numbers of bits on  NUS-WIDE.}
    \label{fig:quantization_error}
    \vspace{1ex}
    \includegraphics[width=\linewidth]{results_images/hamming_hig_64.pdf}
    \vspace{-2ex}
    \caption{Histogram of intra-class and inter-class Hamming distances of DualHash and baselines for 64-bit codes on the NUS-WIDE dataset. The arrow annotation is the separability in Hamming distances, \(\EE [D_{inter}] - \EE [D_{intra}]\). We normalized the frequency so that the sum over all bins equals 1.}
    \label{fig:His_64_NUSWIDE}
    \vspace{1ex}
    \includegraphics[width=\linewidth]{results_images/hamming_hig_16.pdf}
    \vspace{-2ex}
    \caption{Histogram of intra-class and inter-class Hamming distances of DualHash and baselines with 16-bit codes on NUS-WIDE.}
    \label{fig:His_16_NUSWIDE}
    \vspace{-1ex}
\end{figure}
Following~\cite{hoeOneLossAll2021}, we examine the quantization error (\Cref{fig:quantization_error}) and the intra- and inter-class Hamming distance distributions (\Cref{fig:His_64_NUSWIDE,fig:His_16_NUSWIDE}) across bit lengths $K\blue{=d}$ on NUS-WIDE, comparing DualHash-StoM with HashNet~\cite{Cao2017HashNetDL}, OrthoHash~\cite{hoeOneLossAll2021}, DTSH~\cite{wang2016DTSH}, and MDSHC~\cite{wang2023deep}. As shown in \Cref{fig:quantization_error}(a), DualHash-StoM attains the lowest quantization error across all bit lengths, which compresses intra-class distances and enlarges inter-class separability. At $K = 64$ (\Cref{fig:His_64_NUSWIDE}), for instance, DualHash-StoM and HashNet share similar inter-class means near $K/2 = 32$, but DualHash-StoM achieves a lower quantization error ($0.0906$ vs.\ $0.1390$) and wider separability ($24.19$), whereas MDSHC's large quantization error ($0.8023$) yields a much smaller separability ($4.47$).

While lower quantization error consistently improves Hamming separability (\Cref{fig:quantization_error}(b) and (d)), its effect on retrieval accuracy depends on the bit length. In the high-bit regime ($K = 64$, \Cref{fig:quantization_error}(d)--(e)), lower quantization error helps \blue{to} produce high mAP, and DualHash-StoM achieves both the lowest quantization error and the highest mAP. In the low-bit regime ($K = 16$, \Cref{fig:quantization_error}(b)--(c)), however, overly small quantization error can over-compress intra-class codes at the expense of fine-grained semantic relationships on complex multi-label datasets such as NUS-WIDE, where intermediate quantization error (e.g., DTSH) yields higher mAP.
% \noindent\textbf{Parameter sensitivity.} DualHash-StoM is robust to the hyperparameters $\eta$ and $\gamma$: for 64-bit codes, the mAP varies by less than 3\% on CIFAR-10 and 8\% on NUS-WIDE over a grid search (\Cref{app:sensitivity}).
\endgroup

\section{Conclusion}
\label{Section: Conclusion}
This paper proposed DualHash, a stochastic primal--dual framework for deep hashing with non-asymptotic oracle complexity guarantees. By decoupling the nonlinear composition and reformulating the nonconvex W-type regularizer {via its conjugate}, DualHash admits closed-form proximal updates\blue{, while the resulting min--max problem captures the stationarity of the original objective up to the quantization error}. We established $\mathcal{O}(\varepsilon^{-4})$ and $\mathcal{O}(\varepsilon^{-3})$ complexities for DualHash-StoM and DualHash-StoRM, respectively, and validated their low quantization error and strong retrieval performance on CIFAR-10, NUS-WIDE, and ImageNet-100.

\section*{Declarations}
AI-based tools were used to assist with language editing and checking of mathematical derivations; all content was verified by the authors. The authors assume responsibility for all content.

\bibliographystyle{siamplain}
\bibliography{refers}
% \newpage
\appendix
\section*{\textbf{Appendix}}
\section{\blue{Proofs of the Regularizer and Optimality Properties}}
\label{app:optimality-properties}
\subsection{\blue{Proof of \Cref{lem:wtype-conj}}}
\label{app:wtype-conjugate-proof}
\begingroup\color{blue}
\begin{proof}
\emph{Part (a).} Define \(\psi(t):=\phi(\max\{|t|,1\})\). By \ref{W1}--\ref{W2}, \(\phi\) is nonnegative and convex on \([1,+\infty)\) with \(\phi(1)=0\), and is therefore nondecreasing on \([1,+\infty)\). Consequently, \(\psi\) is convex on \(\mathbb{R}\) and satisfies \(\psi\le\phi\). Conversely, any convex minorant \(q\le\phi\) satisfies \(q(\pm1)\le\phi(\pm1)=0\), so convexity gives \(q\le0=\psi\) on \([-1,1]\), while \(q\le\phi=\psi\) on \(\{t:|t|\ge1\}\). By~\cite[Theorem~12.2]{rockafellar1970convex}, \(\phi^{**}=\psi\). Since \(\phi(t)>0=\psi(t)\) for \(|t|<1\) by \ref{W1}, we have \(\phi^{**}(t)=\phi(t)\) if and only if \(|t|\ge1\).

\emph{Part (b).} As a supremum of affine functions, \(\phi^*\) is lower semicontinuous and convex, and the evenness of \(\phi\) implies that \(\phi^*\) is even. Evaluating the supremum at \(t=\pm1\) yields \(\phi^*(y)\ge|y|\), and \(\phi^*(0)=-\inf_{t\in\mathbb{R}}\phi(t)=0\) shows that \(\phi^*\) is proper. Moreover, let \(a:=\phi(2)>0\). By \ref{W2}, \(\phi(1)=0\), and the convexity of \(\phi\) on \([1,+\infty)\), the slopes \((\phi(t)-\phi(1))/(t-1)\) are nondecreasing for \(t>1\), which together with evenness yields \(\phi(t)\ge a(|t|-1)\) for all \(|t|\ge2\). Combined with \(\phi\ge0\) on \([-2,2]\), this linear growth gives \(ty-\phi(t)\le\max\{a,2|y|\}\) for all \(t\in\mathbb{R}\) whenever \(|y|<a\), so \((-a,a)\subset\operatorname{dom}\phi^*\).

Finally, separability gives \(h^*(\bLmbd)=\sum_{i,j}\phi^*(\Lambda_{ij})\) and \(h^{**}(\bB)=\sum_{i,j}\phi^{**}(b_{ij})\). Since \(\phi\ge\phi^{**}\) entrywise, \(h(\bB)=h^{**}(\bB)\) holds if and only if \(\phi(b_{ij})=\phi^{**}(b_{ij})\) for all \(i,j\), which by (a) is equivalent to \(|b_{ij}|\ge1\) for all \(i,j\).
\end{proof}
\endgroup

\subsection{Proof of \Cref{prop:zerogap}}
\label{app:zero-gap-proof}
\begin{proof}
    According to~\cite[Theorem~2.158]{temporary-citekey-11039},
    $(\bx^{*}, \bB^{*}, \bLmbd^{*})$ are optimal solutions to the
    primal problem \eqref{Prop: block-wise regularized problem} and the dual problem
    \eqref{Prop: the dual problem} \blue{with zero duality gap} if and only if:
    \begin{align}
        \label{Eq: the sufficient and necessary condition}
        \begin{cases}
        (\bx^{*}, \bB^{*}) \in \argmin\limits_{\bx \in \RR^{q}, \bB \in \RR^{n\times d}}\calL(\blue{\bx, \bB}, \bLmbd^{*}), \\
        0 = h(\bB^{*}) + h^{*}(\bLmbd^{*}) - \langle \bLmbd^{*}, \bB^{*}\rangle.
        \end{cases}
    \end{align}
    The first condition in \eqref{Eq: the sufficient and necessary condition} implies
    that $\bm{0}$ belongs to the subdifferential of $\calL$ with respect to
    $(\bx, \bB)$ at $(\bx^{*}, \bB^{*}, \bLmbd^{*})$. Since $\calL$
    is continuously differentiable in $(\bx, \bB)$, we have
    \begin{align}
        \label{Eq: the sufficient and necessary condition for primal}
        \begin{cases}
        \bm{0} = \blue{\nabla_{\bx}F(\bx^{*}, \bB^{*})},\\
        \bm{0} = \blue{\nabla_{\bB}F(\bx^{*}, \bB^{*})} + \bLmbd^{*}.
        \end{cases}
    \end{align}
    For the second condition in \eqref{Eq: the sufficient and necessary condition},
    \blue{since $h^{**}\le h$, the Fenchel--Young inequality for $h^{**}$ yields
    $\langle\bLmbd^*,\bB^*\rangle\le h^{**}(\bB^*)+h^*(\bLmbd^*)\le h(\bB^*)+h^*(\bLmbd^*)=\langle\bLmbd^*,\bB^*\rangle$.
    Thus $h^{**}(\bB^*)+h^*(\bLmbd^*)=\langle\bLmbd^*,\bB^*\rangle$, which, by the subgradient inversion rule~\cite[Theorem~23.5]{rockafellar1970convex}, gives}
    \begin{align}
        \label{Eq: the sufficient and necessary condition for dual}\bB^{*}\in \partial h^{*}(\bLmbd^{*}).
    \end{align}
    Combining \eqref{Eq: the sufficient and necessary condition for primal} and \eqref{Eq: the sufficient and necessary condition for dual}, we obtain \blue{$\bm{0}\in\mathcal{K}(\bx^{*}, \bB^{*}, \bLmbd^{*})$ with $\mathcal{K}$ in \eqref{Eq: KKT mapping}}, completing the proof.
\end{proof}

\subsection{\blue{Proof of \Cref{prop:transfer}}}
\label{app:transfer-proof}
\begingroup\color{blue}
\begin{lemma}\label{lem:no-inner-stationary}
If \(\phi\) satisfies \ref{W1} and \ref{W3}, then \(0\notin\partial\phi(t)\) for every \(t\in(-1,1)\).
\end{lemma}
\begin{proof}
Fix \(t\in(-1,1)\) and choose \(\theta\in(|t|,1)\). For any \(t'\in[0,\theta)\), condition~\ref{W3} gives \(\phi(t'+\delta)-\phi(t')\le-c_{\theta}\delta\) for all sufficiently small \(\delta>0\), so every \(v\in\hat\partial\phi(t')\) satisfies \(v\le-c_{\theta}\). By the evenness of \(\phi\) in \ref{W1}, every \(v\in\hat\partial\phi(t')\) with \(t'\in(-\theta,0]\) satisfies \(v\ge c_{\theta}\) (with \(\hat\partial\phi(0)=\emptyset\)). Consequently, \(|v|\ge c_{\theta}\) for all \(t'\in(-\theta,\theta)\) and \(v\in\hat\partial\phi(t')\). By~\eqref{Eq: limiting subdifferential}, \(|v|\ge c_{\theta}>0\) for every \(v\in\partial\phi(t)\), and thus \(0\notin\partial\phi(t)\).
\end{proof}

\begin{proof}[Proof of \Cref{prop:transfer}]
By the KKT conditions~\eqref{Eq: KKT conditions}, we have \(\nabla_{\bx}F(\tilde{\bx},\tilde{\bB})=\bm{0}\), \(\tilde{\bLmbd}=-\nabla_{\bB}F(\tilde{\bx},\tilde{\bB})\), and \(\tilde{\bB}\in\partial h^*(\tilde{\bLmbd})\). Thus condition~(a) holds if and only if \(\tilde{\bLmbd}\in\partial h(\tilde{\bB})\). Moreover, since \(h^*\) is proper, lower semicontinuous, and convex, the subgradient inversion rule~\cite[Theorem~23.5]{rockafellar1970convex} yields
\begin{equation}\label{Eq: inversion rule of the conjugate}
\tilde{\bB}\in\partial h^*(\tilde{\bLmbd})
\iff \tilde{\bLmbd}\in\partial h^{**}(\tilde{\bB})
\iff h^{**}(\tilde{\bB})+h^*(\tilde{\bLmbd})=\langle\tilde{\bLmbd},\tilde{\bB}\rangle.
\end{equation}

\emph{(b) $\iff$ (c).} Combining the Fenchel equality in~\eqref{Eq: inversion rule of the conjugate} with \(h\ge h^{**}\) shows that condition~(b) holds if and only if \(h(\tilde{\bB})=h^{**}(\tilde{\bB})\), which by \Cref{lem:wtype-conj} is equivalent to \(|\tilde{b}_{ij}|\ge1\) for all \(i,j\).

\emph{(b) $\Rightarrow$ (a).} By the definition of \(h^*\), condition~(b) means that \(\tilde{\bB}\in\argmin_{\bB\in\RR^{n\times d}}\{h(\bB)-\langle\tilde{\bLmbd},\bB\rangle\}\). Applying Fermat's rule~\cite[Theorem~10.1 and Exercise~8.8(c)]{Rockafellar1998VariationalA} gives \(\tilde{\bLmbd}\in\partial h(\tilde{\bB})\), and hence (a).

\emph{(a) $\Rightarrow$ (c).} Combining \(\tilde{\bLmbd}\in\partial h(\tilde{\bB})\) with \(\tilde{\bLmbd}\in\partial h^{**}(\tilde{\bB})\) from~\eqref{Eq: inversion rule of the conjugate} and using the separability of \(h\) and \(h^{**}\)~\cite[Proposition~10.5]{Rockafellar1998VariationalA}, we obtain \(\tilde{\Lambda}_{ij}\in\partial\phi(\tilde{b}_{ij})\cap\partial\phi^{**}(\tilde{b}_{ij})\) for all \(i,j\). If \(|\tilde{b}_{ij}|<1\) for some \((i,j)\), then \Cref{lem:wtype-conj}(a) gives \(\partial\phi^{**}(\tilde{b}_{ij})=\{0\}\), which implies \(0\in\partial\phi(\tilde{b}_{ij})\), contradicting \Cref{lem:no-inner-stationary}. Therefore, \(|\tilde{b}_{ij}|\ge1\) for all \(i,j\).
\end{proof}
\endgroup

\subsection{\blue{Proofs of \Cref{lem:smooth-lower-bound,prop:original-certificate}}}
\label{app:original-certificate-proof}
\begingroup
\color{blue}
The proof of \Cref{prop:original-certificate} uses the smooth lower model
\[
s_{\bLmbd}(\by):=f(\by)+\langle\bLmbd,\mathcal{D}(\by)\rangle-h^*(\bLmbd)\le J_0(\by).
\]
\begin{lemma}\label{lem:smooth-lower-bound}
Under Assumption~\ref{Ass: basic regularity}(ii), let $h$ be a W-type regularizer defined in \Cref{def:wtype}, and let $\kappa:=(1+m^2)/(4m)$. Then, for any $(\bx,\bB,\bLmbd)$ satisfying $r_{\mathcal{K}}<\infty$, the following hold:
\begin{enumerate}[label=\textup{(\alph*)}]
\item $s_{\bLmbd}\le J_0$ on $\RR^q$, and $s_{\bLmbd}$ is $(L_f+\|\bLmbd\|L_{\nabla\mathcal{D}})$-smooth;
\item $\|\nabla s_{\bLmbd}(\bx)\|\le\sqrt{1+G_{\mathcal{D}}^2}\;r_{\mathcal{K}}$;
\item the gap at $\bx$ satisfies
\begin{equation}\label{eq:lower-bound-gap}
h(\mathcal{D}(\bx))-h^{**}(\mathcal{D}(\bx))
\le J_0(\bx)-s_{\bLmbd}(\bx)
\le h(\mathcal{D}(\bx))+\kappa\,r_{\mathcal{K}}^2.
\end{equation}
For the $\ell_1$ W-regularizer, if $\mathcal{D}(\bx)\in[-1,1]^{n\times d}$, then
\[
\lambda Q_1(\bx)
\le J_0(\bx)-s_{\bLmbd}(\bx)
\le \lambda Q_1(\bx)+\kappa\,r_{\mathcal{K}}^2.
\]
\end{enumerate}
\end{lemma}
\begin{proof}[Proof of \Cref{lem:smooth-lower-bound}]
Write $D=\mathcal{D}(\bx)$ and choose $V\in\partial h^*(\bLmbd)$ closest to $\bB$. Such a point exists because this set is nonempty and closed when $r_{\mathcal K}<\infty$. Set
\[
a=\nabla_{\bx}F(\bx,\bB),\qquad b=m(\bB-D)+\bLmbd,\qquad c=\bB-V.
\]
Then $r_{\mathcal K}^2=\|a\|^2+\|b\|^2+\|c\|^2$ and
$\nabla s_{\bLmbd}(\bx)=a+\nabla\mathcal D(\bx)^*b$, which proves \blue{\Cref{lem:smooth-lower-bound}(b)}. The asserted Lipschitz constant follows from Assumption~\ref{Ass: basic regularity}(ii). Fenchel--Young gives $s_{\bLmbd}\le J_0$ and
$J_0(\bx)-s_{\bLmbd}(\bx)\ge h(D)-h^{**}(D)$.
Since $\bLmbd\in\partial h^{**}(V)$, Fenchel equality and $h^{**}\ge0$ yield
\[
\langle\bLmbd,D\rangle-h^*(\bLmbd)
=h^{**}(V)+\frac{\|\bLmbd\|^2}{m}
+\left\langle\bLmbd,c-\frac{b}{m}\right\rangle
\ge-\frac{\|mc-b\|^2}{4m}
\ge-\kappa r_{\mathcal K}^2.
\]
This proves \blue{the general bound}~\eqref{eq:lower-bound-gap}. \blue{For the $\ell_1$ W-regularizer and $D\in[-1,1]^{n\times d}$, \Cref{lem:wtype-conj}(a) gives $h^{**}(D)=0$, and $h(D)=\lambda Q_1(\bx)$ by \Cref{tab:wtype}, which yields the $\ell_1$ bound.}
\end{proof}

\begin{proof}[Proof of \Cref{prop:original-certificate}]
Let $s\le J_0$ have an $L$-Lipschitz gradient, where $L>0$, and suppose $J_0(\bx)-s(\bx)\le E$. Choose a global minimizer $\widehat{\bx}$ of
\[
H(\by):=J_0(\by)-\langle\nabla s(\bx),\by-\bx\rangle+L\|\by-\bx\|^2.
\]
The lower smoothness bound gives $H(\by)\ge s(\bx)+(L/2)\|\by-\bx\|^2$, so continuity of $J_0$ implies existence of this minimizer. Since $H(\widehat{\bx})\le H(\bx)=J_0(\bx)$,
\[
\|\widehat{\bx}-\bx\|\le\sqrt{2E/L}.
\]
Fermat's rule and the smooth sum rule give
$\nabla s(\bx)-2L(\widehat{\bx}-\bx)\in\partial J_0(\widehat{\bx})$, hence
\[
\operatorname{dist}(\bm0,\partial J_0(\widehat{\bx}))
\le\|\nabla s(\bx)\|+2\sqrt{2LE}.
\]
Applying this argument with $s=s_{\bLmbd}$ and $E=h(\mathcal{D}(\bx))+\kappa r_{\mathcal{K}}^2$, and using \Cref{lem:smooth-lower-bound}, we have
\[
\sqrt{\frac{2E}{L}}\le\sqrt{\frac{2}{L}}\Bigl(\sqrt{\kappa}\,r_{\mathcal{K}}+\sqrt{h(\mathcal{D}(\bx))}\Bigr),
\qquad
2\sqrt{2LE}\le 2\sqrt{2\kappa L}\,r_{\mathcal{K}}+2\sqrt{2L\,h(\mathcal{D}(\bx))},
\]
by $\sqrt{a+b}\le\sqrt{a}+\sqrt{b}$, which completes the proof.
\end{proof}

\endgroup

\section{\blue{Common Lemmas} for the Complexity Analysis}
\label{app:complexity-auxiliary-results}
This appendix collects the results shared by DualHash-StoM and DualHash-StoRM and used in \Cref{app:stom-complexity,app:storm-complexity}. \blue{For readability, we summarize the proof sketch of the complexity analysis in \Cref{fig:complexity-roadmap}.}
\begin{figure}[t]
\centering
\resizebox{\linewidth}{!}{% Proof roadmap of the complexity analysis (Appendix B). Uses \ref so numbers follow the manuscript.
\begin{tikzpicture}[
  font=\footnotesize,
  >=latex,
  box/.style={draw, rounded corners=1pt, align=left, text width=10.2cm, inner sep=4pt},
  main/.style={box},
  thm/.style={draw, rounded corners=1pt, align=center, inner sep=3pt},
  lab/.style={draw, align=left, font=\scriptsize, inner sep=2pt, fill=white}
]
% common lemmas
\node[box] (B) {%
  {\centering\textbf{Appendix~\ref{app:complexity-auxiliary-results}: Common Lemmas for the Complexity Analysis}\par}\vspace{3pt}
  $\bullet$ \Cref{Lemma: dual boundedness}: Dual increment bound\\[2pt]
  $\bullet$ \Cref{Lemma: One-iteration progress of B and Lambda}: One-iteration progress on $\bB$ and $\bLmbd$};

% StoM branch (Appendix C)
\node[box, anchor=north east, xshift=-3mm] (C) at ([yshift=-9mm]B.south) {%
  {\centering\textbf{Appendix~\ref{app:stom-complexity}: Complexity Analysis of DualHash-StoM (\Cref{alg:our_algorithm_sgdm})}\par}\vspace{3pt}
  $\bullet$ \Cref{Lemma: One-iteration progress of x in SGDM}: One-iteration progress on $\bx$\\[2pt]
  $\bullet$ Lyapunov function $\Psi_m$ in \eqref{Eq: the Lyapunov function in SGDM}\\[2pt]
  $\bullet$ \Cref{Lemma: lower bound of Psi_m}: Lower bound of $\EE[\Psi_m(\calQ^k)]$\\[2pt]
  $\bullet$ \Cref{Lemma: Stationarity error bound in SGDM}: Upper bound of $\EE[\operatorname{dist}^2(\bm{0},\mathcal{K}(\bx^k,\bB^k,\bLmbd^k))]$};
\node[main, below=13mm of C] (SM) {%
  {\centering\textbf{\Cref{Section: complexity of StoM}: Oracle Complexity of DualHash-StoM}\par}\vspace{3pt}
  $\bullet$ \Cref{Lemma: approximate sufficient descent property in SGDM}: Approximate sufficient descent of $\EE[\Psi_m(\calQ^k)]$\\[2pt]
  $\bullet$ \Cref{lemma: convergence rate sgdm}: Convergence rate of $\frac{1}{T}\sum_{k=1}^{T}\EE[\operatorname{dist}^2(\bm{0},\mathcal{K}(\bx^k,\bB^k,\bLmbd^k))]$};
\node[thm, below=13mm of SM] (T1) {\Cref{Theorem: the oracle complexity in SGDM}: $\calO(\varepsilon^{-4})$ SFO complexity};
\draw[->] (C) -- node[lab, right=1mm] {$\{\eta_k,\alpha_k,\beta_k,\tau\}$ satisfying \eqref{Eq: the parameters conditions of dualhash-sgdm in suppl}} (SM);
\draw[->] (SM) -- node[lab, right=1mm] {Batch size $b_k$ satisfying \eqref{Eq: the parameters conditions of dualhash-sgdm}} (T1);

% StoRM branch (Appendix D)
\node[box, anchor=north west, xshift=3mm] (D) at ([yshift=-9mm]B.south) {%
  {\centering\textbf{Appendix~\ref{app:storm-complexity}: Complexity Analysis of DualHash-StoRM (\Cref{alg:our_algorithm_storm})}\par}\vspace{3pt}
  $\bullet$ \Cref{Lemma: One-iteration progress of x in STORM}: One-iteration progress on $\bx$\\[2pt]
  $\bullet$ Lyapunov function $\Psi_r$ in \eqref{Eq: the Lyapunov function in STORM}\\[2pt]
  $\bullet$ \Cref{Lemma: lower bound of W}: Lower bound of $\EE[\Psi_r(\calQ^k)]+\frac{\eta_T}{\rho_T}\EE[\|\be^k\|^2]$\\[2pt]
  $\bullet$ \Cref{Lemma: Recursive bound on gradient error}: Recursive bound on $\EE[\|\be^{k+1}\|^2]$\\[2pt]
  $\bullet$ \Cref{Lemma: Stationarity error bound}: Upper bound of $\operatorname{dist}^2(\bm{0},\mathcal{K}(\bx^k,\bB^k,\bLmbd^k))$};
\node[main, below=13mm of D] (RM) {%
  {\centering\textbf{\Cref{Section: complexity of StoRM}: Oracle Complexity of DualHash-StoRM}\par}\vspace{3pt}
  $\bullet$ \Cref{Lemma: approximate sufficient descent property storm}: Approximate sufficient descent of $\EE[\Psi_r(\calQ^k)]$\\[2pt]
  $\bullet$ \Cref{Lemma: convergence rate of storm}: Convergence rate of $\frac{1}{T}\sum_{k=1}^{T}\EE[\operatorname{dist}^2(\bm{0},\mathcal{K}(\bx^k,\bB^k,\bLmbd^k))]$};
\node[thm, below=13mm of RM] (T2) {\Cref{Theorem: the oracle complexity in STORM}: $\calO(\varepsilon^{-3})$ SFO complexity};
\draw[->] (D) -- node[lab, right=1mm] {$\{\eta_k,\rho_k,\tau\}$ satisfying \eqref{Eq: the parameters conditions in dualhash-storm}} (RM);
\draw[->] (RM) -- node[lab, right=1mm] {Batch size $b_1$ satisfying \eqref{Eq: the initial batch size of dualhash-storm}} (T2);

% common lemmas feed both branches
\draw[->] (B.south) -- ++(0,-4mm) -| (C.north);
\draw[->] (B.south) -- ++(0,-4mm) -| (D.north);
\end{tikzpicture}}
\caption{\blue{This figure shows the proof roadmap of the complexity analysis, where $\operatorname{dist}(\bm{0},\mathcal{K}(\bx^k,\bB^k,\bLmbd^k))$ is the KKT residual~\eqref{Eq: kkt residual} at the $k$-th iterate $(k\ge1)$ with $\mathcal{K}$ defined in~\eqref{Eq: KKT mapping}, and $\be^k$ denotes the gradient error of each variant.}}
\label{fig:complexity-roadmap}
\end{figure}
\subsection{Smoothness Constant of $F$}
\label{app:smoothness-discussion}
\begingroup\color{blue}
Under Assumption~\ref{Ass: basic regularity}(ii), $F$ is $L_F$-smooth on $\RR^q\times\{\bB:\|\bB\|\le R_B\}$, where
\begin{equation}
\label{eq:derived-smoothness-constant}
L_F:=L_f+m\left(L_{\nabla\mathcal{D}}C_{\mathcal{D}B}+G_{\mathcal{D}}^2+1\right).
\end{equation}
Here $C_{\mathcal{D}B}:=C_{\mathcal{D}}+R_B$ and
\[
R_B:=C_{\mathcal{D}}+\frac{\lambda\sqrt{nd}}{m}.
\]
This region contains all primal iterates under the prescribed initializations and step sizes, and $L_F$ is independent of $T$.

\begin{proof}
For a W-type regularizer with $\operatorname{dom}\phi^*\subseteq[-\lambda,\lambda]$, $\bB^1\in\partial h^*(\bLmbd^1)$ and \eqref{Eq: the update of lambda} give $\|\bLmbd^k\|\le\lambda\sqrt{nd}$ for all $k\ge1$.
Since $L_F\ge m$, the prescribed step sizes satisfy $0<\tau m<1$.
Since $\phi^{**}$ is even and convex, its proximal mapping preserves sign and does not increase magnitude. Thus the algorithm initialization and Assumption~\ref{Ass: basic regularity}(ii) give $\|\bB^1\|\le\|\mathcal{D}(\bx^1)\|\le C_{\mathcal{D}}\le R_B$. If $\|\bB^k\|\le R_B$, then \eqref{Eq: the update of B} gives
\[
\|\bB^{k+1}\|\le\left(1-\tau m\right)R_B
+\tau mC_{\mathcal{D}}+\tau\lambda\sqrt{nd}=R_B,
\]
where the equality follows from the definition of $R_B$. By induction, $\|\bB^k\|\le R_B$ for all $k\ge1$.
Thus $\|\mathcal{D}(\bx)-\bB\|\le C_{\mathcal{D}B}$ throughout the stated region.

For two points $(\bx,\bB)$ and $(\bx',\bB')$ in this region, the gradient of $P(\bx,\bB)=\frac{m}{2}\|\mathcal{D}(\bx)-\bB\|^2$ satisfies
\begin{align*}
\|\nabla_{\bx}P(\bx,\bB)-\nabla_{\bx}P(\bx',\bB')\|
&\le m\left[\left(L_{\nabla\mathcal{D}}C_{\mathcal{D}B}+G_{\mathcal{D}}^2\right)\|\bx-\bx'\|+G_{\mathcal{D}}\|\bB-\bB'\|\right],\\
\|\nabla_{\bB}P(\bx,\bB)-\nabla_{\bB}P(\bx',\bB')\|
&\le m\left[G_{\mathcal{D}}\|\bx-\bx'\|+\|\bB-\bB'\|\right],
\end{align*}
by Assumption~\ref{Ass: basic regularity}(ii).
The coefficient matrix of these two bounds is positive semidefinite, so its spectral norm is at most its trace.
Hence $P$ is jointly smooth with constant $m(L_{\nabla\mathcal{D}}C_{\mathcal{D}B}+G_{\mathcal{D}}^2+1)$. Since $f$ is $L_f$-smooth, $F=f+P$ is $L_F$-smooth.
\end{proof}

\endgroup

\subsection{Dual Increment Bound}
\label{app:common-update-estimates}
\begin{lemma}[Dual increment bound]
\label{app:dual-difference-proof}
\label{Lemma: dual boundedness}
Under \blue{Assumption~\ref{Ass: basic regularity}(ii)}, for all $k \geq 1$,
    \begin{align}
        &\|\bm{\Lambda}^{k+1}- \bm{\Lambda}^k \|^2 \notag \\
        \leq&\,3 \tau^{-2} \|\bB^{k+2} - \bB^{k+1}\|^2 + 3(\tau^{-1} + L_F)^2  \|\bB^{k+1} - \bB^{k}\|^2 + 3L_F^2  \|\bx^{k+2} - \bx^{k+1}\|^2.
         \label{Eq: dual boundedness}
\end{align}
\end{lemma}

\begin{proof}
    Applying the $\bB$-update~\eqref{Eq: the update of B} at iterations $k+1$ and $k+2$ gives
    \begin{align*}
    \bLmbd^{k+1}-\bLmbd^k
    ={}&\tau^{-1}(\bB^{k+1}-\bB^{k+2})
    +\tau^{-1}(\bB^{k+1}-\bB^k)\\
    &+\nabla_{\bB}F(\bx^{k+1},\bB^k)
    -\nabla_{\bB}F(\bx^{k+2},\bB^{k+1}).
    \end{align*}
    By the \blue{$L_F$-smoothness of $F$ under \Cref{Ass: basic regularity}(ii)},
    \begin{align*}
    \|\bLmbd^{k+1}-\bLmbd^k\|
    \leq{}&\tau^{-1}\|\bB^{k+2}-\bB^{k+1}\|
    +(\tau^{-1}+L_F)\|\bB^{k+1}-\bB^k\|\\
    &+L_F\|\bx^{k+2}-\bx^{k+1}\|.
    \end{align*}
    Squaring this inequality and using $(a+b+c)^2\leq3(a^2+b^2+c^2)$ proves
    \begin{align}
    \|\bLmbd^{k+1}-\bLmbd^k\|^2
    \leq{}&3\tau^{-2}\|\bB^{k+2}-\bB^{k+1}\|^2
    +3(\tau^{-1}+L_F)^2\|\bB^{k+1}-\bB^k\|^2\notag\\
    &+3L_F^2\|\bx^{k+2}-\bx^{k+1}\|^2,
    \label{Eq: dual boundedness in proof}
    \end{align}
    which is the claimed result.
\end{proof}

\subsection{One-Iteration Progress on \texorpdfstring{$\bB$ and $\bLmbd$}{B and Lambda}}
\label{app:common-progress}
We first decompose the KKT residual. \blue{Our complexity bounds concern $\mathbb{E}[\mathrm{dist}^{2}(\bm{0},\mathcal K(\bx^{k},\bB^{k},\bLmbd^{k}))]$.} For $\bd^{k}\in\blue{\mathcal K}(\bx^{k},\bB^{k},\bLmbd^{k})$, let us denote $\bd^{k}=(\bd_{\bx}^{k},\bd_{\bB}^{k},\bd_{\bLmbd}^{k})$ and
\begin{align}
    \label{Eq: the components of the subgradient}
    \begin{cases}&\bd_{\bx}^{k}= \nabla_{\bx}F(\bx^{k}, \bB^{k}),\\&\bd_{\bB}^{k}=\nabla_{\bB}F(\bx^{k}, \bB^{k}) + \bLmbd^{k}, \\&\bd_{\bLmbd}^{k}= \bB^{k}- \calG_{\bLmbd}^{k},\end{cases}
\end{align}
where \blue{$\calG_{\bLmbd}^{1}:=\bB^1\in\partial h^*(\bLmbd^1)$ by proximal optimality of the algorithm initialization and subgradient inversion, and $\calG_{\bLmbd}^{k}:=2\bB^{k}-\bB^{k-1}+\tau(\bLmbd^{k-1}-\bLmbd^{k})\in\partial h^{*}(\bLmbd^{k})$ for $k\ge2$ by the optimality condition of the $\bLmbd$-update \eqref{Eq: the update of lambda}, implying}
\begin{align}
    \label{Eq: the stationarity error bound in proof}
    \mathrm{dist}^{2}(\bm{0},\blue{\mathcal K}(\bx^{k},\bB^{k},\bLmbd^{k}))\le\|\bd_{\bx}^{k}\|^{2}+\|\bd_{\bB}^{k}\|^{2}+\|\bd_{\bLmbd}^{k}\|^{2}.
\end{align}
\begin{lemma}[One-iteration progress on $\bB$ and $\bLmbd$]
    \label{app:primal-dual-progress-proof}
\label{Lemma: One-iteration progress of B and Lambda}
    Consider the sequence $\{
    \bx^{k}, \bB^{k}, \bm{\Lambda}^{k}\}_{k \in \NN}$ generated by \Cref{alg:our_algorithm_sgdm}
    or \Cref{alg:our_algorithm_storm}. Under \Cref{Ass: basic regularity}(ii), \blue{for all $k \geq 1$},
    \begin{align}
\label{item: B and Lambda}
&\calL(\bx^{k+1}, \bB^{k+1}, \bLmbd^{k+1}) - \calL(\bx^{k+1}, \bB^{k}, \bLmbd^{k}) \notag \\
\le& K_1 \Delta_{\bB}^{k+2} - K_2 \Delta_{\bB}^{k+1} + K_3 \Delta_{\bB}^{k}
+ K_4 \Delta_{\bx}^{k+2} + K_3 \Delta_{\bx}^{k+1},
\end{align}
where $K_1=\blue{\frac{1}{2\tau}}$, $K_2=\blue{\frac{3}{2\tau}-2L_F-\frac{\tau L_F^2}{2}}$, $K_3=\blue{6\tau L_F^2}$, and $K_4=\blue{\frac{\tau L_F^2}{2}}$.
\end{lemma}

\begin{proof}
    \blue{By the descent lemma~\eqref{Eq: the descent lemma} for $F(\bx^{k+1},\cdot)$,
    \[
        F(\bx^{k+1},\bB^{k+1})-F(\bx^{k+1},\bB^{k})\le\langle\nabla_{\bB}F(\bx^{k+1},\bB^{k}),\bB^{k+1}-\bB^{k}\rangle+\frac{L_F}{2}\|\bB^{k+1}-\bB^{k}\|^{2}.
    \]
    The bilinear term satisfies
    \[
        \langle\bB^{k+1},\bLmbd^{k+1}\rangle-\langle\bB^{k},\bLmbd^{k}\rangle=\langle\bLmbd^{k},\bB^{k+1}-\bB^{k}\rangle+\langle\bB^{k+1},\bLmbd^{k+1}-\bLmbd^{k}\rangle,
    \]
    and the $\bB$-update gives
    \[
        \nabla_{\bB}F(\bx^{k+1},\bB^{k})+\bLmbd^{k}=-\tau^{-1}(\bB^{k+1}-\bB^{k}).
    \]
    For the conjugate term, $\calG_{\bLmbd}^{k}\in\partial h^{*}(\bLmbd^{k})$ defined after \eqref{Eq: the components of the subgradient} and the convexity of $h^{*}$ give
    \[
        -h^{*}(\bLmbd^{k+1})+h^{*}(\bLmbd^{k})\le-\langle\calG_{\bLmbd}^{k},\bLmbd^{k+1}-\bLmbd^{k}\rangle.
    \]
    Adding these three relations yields
    \begin{align*}
        &\calL(\bx^{k+1},\bB^{k+1},\bLmbd^{k+1})-\calL(\bx^{k+1},\bB^{k},\bLmbd^{k})\\
        &\le-\Bigl(\tau^{-1}-\frac{L_F}{2}\Bigr)\|\bB^{k+1}-\bB^{k}\|^{2}+\langle\bLmbd^{k+1}-\bLmbd^{k},\bB^{k+1}-\calG_{\bLmbd}^{k}\rangle.
    \end{align*}
    For $k\ge2$, the $\bB$-updates at iterations $k$ and $k-1$ and the expression of $\calG_{\bLmbd}^{k}$ give the identity below. For $k=1$, it follows from $\calG_{\bLmbd}^{1}=\bB^1$, $\bB^0=\bB^1$, and $\nabla_{\bB}F(\bx^1,\bB^1)+\bLmbd^1=\bm{0}$:
    \[
        \bB^{k+1}-\calG_{\bLmbd}^{k}=\tau\bigl(\nabla_{\bB}F(\bx^{k},\bB^{k-1})-\nabla_{\bB}F(\bx^{k+1},\bB^{k})\bigr),
    \]
    and hence}
    \begin{align}
        \label{Eq: the recursive relation of L on B and Lambda} & \calL(\bx^{k+1}, \bB^{k+1}, \bLmbd^{k+1}) - \calL(\bx^{k+1}, \bB^{k}, \bLmbd^{k}) \notag \\
        \leq & - \left(\tau^{-1}- \frac{L_{F}}{2}\right) \|\bB^{k+1}- \bB^{k}\|^{2}+ \Big\langle \bLmbd^{k+1}- \bLmbd^{k}, \tau \Big( \nabla_{\bB}F(\bx^{k}, \bB^{k-1}) - \nabla_{\bB}F(\bx^{k+1}, \bB^{k}) \Big) \Big\rangle.
    \end{align}
    For the inner product term in \eqref{Eq: the recursive relation of L on B and Lambda}, we obtain
    \begin{align}
        \label{Eq: the inner product in the recursive relation of L on B and Lambda} & \langle \bLmbd^{k+1}- \bLmbd^{k}, \tau (\nabla_{\bB}F(\bx^{k}, \bB^{k-1}) - \nabla_{\bB}F(\bx^{k+1}, \bB^{k})) \rangle \notag \\
        \leq & \, \blue{\frac{\tau}{12}}\|\bLmbd^{k+1}- \bLmbd^{k}\|^{2}+ \blue{3\tau}\|\nabla_{\bB}F(\bx^{k+1}, \bB^{k}) - \nabla_{\bB}F(\bx^{k}, \blue{\bB^{k-1}}) \|^{2} \notag \\
        \leq & \, \blue{\frac{\tau}{12}}\|\bLmbd^{k+1}- \bLmbd^{k}\|^{2}+ \blue{3\tau L_{F}^{2}}\left( \|\bx^{k+1}- \bx^{k}\|^{2}+ \|\bB^{k}- \bB^{k-1}\|^{2}\right),
    \end{align}
    where the first inequality follows from the fact that
    $\langle \bu, \bv \rangle \leq \blue{\frac{1}{12}\|\bu\|^{2}+ 3\|\bv\|^{2}}$ and the second inequality follows from the $L_F$-smoothness of $F$. Then, substituting \eqref{Eq: the inner product in the recursive relation of L on B and Lambda}
    and \eqref{Eq: dual boundedness in proof} into \eqref{Eq: the recursive relation of L on B and Lambda}
    and rearranging the inequality yields \eqref{item: B and Lambda}.
\end{proof}

Before proceeding, throughout Appendices~\ref{app:stom-complexity} and~\ref{app:storm-complexity}, we adopt the notation, \blue{the probabilistic setting} and intermediate results from \Cref{Section: Convergence analysis}. 

\subsection{Estimates for the Lower Bounds of the Potentials}
\label{app:lower-bound-estimates}
\begingroup\color{blue}
We record two estimates used in \Cref{Lemma: lower bound of Psi_m,Lemma: lower bound of W} to bound the potentials of both variants from below.
Let $t:=\tau L_F$. Both parameter settings \eqref{Eq: the parameters conditions of dualhash-sgdm in suppl} and \eqref{Eq: the parameters conditions in dualhash-storm} give $t<\frac{\sqrt{30}-2}{13}<0.27$. Hence, by \eqref{Eq: the coefficients of the Lyapunov function in SGDM},
\[
C_1=K_1=\frac{1}{2\tau},\qquad C_3=K_4=\frac{\tau L_F^2}{2}.
\]
By \blue{\Cref{app:smoothness-discussion}}, $\|\bB^k\|\le R_B$ and $\|\bLmbd^k\|\le\lambda\sqrt{nd}$ for all $k\ge1$, and
\[
0\le h^*\le nd\sup_{y\in\operatorname{dom}\phi^*}\phi^*(y)
\quad\text{on }\operatorname{dom}h^*
\]
by \Cref{lem:wtype-conj} and the condition on $\phi^*$ in \Cref{Section: W-regularizer}. Therefore $\langle\bB^k,\bLmbd^k\rangle\ge-\lambda\sqrt{nd}\,R_B$ and $\Delta_{\bB}^{k+1}=\frac12\|\bB^{k+1}-\bB^k\|^2\le2R_B^2$, so for both algorithms and every $k\ge1$,
\begin{equation}
\label{eq:lb-common}
\calL(\bx^k,\bB^k,\bLmbd^k)-C_1\Delta_{\bB}^{k+1}
\ge F(\bx^k,\bB^k)-c_1,
\qquad
c_1:=\lambda\sqrt{nd}\,R_B+nd\sup_{y\in\operatorname{dom}\phi^*}\phi^*(y)+\frac{R_B^2}{\tau}.
\end{equation}
It remains to control $C_3\Delta_{\bx}^{k+1}=\frac{\tau L_F^2}{4}\|\bx^{k+1}-\bx^k\|^2$ by $F-\underline{\ell}$. Since $f\ge\underline{\ell}$ by \Cref{Ass: basic regularity}(i) and $P\ge0$, we have $F\ge\underline{\ell}$. For $\bx\in\RR^q$ and $\|\bB\|\le R_B$, the point $(\bx,\bB)-L_F^{-1}\nabla F(\bx,\bB)$ has $\bB$-component $(1-\frac{m}{L_F})\bB+\frac{m}{L_F}\mathcal{D}(\bx)$, a convex combination of $\bB$ and $\mathcal{D}(\bx)$ since $L_F\ge m$; it therefore remains in the convex set $\RR^q\times\{\bB:\|\bB\|\le R_B\}$ because $R_B\ge C_{\mathcal{D}}$. The descent lemma on this set and $F\ge\underline{\ell}$ give
\begin{equation}
\label{eq:lb-gradient}
\|\nabla F(\bx,\bB)\|^2\le2L_F\bigl(F(\bx,\bB)-\underline{\ell}\bigr),\qquad \bx\in\RR^q,\ \|\bB\|\le R_B.
\end{equation}
\endgroup

\section{Complexity Analysis of DualHash-StoM (\texorpdfstring{\Cref{alg:our_algorithm_sgdm}}{Algorithm 1})}
\label{app:stom-complexity}
% \blue{This appendix proves the complexity results for DualHash-StoM stated in \Cref{Section: complexity of StoM}.} We first give the one-iteration progress on $\bx$ (\Cref{Lemma: One-iteration progress of x in SGDM}) and prove \Cref{Lemma: approximate sufficient descent property in SGDM}. We then establish the lower bound of $\Psi_m$ (\Cref{Lemma: lower bound of Psi_m}) and the stationarity error bound (\Cref{Lemma: Stationarity error bound in SGDM}), and prove \Cref{lemma: convergence rate sgdm} and \Cref{Theorem: the oracle complexity in SGDM}.
% We first present a lemma \blue{on} continuously differentiable \blue{functions} with Lipschitz continuous \blue{gradients. Since the analysis also requires the relations among the extrapolated iterates, we next establish the key relationships among the iterates} generated by \blue{\eqref{Eq: the update of x in SGDM} and \eqref{Eq: extrapolated points in SGDM}}.
\begin{proposition}[\blue{Proposition 4.1 in~\cite{InertialProximalAlternating}}]
    \label{Pp: relations of xyz} Let $\{\bx^{k}\}_{k \in \mathbb{N}}$,
    $\{\by^{k}\}_{k \in \mathbb{N}}$, and $\{\bz^{k}\}_{k \in \mathbb{N}}$ be sequences
    generated by \eqref{Eq: the update of x in SGDM} and \eqref{Eq: extrapolated points in SGDM}. For all $k \in \mathbb{N}$,
    we have:
    \begin{enumerate}[label=(\roman{*})]
        \item $\|\bx^{k}- \by^{k}\|^{2}= \alpha_{k}^{2}\|\bx^{k}- \bx^{k-1}\|^{2}$,

        \item $\|\bx^{k}- \bz^{k}\|^{2}= \beta_{k}^{2}\|\bx^{k}- \bx^{k-1}\|^{2}$,

        \item $\|\bx^{k+1}- \by^{k}\|^{2}\geq (1 - \alpha_{k})\|\bx^{k+1}- \bx^{k}
            \|^{2}+ \alpha_{k}(\alpha_{k}- 1)\|\bx^{k}- \bx^{k-1}\|^{2}$.
    \end{enumerate}
\end{proposition}

\begin{lemma}
    \label{Lemma: auxiliary lemma for f} Let $f(\bx): \mathbb{R}^{q}\rightarrow \mathbb{R}$
    be a continuously differentiable function with an $L$-Lipschitz continuous gradient.
    Then, \blue{for all $\bx,\, \by,\,\blue{\bg} \in \mathbb{R}^{q}$ and any $\bx^{+}\in \mathbb{R}^{q}$ defined by}
    \[
        \bx^{+}= \by - \eta \blue{\bg}, \hspace{5pt}\eta > 0,
    \]
    \begin{align}
        \label{Eq: auxiliary lemma for f}f(\bx^{+}) \leq f(\bx) + \langle \nabla f(\bx) - \blue{\bg}, \bx^{+}- \bx \rangle + \frac{L}{2}\|\bx^{+}- \bx\|^{2}+ \frac{1}{2\eta}\|\bx - \by\|^{2}- \frac{1}{2\eta}\|\bx^{+}- \by\|^{2}.
    \end{align}
\end{lemma}
\begin{proof}
    By the definition of $\bx^{+}$, we obtain that
    \[
        \bx^{+}= \argmin\limits_{\bu \in \mathbb{R}^{q}}\{ \langle \blue{\bg}, \bu
        - \by \rangle + \frac{1}{2\eta}\| \bu - \by\|^{2}\},
    \]
    and hence, in particular, by taking $\bu = \bx$, we obtain
    \begin{align*}
        0 \leq \langle \blue{\bg}, \bx - \bx^{+}\rangle + \frac{1}{2\eta}\| \bx - \by\|^{2}- \frac{1}{2\eta}\| \bx^{+}- \by\|^{2}.
    \end{align*}
    Invoking the descent lemma for a smooth function yields the result. 
\end{proof}

Let us denote the gradient error in the $k$-th iteration as
\[
\be^{k} = \blue{\bg_m^{k}} - \nabla_{\bx}F(\bz^{k}, \bB^{k}),
\]where $\blue{\bg_m^k}$ is defined in \eqref{Eq: the stochastic approximate in SGDM}. Under Assumption~\ref{Ass: main assumptions}\blue{ and the sampling rule in Section~\ref{Sec: DualHash framework}~\cite[Chapter~2]{cochran1977sampling}}, one has
\begin{align}
    \label{eq: the gradient error in SGDM}
    \mathbb{E}\!\left[\|\be^k\|^2\right]
    \leq \blue{\frac{n-b_k}{n-1}\cdot\frac{\sigma^2}{b_k}\leq} \frac{\sigma^2}{b_k}.
\end{align}
\blue{The one-iteration progress on $\bx$ is estimated as follows.}
\begin{lemma}[One-iteration progress on $\bx$]
    \label{app:stom-primal-progress-proof}
\label{Lemma: One-iteration progress of x in SGDM}
    Consider the sequence $\{\bx^k, \bB^k, \bm{\Lambda}^k\}_{k \in \NN}$ generated by \Cref{alg:our_algorithm_sgdm}. Under \Cref{Ass: basic regularity}(ii), \blue{with $\eta_k>0$ and $\alpha_k,\beta_k\in[0,1)$,} \blue{for all $k \geq 1$},
    \begin{align}
        \label{Eq: One-iteration progress of x in SGDM}
        \calL (\bx^{k+1}, \bB^{k}, \bLmbd^{k}) - \calL (\bx^{k}, \bB^{k}, \bLmbd^{k})
        \leq -K_5\Delta_{\bx}^{k+1}+K_6\Delta_{\bx}^{k}+ \frac{1}{L_{F}}\|\be^{k}\|^{2},
    \end{align}
    where \(K_5 = \frac{1 - \alpha_{k}}{\eta_{k}}- 2L_F\) and \(K_6 = \frac{\alpha_{k}}{\eta_{k}}+ 2L_F\beta_k^2\).
\end{lemma}

\begin{proof}
    Recalling the equivalent relation:
    \begin{align}
        \label{Eq: the update of x one progress in StoM}\mathcal{L}(\bx^{k+1}, \bm{B}^{k}, \bm{\Lambda}^{k}) - \mathcal{L}(\bx^{k}, \bm{B}^{k}, \bm{\Lambda}^{k}) = F(\bx^{k+1}, \bm{B}^{k}) - F(\bx^{k}, \bm{B}^{k}).
    \end{align}
    Then,\label{loc:q75} applying \eqref{Eq: auxiliary lemma for f} \blue{to $F(\cdot, \bB^{k})$, which is $L_F$-smooth, with $\bx=\bx^{k}$, $\by=\by^{k}$, $\eta=\eta_{k}$, and $\bg=\bg_m^{k}$ (so that $\bx^{+}=\bx^{k+1}$ by \eqref{Eq: the update of x in SGDM}),}
    yields
    \begin{align}
        \label{Eq: the update of theta core}F(\bx^{k+1}, \bm{B}^{k}) \leq & F(\bx^{k}, \bm{B}^{k}) + \langle \nabla_{\bx}F(\bx^{k}, \bm{B}^{k}) - \blue{\bg_m^{k}}, \bx^{k+1}- \bx^{k}\rangle \notag                              \\
        +                                                                 & \frac{L_{F}}{2}\|\bx^{k+1}- \bx^{k}\|^{2}+ \frac{1}{2\eta_{k}}\|\bx^{k}- \bm{y}^{k}\|^{2}- \frac{1}{2\eta_{k}}\|\bx^{k+1}- \bm{y}^{k}\|^{2}.
    \end{align}
    For the inner product term in \eqref{Eq: the update of theta core}, using
    the facts that
    $\langle a, b \rangle \leq \frac{s}{2}\|a\|^{2}+ \frac{1}{2s}\|b\|^{2}$ for
    any $s > 0$ and $\| a - c \|^{2}\leq 2\|a - b \|^{2}+ 2\|b -c\|^{2}$ yields
    { \begin{align}\label{Eq: Inner product for theta}&\langle \nabla_{\bx}F(\bx^{k}, \bm{B}^{k}) - \blue{\bg_m^{k}}, \bx^{k+1}- \bx^{k}\rangle \notag \\ \leq&\,\frac{s}{2}\|\bx^{k+1}- \bx^{k}\|^{2}+ \frac{1}{s}\|\nabla_{\bx}F(\bx^{k}, \bm{B}^{k}) - \nabla_{\bx}F(\bm{z}^{k}, \bm{B}^{k}) \|^{2}+ \frac{1}{s}\|\nabla_{\bx}F(\bm{z}^{k}, \bm{B}^{k}) - \blue{\bg_m^{k}}\|^{2}\notag \\ \leq&\, \frac{s}{2}\|\bx^{k+1}- \bx^{k}\|^{2}+ \frac{L^{2}_{F}}{s}\|\bx^{k}- \bm{z}^{k}\|^{2}+ \frac{1}{s}\|\nabla_{\bx}F(\bm{z}^{k}, \bm{B}^{k}) - \blue{\bg_m^{k}}\|^{2}.\end{align}}
    Setting $s = L_{F}$ in \eqref{Eq: Inner product for theta} and substituting into
    \eqref{Eq: the update of theta core}, then applying \cref{Pp: relations of xyz},
    we obtain
    \begin{align}
        F(\bx^{k+1}, \bm{B}^{k}) & \leq F(\bx^{k}, \bm{B}^{k}) - (\frac{1 - \alpha_{k}}{\eta_{k}}- 2L_{F})\frac{\|\bx^{k+1}- \bx^{k}\|^{2}}{2}\notag                                                                                           \\
                                 & +\,(\frac{\alpha_{k}}{\eta_{k}}+ 2L_{F}\beta_{k}^{2})\frac{\|\bx^{k}- \bx^{k-1}\|^{2}}{2}+ \frac{1}{L_{F}}\|\nabla_{\bx}F(\bm{z}^{k}, \bm{B}^{k}) - \blue{\bg_m^{k}}\|^{2}. \label{Eq: the update of x of L in StoM}
    \end{align}
    Substituting \eqref{Eq: the update of x of L in StoM} into \eqref{Eq: the update of x one progress in StoM}
    yields \eqref{Eq: One-iteration progress of x in SGDM} and completes the proof.
    
    % \begin{align}
    %     \label{item: x}
    %     &\mathcal{L}(\bx^{k+1}, \bm{B}^k, \bm{\Lambda}^k) -\mathcal{L}(\bx^k, \bm{B}^k, \bm{\Lambda}^k)  \notag \\
    %     \leq &\,-(\frac{1 - \alpha_k}{\eta_k}- 2L_F) \frac{ \|\bx^{k+1} - \bx^k \|^2}{2}
    %     + (\frac{\alpha_k}{\eta_k}+ 2L_F\beta_k^2) \frac{\|\bx^k - \bx^{k-1} \|^2}{2} \notag \\
    %     \quad&+ \frac{1}{L_F} \|\nabla_{\bx} F(\bm{z}^k, \bm{B}^k) - \bG^k \|^2.
    % \end{align} 
\end{proof}
\subsection{Proof of \Cref{Lemma: approximate sufficient descent property in SGDM}}
\label{app:stom-descent-proof}
\begin{proof}
    Substituting one-iteration progress \eqref{Eq: One-iteration progress of x in SGDM}
    and \eqref{item: B and Lambda} into \eqref{Eq: One-iteration progress decomposition}
    , we obtain
    \begin{align}
        \label{Eq: One-iteration progress in sgdm} & \calL (\bx^{k+1}, \bB^{k+1}, \bLmbd^{k+1}) - \calL (\bx^{k}, \bB^{k}, \bLmbd^{k}) \notag                                                                                                   \\
        \leq \,                                    & K_{1}\Delta_{\bB}^{k+2}- K_{2}\Delta_{\bB}^{k+1}+ K_{3}\Delta_{\bB}^{k}+ K_{4}\Delta_{\bx}^{k+2}- (K_{5}- K_{3})\Delta_{\bx}^{k+1}+ K_{6}\Delta_{\bx}^{k}+ \frac{1}{L_{F}}\|\be^{k}\|^{2}.
    \end{align}
    Applying the definition of $\Psi_m$ from \eqref{Eq: the Lyapunov function in SGDM}
    to \eqref{Eq: One-iteration progress in sgdm}, rearranging, and taking the
    expectation on both sides, we obtain\label{loc:stom-expectation}
    {\small
    \begin{align}
        \label{the descent property of Lyapunov func in SGDM first} & \blue{\EE}\left[\Psi_m(\mathcal{Q}^{k+1})\right] - \blue{\EE}\left[\Psi_m(\mathcal{Q}^{k})\right]\notag                                                                               \\
        \leq                                                        & - (C_{1}- K_{1}) \blue{\EE}[\Delta_{\bm{B}}^{k+2}] -(K_{2}- C_{1}- C_{2}) \blue{\EE}[\Delta_{\bm{B}}^{k+1}] - (C_{2}- K_{3}) \blue{\EE}[\Delta_{\bm{B}}^{k}] \notag                     \\
                                                                    & - (C_{3}-K_{4}) \blue{\EE}[\Delta_{\bx}^{k+2}]- (K_{5}-K_{3}-C_{3}-C_{4}) \blue{\EE}[\Delta_{\bx}^{k+1}] - (C_{4}- K_{6}) \blue{\EE}[\Delta_{\bx}^{k}] + \frac{\sigma^{2}}{b_{k}L_{F}}\blue{,}
    \end{align}}\blue{where the last term follows from \eqref{eq: the gradient error in SGDM}.}\label{loc:stom-noise} Then, \blue{by \Cref{lem:lyapunov-coef-positivity},} substituting the values of coefficients in \eqref{Eq: the coefficients of the Lyapunov function in SGDM}
    and parameter conditions \eqref{Eq: the parameters conditions of dualhash-sgdm in suppl}
    into \eqref{the descent property of Lyapunov func in SGDM first}, we derive
    \begin{align*}
             & \EE[\Psi_m(\mathcal{Q}^{k+1})] - \EE[\Psi_m(\mathcal{Q}^{k})] \nonumber                                                                                                                                                        \\
        \leq & \, -C_{\bB}\left( \EE[\Delta_{\bB}^{k+1}] + \EE[\Delta_{\bB}^{k}] \right) -C_{\bx}\left( \EE[\Delta_{\bm{x}}^{k+1}]+ \EE[\Delta_{\bm{x}}^{k}] \right) + \frac{\sigma^{2}}{b_{k}L_{F}},
    \end{align*}which is exactly \eqref{Eq: the descent coefficients of potential function in SGDM},
    thereby completing the proof. 
\end{proof}
The technical proof that the coefficients of $\Psi_m$ in \eqref{Eq: the coefficients of the Lyapunov function in SGDM} and the descent constants $C_{\bB}$, $C_{\bx}$ are indeed positive is given as follows.
\begin{lemma}[Positivity of Lyapunov related coefficients]
    \label{lem:lyapunov-coef-positivity} If parameter conditions \eqref{Eq: the parameters conditions of dualhash-sgdm in suppl}
    hold, then the following inequalities are satisfied:
    \begin{align}
        % \label{Eq: the positive coefficients}
        \begin{cases}K_{5}- K_{3}- C_{3}- C_{4}\geq \nu C_{4}, \\ C_{4}- K_{6}\geq \nu C_{4},\\ C_{4}> 0,\\ C_{\bB}> 0 ,\\ C_{2},\, C_{1}> 0, \\ C_{\bx}> 0.\end{cases} \notag
    \end{align}
\end{lemma}
\begin{proof}
    For $K_{5}- K_{3}- C_{3}- C_{4}$ and $C_{4}- K_{6}$ in \eqref{the descent property of Lyapunov func in SGDM first},
    we obtain by straightforward computations by choosing some $\nu > 0$\blue{, where we recall that $K_{5}= \frac{1 - \alpha_{k}}{\eta_{k}}- 2L_{F}$, $C_{3}= K_{4}$, and $\bar{L}_{F}= 2L_{F}+ K_{3}+ K_{4}$}:
    {\begin{align}
        K_{5}- K_{3}- C_{3}- C_{4}
        &= \frac{1 - \alpha_{k}}{\eta_{k}}- \bar{L}_{F}- (1 + \nu)C_{4}+ \nu C_{4}\notag\\
        &= (1 - \alpha_{k})\left(\frac{1}{\eta_{k}}- \frac{\bar{L}_{F}+ (1 + \nu)C_{4}}{1 - \alpha_{k}}\right) + \nu C_{4}\notag \\
        &\geq \nu C_{4}. \notag % \label{Eq: the scaling of the delta1}
    \end{align}}
    where the last inequality follows from $\alpha_{k}= \alpha$ and
    $\frac{1}{\eta_{k}}= \frac{\bar{L}_{F}+ (1 + \nu)C_{4}}{1 - \alpha}$. Then, \blue{using $\alpha_{k}= \alpha$ and $\beta_{k}= \beta$ from \eqref{Eq: the parameters conditions of dualhash-sgdm in suppl},}
    one has
    \begin{align}
        C_{4}- K_{6} & = C_{4}- \frac{\alpha}{\eta_{k}}- 2L_{F}\beta_{k}^{2}\notag                                                                                                                                   \\
                     & = \frac{1}{1 - \alpha}\left( (1 - \alpha)C_{4}- \alpha\bar{L}_{F}- (1 + \nu)C_{4}\alpha - 2L_{F}\beta^{2}(1 - \alpha) \right)\notag                                                           \\
                     & \blue{=} \frac{1}{1 - \alpha}\left( (1 - 2\alpha - \nu)C_{4}- \alpha \left(\bar{L}_{F}+ \frac{2L_{F}\beta^{2}}{\alpha}\right) \right) + \nu C_{4}+ \frac{2L_{F}\beta^{2}\alpha}{1 - \alpha} \geq \nu C_{4},\notag % \label{Eq: the scaling of the delta2}
    \end{align}where the last inequality follows from $\alpha \in (0, \frac{1 - \nu}{2})$
    and
    $C_{4}= \frac{\bar{L}_{F}+ \frac{2L_F \beta^2}{\alpha}}{1 - 2\alpha - \nu}> 0$.
    We rearrange $\frac{1}{\eta_{k}}$ using the definition of $C_{4}$ and obtain
    $\eta_{k}= \frac{\eta}{L_{F}}$ where $\eta = \frac{(1 - 2\alpha - \nu)/2}{\frac{\bar{L}_{F}}{L_{F}}
    + \frac{(1 + \nu)\beta^{2}}{(1 - \alpha)\alpha}}$.

    Next, we verify that the coefficients of $\Psi_m$ satisfying \eqref{Eq: the coefficients of the Lyapunov function in SGDM}
    and $\tau$ satisfying \eqref{Eq: the parameters conditions of dualhash-sgdm in suppl}
    ensure that $C_{\bB},\, C_{\bx}> 0$. Let $t:=\tau L_F$. \blue{By the definitions of $K_1,K_2,K_3$,}
    \[
        \blue{C_{\bB}=\frac{K_2-K_3-K_1}{2}=\frac{L_F}{2t}\left(1-2t-\frac{13}{2}t^2\right).}
    \]
    \blue{The quadratic $1 - 2t - \frac{13}{2}t^{2}$ is positive if and only if $0 < t < \frac{\sqrt{30}-2}{13}$. Hence selecting
    $\tau \in (0, \frac{\sqrt{30}-2}{13L_{F}})$ ensures $C_{\bB}> 0$. Consequently, $C_1=K_1=\frac{1}{2\tau}>0$ and $C_2=\frac{1}{2}(K_2-K_1+K_3)=C_{\bB}+K_3>0$, since $K_3=6\tau L_F^2>0$.}
    \blue{Finally, $C_3=K_4>0$ and $C_{\bx}=\nu C_4>0$. Thus all the stated coefficients are positive under \eqref{Eq: the parameters conditions of dualhash-sgdm in suppl}.}
\end{proof}

\begingroup\color{blue}
\blue{Unlike $F$, the potential $\Psi_m$ contains the coupling term $\langle\bB^k,\bLmbd^k\rangle$ and the negative increments $-C_1\Delta_{\bB}^{k+1}$ and $-C_3\Delta_{\bx}^{k+1}$, so it is not trivially bounded below. The coupling term is controlled by the bounded auxiliary and dual iterates, and under \eqref{Eq: the parameters conditions of dualhash-sgdm in suppl} the negative increments are dominated by $F-\underline{\ell}$.}
Using the estimates in \Cref{app:lower-bound-estimates}, we obtain the following lower bound.
\endgroup

\begingroup\color{blue}
\begin{lemma}[Lower bound of $\Psi_m$]
\label{Lemma: lower bound of Psi_m}
Under Assumptions~\ref{Ass: basic regularity} and~\ref{Ass: main assumptions} and the parameter conditions \eqref{Eq: the parameters conditions of dualhash-sgdm in suppl}, for every $k\ge1$,
\[
\EE[\Psi_m(\calQ^k)]\ge\Psi_m^*:=\underline{\ell}-c_1-\frac34\tau\eta^2\sigma^2,
\]
where \blue{$\underline{\ell}$ is given in Assumption~\ref{Ass: basic regularity}(i),} $c_1$ is defined in \eqref{eq:lb-common}, and $\Psi_m^*$ is independent of $T$.
\end{lemma}

\begin{proof}
By \eqref{Eq: the update of x in SGDM} and the definition of $\be^k$,
\[
\bx^{k+1}-\bx^k=\alpha(\bx^k-\bx^{k-1})-\frac{\eta}{L_F}\bigl(\nabla_{\bx}F(\bz^k,\bB^k)+\be^k\bigr).
\]
Since $\|\bB^k\|\le R_B$, the $L_F$-smoothness of $F$ and \eqref{eq:lb-gradient} give
\[
\|\nabla_{\bx}F(\bz^k,\bB^k)\|^2\le2\|\nabla_{\bx}F(\bx^k,\bB^k)\|^2+4L_F^2\beta^2\Delta_{\bx}^k
\le4L_F\bigl(F(\bx^k,\bB^k)-\underline{\ell}\bigr)+4L_F^2\beta^2\Delta_{\bx}^k .
\]
Combining these two displays with $\EE[\|\be^k\|^2]\le\sigma^2$ from \eqref{eq: the gradient error in SGDM} yields
\[
C_3\EE[\Delta_{\bx}^{k+1}]
\le3t\eta^2\,\EE[F(\bx^k,\bB^k)-\underline{\ell}]
+\frac{tL_F}{2}\left(3\alpha^2+6\eta^2\beta^2\right)\EE[\Delta_{\bx}^k]
+\frac34\tau\eta^2\sigma^2 .
\]
Under \eqref{Eq: the parameters conditions of dualhash-sgdm in suppl}, $\bar L_F\ge2L_F$ gives $\eta<\frac14$, and $0<1-2\alpha-\nu<1$ gives $C_4\ge\bar L_F\ge2L_F$. With $\alpha<\frac12$ and $\beta<1$,
\[
3t\eta^2<1,
\qquad
\frac{tL_F}{2}\left(3\alpha^2+6\eta^2\beta^2\right)<\frac{9tL_F}{16}<L_F\le\frac{C_4}{2}.
\]
Hence the term $C_4\Delta_{\bx}^k$ dominates the $\Delta_{\bx}^k$-term. Using \eqref{eq:lb-common} and $C_2\Delta_{\bB}^k\ge0$,
\[
\EE[\Psi_m(\calQ^k)]
\ge\underline{\ell}+\left(1-3t\eta^2\right)\EE[F(\bx^k,\bB^k)-\underline{\ell}]-c_1-\frac34\tau\eta^2\sigma^2
\ge\Psi_m^* .
\]
\end{proof}
\endgroup

\subsection{Proof of \Cref{lemma: convergence rate sgdm}}
\label{app:stom-rate-proof}
We then bound the three terms on the right-hand side of \eqref{Eq: the stationarity error bound in proof}.

\begin{lemma}[Stationarity error bound]
    \label{Lemma: Stationarity error bound in SGDM} Under Assumptions~\ref{Ass: basic regularity} and~\ref{Ass: main assumptions},
    \blue{for all $k \geq 1$},
    \begin{align}\label{Eq: the stationarity error bound in SGDM Lemma}
        \EE [\mathrm{dist}^{2}(\blue{\bm{0}, \mathcal K}(\bx^{k}, \bB^{k}, \bm{\Lambda}^{k}))] \leq S_{1}\EE[ \Delta_{\bB}^{k+1}] + S_{2}\EE[ \Delta_{\bB}^{k}] + S_{3}^{k}\EE[ \Delta_{\bm{x}}^{k+1}] + S_{4}^{k}\EE [\Delta_{\bm{x}}^{k}] + \blue{\frac{3\sigma^{2}}{b_{k}}},
    \end{align}
    where the coefficients are given by
    \begin{align}
        \label{Eq: the coefficients of stationarity error bound in SGDM}\begin{cases}S_{1}&= \blue{4(\tau^{-2}+ 3)},\\ S_{2}&= 4(1 + 3 (1 + \tau L_{F})^{2}),\\ S_{3}^{k}&= 4(3\eta_{k}^{-2}+ L_{F}^{2}+ 3 (\tau L_{F})^{2}),\\ S_{4}^{k}&= 6 ( L^2_{F}\beta_{k}^{2}+ \blue{2} \eta_{k}^{-2}\alpha_{k}^{2}).\end{cases}
    \end{align}
\end{lemma}
\begin{proof}
    For estimating $\| \bd_{\bx}^{k}\|^{2}$, using the optimality condition of
    $\bx^{k}$, we obtain
    {\small
    \begin{align}
             & \|\bd_{\bx}^{k}\|^{2}\notag                                                                                                                                                                                                                   \\
        =    & \,\|\nabla_{\bx}F(\bx^{k}, \bm{B}^{k}) \|^{2}\notag                                                                                                                                                                                           \\
        =    & \,\eta_{k}^{-2}\left\| \eta_{k}\nabla_{\bx}F(\bx^{k}, \bm{B}^{k}) + 0\right\|^{2}\notag                                                                                                                                                       \\
        =    & \, \eta_{k}^{-2}\left\|\eta_{k}\left( \nabla_{\bx}F(\bx^{k}, \bm{B}^{k}) - \blue{\bg_m^{k}}\right) + \bm{y}^{k}- \bx^{k+1}\right\|^{2}\notag                                                                                                           \\
        =    & \,\eta_{k}^{-2}\bigg\|\eta_{k}\Big( \nabla_{\bx}F(\bx^{k}, \bm{B}^{k}) - \nabla_{\bx}F(\bm{z}^{k}, \bm{B}^{k}) + \nabla_{\bx}F(\bm{z}^{k}, \bm{B}^{k})- \blue{\bg_m^{k}}\Big) + \bm{y}^{k}- \bx^{k}+ \bx^{k}- \bx^{k+1}\bigg\|^{2}\label{eq:d_x_ineq1} \\
        \leq & \, 3 \|\nabla_{\bx}F(\bx^{k}, \bm{B}^{k}) - \nabla_{\bx}F(\bm{z}^{k}, \bm{B}^{k})\|^{2}+ 3\eta_{k}^{-2}\|\bx^{k+1}- \bx^{k}+ \bx^{k}- \bm{y}^{k}\|^{2}+ 3 \|\be^{k}\|^{2}\label{eq:d_x_ineq2}                                                 \\
        \leq & \, 3L_{F}^{2}\|\bx^{k}- \bm{z}^{k}\|^{2}+ 6\eta_{k}^{-2}\left(\|\bx^{k+1}- \bx^{k}\|^{2}+ \|\bx^{k}- \bm{y}^{k}\|^{2}\right) + 3 \|\be^{k}\|^{2},\label{eq:d_x_ineq3}                                                                      
    \end{align}}where \eqref{eq:d_x_ineq2} is deduced by using the fact that
    $(a + b+ c)^{2}\leq 3(a^{2}+ b^{2}+c^{2})$ and \eqref{eq:d_x_ineq3} is deduced
    \blue{from the $L_F$-smoothness of $F(\cdot, \bB^{k})$} and $(a + b)^{2}\leq 2(a^{2}+ b^{2})$.
    Then applying \Cref{Pp: relations of xyz} to \eqref{eq:d_x_ineq3} concludes
    \begin{align}
        \label{Eq: dx}\| \bd_{\bx}^{k}\|^{2}\leq 6( L^2_{F}\beta_{k}^{2}+ 2 \eta_{k}^{-2}\alpha_{k}^{2}) \Delta_{\bx}^{k}+ 12 \eta_{k}^{-2}\Delta_{\bx}^{k+1}+ 3 \| \be^{k}\|^{2}.
    \end{align}
    For estimating $\|\bd^{k}_{\bB}\|^{2}$, using the $\bB$-update
    \eqref{Eq: the update of B}, \blue{the $L_F$-smoothness of $F$, and the inequality $(a + b)^{2}\leq 2(a^{2}+ b^{2})$}, we obtain
    \begin{align}
        \|\bd_{\bB}^{k}\|^{2} & = \| \nabla_{\bB}F(\bx^{k+1}, \bB^{k})+ \bm{\Lambda}^{k}+ \nabla_{\bB}F(\bx^{k}, \bB^{k}) - \nabla_{\bB}F(\bx^{k+1}, \bB^{k}) \|^{2}\notag \\
                              & \leq4 \tau^{-2}\Delta^{k+1}_{\bB}+ 4L_{F}^{2}\Delta_{\bx}^{k+1}.\label{Eq: dB}
    \end{align}

    \blue{For $k=1$, the initialization gives $\bd_{\bLmbd}^{1}=\bm{0}$, so \eqref{Eq: dLambda} holds. For $k\ge2$, using the $\bm{\Lambda}$-update \eqref{Eq: the update of lambda}, we obtain}
\begin{align}
    \|\bd^{k}_{\bm{\Lambda}}\|^{2} & \leq 2\tau^{2}\|\bm{\Lambda}^{k}- \bm{\Lambda}^{k-1}\|^{2}+ 2\|\bB^{k}- \bB^{k-1}\|^{2}. \notag
\end{align}

    Then, \blue{bounding $\|\bm{\Lambda}^{k}- \bm{\Lambda}^{k-1}\|^{2}$ by \Cref{Lemma: dual boundedness} (with $k$ replaced by $k-1$)}, we obtain
    \begin{align}
        \label{Eq: dLambda}\|\bd^{k}_{\bm{\Lambda}}\|^{2}\leq 12 \Delta_{\bB}^{k+1}+ 4(1 + 3 (1 + \tau L_{F})^{2})\Delta_{\bB}^{k}+ 12(\tau L_{F})^{2}\Delta_{\bx}^{k+1}.
    \end{align}

    Finally, \blue{we substitute} \eqref{Eq: dx}, \eqref{Eq: dB}, and \eqref{Eq: dLambda}
    into \eqref{Eq: the stationarity error bound in proof} \blue{and take} the full expectation
    on both sides, \blue{which leads to} \eqref{Eq: the stationarity error bound in SGDM Lemma},
    completing the proof. 
\end{proof}

% We now establish the convergence rate of \Cref{alg:our_algorithm_sgdm} as follows.
% We set $T$ as the maximum number of iterations and denote 
%     $L_{\Delta}= \frac{\max\{S_{1}, S_{2}, S_{3}, S_{4}\}}{\min\{C_{\bB},
%     C_{\bx}\}}$, $C_{\sigma}= \frac{L_{\Delta}}{L_{F}}+ 3$, and
%     $\Delta_{1}= \Psi_m^{1}- \Psi_m(\mathcal{Q}^{*})$
%     are constants independent of $T$, where \(\Psi_m^{1} = \Psi_m(\calQ^1)\).
% \begin{lemma}[Convergence rate of \Cref{alg:our_algorithm_sgdm}]
%     \label{lemma: convergence rate sgdm} Under the conditions of Lemma \ref{Lemma:
%     approximate sufficient descent property in SGDM}, for any $T \geq 1$, there exists
%     $R$ uniformly selected from $\{2,\ldots,T+1\}$ such that $(\bx^{R}, \bB
%     ^{R}, \bLmbd^{R})$ satisfies \eqref{equ:esp} with
%     \begin{align}
%         \label{Eq: the convergence rate of sgdm}\mathbb{E}[\mathrm{dist}^{2}(0, \partial \mathcal{L}(\bx^{R}, \bB^{R}, \bLmbd^{R}))] \leq \frac{L_{\Delta}\Delta_{1}}{T}+ C_{\sigma}\frac{\sigma^{2}}{b_{k}}.
%     \end{align}
% \end{lemma}
\begin{proof}
    \blue{Under \eqref{Eq: the parameters conditions of dualhash-sgdm in suppl}, $\eta_k=\eta/L_F$, $\alpha_k=\alpha$, and $\beta_k=\beta$ are constant, so $S_3^k=S_3$ and $S_4^k=S_4$. Since the four increment terms are nonnegative, \eqref{Eq: the stationarity error bound in SGDM Lemma} gives}
    {\small
    \begin{align}\label{Eq: the stationarity error bound in SGDM rearrange}
        \mathbb{E}[\mathrm{dist}^{2}(\blue{\bm{0}, \mathcal K}(\bx^{k}, \bB^{k}, \bLmbd^{k}))] \leq \blue{\max\{S_{1}, S_{2}, S_{3}, S_{4}\}} \EE [\Delta_{\bB}^{k+1}+ \Delta_{\bB}^{k}+ \Delta_{\bx}^{k+1}+ \Delta_{\bx}^{k}] + \blue{\frac{3 \sigma^{2}}{b_{k}}}.\end{align}}
    Similarly, from \eqref{Eq: the descent property of Lyapunov func in SGDM}, we
    obtain
    \[
        \EE [\Psi_m(\mathcal{Q}^{k+1})] - \EE[\Psi_m(\mathcal{Q}^{k})]
        \leq -\min \{ C_{\bB}, C_{\bx}\}\EE[\Delta_{\bB}^{k+1}+ \Delta_{\bB}^{k}+
        \Delta_{\bm{x}}^{k+1}+ \Delta_{\bm{x}}^{k}] + \frac{\sigma^{2}}{b_{k}L_{F}},
    \]
    which yields that
    \[
        \EE[\Delta_{\bB}^{k+1}+ \Delta_{\bB}^{k}+ \Delta_{\bm{x}}^{k+1}+ \Delta_{\bm{x}}
        ^{k}] \leq \frac{1}{\min \{ C_{\bB}, C_{\bx}\}}\left( \EE [\Psi_m(\mathcal{Q}
        ^{k})] - \EE[\Psi_m(\mathcal{Q}^{k+1})] + \frac{\sigma^{2}}{b_{k}L_{F}}
        \right).
    \]
    Substituting this inequality into \eqref{Eq: the stationarity error bound in SGDM rearrange},
    summing over $k = 1, \cdots,T$, and dividing by $T$, we obtain
    \begin{align*}
        &\frac{1}{T}\sumT \mathbb{E}[\mathrm{dist}^{2}(\blue{\bm{0}, \mathcal K}(\bx
        ^{k}, \bB^{k}, \bLmbd^{k}))]\\
        &\leq \frac{L_{\Delta}}{T}\sumT
        \left(\EE [\Psi_m(\mathcal{Q}^{k})] - \EE [\Psi_m
        (\mathcal{Q}^{k+1})] \right) + \blue{\left( 3 + \frac{L_{\Delta}}{L_F}\right) \frac{\sigma^{2}}{T}\sumT\frac{1}{b_{k}}}.
    \end{align*}\blue{Since $\bx^0=\bx^1$ and $\bB^0=\bB^1$, the negative increment terms give $\Psi_m(\calQ^1)\le\calL(\bx^1,\bB^1,\bLmbd^1)$, so $\EE[\Psi_m(\calQ^1)]-\Psi_m^*\le\Delta_1$.} Denoting
    $a_{k}= \mathbb{E}[\mathrm{dist}^{2}(\blue{\bm{0}, \mathcal K}(\bx^{k}, \bB
    ^{k}, \bLmbd^{k}))]$, we have
    \[
        \frac{1}{T}\sumT a_{k}\leq \frac{L_{\Delta}\Delta_{1}}{T}+ \blue{\left( 3 + \frac{L_{\Delta}}{L_F}
        \right) \frac{\sigma^{2}}{T}\sumT\frac{1}{b_{k}}},
    \]
    which completes the proof.
\end{proof}
% From \eqref{Eq: the convergence rate of sgdm}, to achieve $\calO(T^{-1})$
% convergence rate, we need to set $b_{k}= \Theta (T)$. Next, we set
% $b_{k}= b = c_{b}T$ for all $k \geq 1$ in this way and establish the complexity
% results of \Cref{alg:our_algorithm_sgdm} with the parameter conditions \eqref{Eq: the parameters conditions of dualhash-sgdm in suppl}.
\subsection{Proof of \Cref{Theorem: the oracle complexity in SGDM}}
\label{app:stom-oracle-proof}
\begin{proof}
    When $b_{k}= b = \lceil c_{b}T\rceil$, \blue{we have $\frac{1}{T}\sumT\frac{1}{b_{k}}\leq\frac{1}{c_{b}T}$. Since $R$ is drawn uniformly from $\{1,\ldots,T\}$, independently of the mini-batches, the law of total expectation and \Cref{lemma: convergence rate sgdm} give}
    \[
        \mathbb{E}[\mathrm{dist}^{2}(\blue{\bm{0}, \mathcal K}(\bx^{R}, \bB^{R}
        , \bLmbd^{R}))] \leq \frac{1}{T}\left( L_{\Delta}\Delta_{1}+ C_{\sigma}
        \frac{\sigma^{2}}{c_{b}}\right).
    \]
    With the conditions on $T$ in \eqref{Eq: the iterate number upper bound of algorithm SGDM},
    this gives
    $\mathbb{E}[\mathrm{dist}^{2}(\blue{\bm{0}, \mathcal K}(\bx^{R}, \bB^{R}
    , \bLmbd^{R}))] \leq \varepsilon^{2}$. Then, the oracle complexity
    of SFO is
    \[
        \text{Complexity}= \blue{Tb}=\calO(T^{2}) = \calO(\varepsilon^{-4}).
    \]
\end{proof}

\section{Complexity Analysis of DualHash-StoRM (\texorpdfstring{\Cref{alg:our_algorithm_storm}}{Algorithm 2})}
\label{app:storm-complexity}
% This appendix proves the complexity results for DualHash-StoRM stated in \Cref{Section: complexity of StoRM}. We adopt the same notation as in \Cref{app:stom-complexity}. We first give the one-iteration progress on $\bx$ (\Cref{Lemma: One-iteration progress of x in STORM}) and prove \Cref{Lemma: approximate sufficient descent property storm}. We then establish the auxiliary results in \Cref{lem:auxiliary-tracking,Lemma: lower bound of W,Lemma: Recursive bound on gradient error,Lemma: Weighted error accumulation,Lemma: Stationarity error bound}, and prove \Cref{Lemma: convergence rate of storm} and \Cref{Theorem: the oracle complexity in STORM}.
% Similarly, the following lemma estimates the one-iteration progress
% of the Lagrangian function $\calL$ via the $\bx$-update \eqref{Eq: the update of x in STORM}. \blue{Its proof follows from~\cite[Lemma~2.1]{xu2023momentum}.}
\begin{lemma}[One-iteration progress on $\bx$~\cite{xu2023momentum}]
    \label{app:storm-primal-progress-proof}
\label{Lemma: One-iteration progress of x in STORM} Consider the sequence $\{\bx^k, \bB^k, \bm{\Lambda}^k\}_{k \in \NN}$ generated by \Cref{alg:our_algorithm_storm}. \blue{Under \Cref{Ass: basic regularity}(ii),} \blue{for all $k \geq 1$},
    \begin{align}
        \label{Eq: the recursive relation of L on x}\calL (\bx^{k+1}, \bB^{k}, \bLmbd^{k}) - \calL (\bx^{k}, \bB^{k}, \bLmbd^{k}) \leq \frac{\eta_{k}}{2}\| \be^{k}\|^{2}- (\eta_{k}^{-1}- L_{F})\Delta_{\bx}^{k+1}.
    \end{align}
\end{lemma}

\begin{proof}
    \blue{Recalling the equivalent relation:
    \[
        \calL(\bx^{k+1},\bB^{k},\bLmbd^{k})-\calL(\bx^{k},\bB^{k},\bLmbd^{k})=F(\bx^{k+1},\bB^{k})-F(\bx^{k},\bB^{k}).
    \]
    By the $L_F$-smoothness of $F(\cdot,\bB^{k})$ and the $\bx$-update \eqref{Eq: the update of x in STORM}, we have
    \begin{align*}
        F(\bx^{k+1},\bB^{k})-F(\bx^{k},\bB^{k})&\leq\langle\nabla_{\bx}F(\bx^{k},\bB^{k}),\bx^{k+1}-\bx^{k}\rangle+\frac{L_F}{2}\|\bx^{k+1}-\bx^{k}\|^{2}\\
        &=-\eta_k\langle\nabla_{\bx}F(\bx^{k},\bB^{k}),\bg_s^{k}\rangle+\frac{\eta_k^{2}L_F}{2}\|\bg_s^{k}\|^{2}.
    \end{align*}
    Using the definition of $\be^{k}$, we have $-\langle\nabla_{\bx}F(\bx^{k},\bB^{k}),\bg_s^{k}\rangle=\langle\be^{k},\bg_s^{k}\rangle-\|\bg_s^{k}\|^{2}$. By the Cauchy--Schwarz inequality, it holds $\eta_k\langle\be^{k},\bg_s^{k}\rangle\leq\frac{\eta_k}{2}\|\be^{k}\|^{2}+\frac{\eta_k}{2}\|\bg_s^{k}\|^{2}$, which together with the above inequality implies
    \begin{align*}
        F(\bx^{k+1},\bB^{k})-F(\bx^{k},\bB^{k})&\leq\frac{\eta_k}{2}\|\be^{k}\|^{2}-\frac{\eta_k}{2}(1-\eta_kL_F)\|\bg_s^{k}\|^{2}\\
        &=\frac{\eta_k}{2}\|\be^{k}\|^{2}-(\eta_k^{-1}-L_F)\Delta_{\bx}^{k+1},
    \end{align*}
    where the equality uses $\|\bx^{k+1}-\bx^{k}\|^{2}=\eta_k^{2}\|\bg_s^{k}\|^{2}=2\Delta_{\bx}^{k+1}$.}
\end{proof}

\subsection{Proof of \Cref{Lemma: approximate sufficient descent property storm}}
\label{app:storm-descent-proof}
\begin{proof}
    Substituting the one-iteration progress \eqref{Eq: the recursive relation of L on x}
    and \eqref{item: B and Lambda} into \eqref{Eq: One-iteration progress decomposition}, we obtain
    \begin{align}
             & \calL (\bx^{k+1}, \bB^{k+1}, \bLmbd^{k+1}) - \calL (\bx^{k}, \bB^{k}, \bLmbd^{k}) \nonumber                                                                                                                                          \\
        \leq & \,\frac{\eta_{k}}{2}\| \be^{k}\|^{2}- (\eta_{k}^{-1}- L_{F}- K_{3}) \Delta_{\bx}^{k+1}+ K_{1}\Delta_{\bB}^{k+2}+ K_{4}\Delta_{\bx}^{k+2}- K_{2}\Delta_{\bB}^{k+1}+ K_{3}\Delta_{\bB}^{k}.\label{Eq: One-iteration progress in storm}
    \end{align}
    Applying the definition of $\Psi_r$ from \eqref{Eq: the Lyapunov function in STORM} into \eqref{Eq: One-iteration progress in storm}, we rearrange the above inequality
    \begin{align}
        \label{the descent property of Lyapunov func in STORM first} & \Psi_r(\mathcal{Q}^{k+1}) - \Psi_r(\mathcal{Q}^{k}) \notag                                                                                         \\
        \leq                                                         & - (C_{1}- K_{1}) \Delta_{\bm{B}}^{k+2}-(K_{2}- C_{1}- C_{2}) \Delta_{\bm{B}}^{k+1}- (C_{2}- K_{3}) \Delta_{\bm{B}}^{k}- (C_{3}-K_{4}) \Delta_{\bx}^{k+2}\notag \\
                                                                     & -(\eta_{k}^{-1}- L_{F}- K_{3}- C_{3}) \Delta_{\bx}^{k+1}+ \frac{\eta_{k}}{2}\| \be^{k}\|^{2}.
    \end{align}
    By substituting the Lyapunov function coefficients \eqref{Eq: the coefficients of the Lyapunov function in SGDM} into \eqref{the descent property of Lyapunov func in STORM first} \blue{and taking full expectation}, we obtain \eqref{Eq: the descent property of Lyapunov func in storm}. With $\tau$ satisfying \blue{$0<\tau<\frac{\sqrt{30}-2}{13L_F}$}, \Cref{lem:lyapunov-coef-positivity} guarantees that $C_{\bB}>0$.
\end{proof}

\begin{lemma}[\blue{Auxiliary-block tracking}]
\label{app:auxiliary-tracking-proof}
\label{lem:auxiliary-tracking}
\blue{Under \Cref{Ass: basic regularity}(ii), for the iterates generated by
\Cref{alg:our_algorithm_storm} with parameters satisfying
\eqref{Eq: the parameters conditions in dualhash-storm},
let $\tilde\tau:=\tau m$ and}
\blue{\begin{align*}
H_k:={}&(1-\tilde\tau)\bigl\|\tau(\bLmbd^{k}-\bLmbd^{k-1})-(\bB^{k}-\bB^{k-1})\bigr\|^{2}\\
&\quad+\tilde\tau\tau^{2}\bigl\|\bLmbd^{k}-\bLmbd^{k-1}\bigr\|^{2},
\end{align*}}
\blue{with $\bB^{0}=\bB^{1}$ and $\bLmbd^{0}=\bLmbd^{1}$. Then $H_1=0$, $H_k\geq0$, and, for every $k\geq1$,}
\blue{\begin{align}
H_{k+1}-H_k
&\leq-\frac{\tilde\tau}{8}\Bigl(\|\bB^{k}-\bB^{k-1}\|^{2}
+\|\tau(\bLmbd^{k}-\bLmbd^{k-1})-(\bB^{k}-\bB^{k-1})\|^{2}\Bigr)\notag\\
&\qquad+\blue{146\tilde\tau G_{\mathcal{D}}^{2}\Delta_{\bx}^{k+1}},
\label{Eq: auxiliary tracking descent}\\
\Delta_{\bB}^{k+1}
&\leq\blue{\frac{12}{\tilde\tau}(H_k-H_{k+1})+1753G_{\mathcal{D}}^{2}\Delta_{\bx}^{k+1}}.
\label{Eq: tracking of B increments}
\end{align}}
\end{lemma}

\begingroup\color{blue}
\begin{proof}
    The parameter condition gives $0<\tilde\tau=\tau m<2/7$. Set $\bm{V}^k:=\tau(\bLmbd^k-\bLmbd^{k-1})-(\bB^k-\bB^{k-1})$. The initialization gives $\bB^1-\bB^0=\bm{V}^1=\bm{0}$, hence $H_1=0$, and $H_k=\sum_{i,j}(\bm{y}_{ij}^k)^\top\bm{\Pi}\bm{y}_{ij}^k\ge0$ with $\bm{y}_{ij}^k:=(b_{ij}^k-b_{ij}^{k-1},V_{ij}^k)^\top$ and $\bm{\Pi}:=\bigl(\begin{smallmatrix}\tilde\tau&\tilde\tau\\ \tilde\tau&1\end{smallmatrix}\bigr)\succ0$.

    By \eqref{Eq: the update of B}, we have $\bB^{k+1}=(1-\tilde\tau)\bB^k-\tau\bLmbd^k+\tilde\tau\mathcal{D}(\bx^{k+1})$, so $\bB^{k+1}-\bB^k=-\tilde\tau(\bB^k-\bB^{k-1})-\bm{V}^k+\tilde\tau(\mathcal{D}(\bx^{k+1})-\mathcal{D}(\bx^k))$. Since $h^*$ is separable and convex, $\operatorname{prox}_{\tau^{-1}h^*}$ acts componentwise as a nondecreasing $1$-Lipschitz map; by \eqref{Eq: the update of lambda}, for each $(i,j)$ there is $\omega=\omega_{ij}^k\in[0,1]$ with $\tau(\Lambda_{ij}^{k+1}-\Lambda_{ij}^k)=\omega(V_{ij}^k+2(b_{ij}^{k+1}-b_{ij}^k))$. Combining the two relations yields
    \begin{equation}
        \label{Eq: increment dynamics}
        \bm{y}_{ij}^{k+1}=\bm{M}_\omega\bm{y}_{ij}^k+\tilde\tau\begin{pmatrix}1\\2\omega-1\end{pmatrix}[\mathcal{D}(\bx^{k+1})-\mathcal{D}(\bx^k)]_{ij},\qquad
        \bm{M}_\omega:=\begin{pmatrix}-\tilde\tau&-1\\ \tilde\tau (1-2\omega)&1-\omega\end{pmatrix}.
    \end{equation}

    Next, for $0<\tilde\tau\le\frac27$, a direct computation gives
    \begin{align*}
        \bm{\Pi}-\bm{M}_0^\top\bm{\Pi}\bm{M}_0&=\begin{pmatrix}\tilde\tau-\tilde\tau^2+\tilde\tau^3&\tilde\tau^2\\ \tilde\tau^2&\tilde\tau\end{pmatrix}\succeq\frac{\tilde\tau}{4}\bm{I},\\
        \bm{\Pi}-\bm{M}_1^\top\bm{\Pi}\bm{M}_1&=\begin{pmatrix}\tilde\tau-\tilde\tau^2-3\tilde\tau^3&\tilde\tau-2\tilde\tau^2\\ \tilde\tau-2\tilde\tau^2&1-\tilde\tau\end{pmatrix}\succeq\frac{\tilde\tau}{4}\bm{I},
    \end{align*}
    as both shifted matrices have positive leading entries and determinants on $(0,\frac27]$. Since $\bm{M}_\omega=(1-\omega)\bm{M}_0+\omega\bm{M}_1$ and $\bm{M}\mapsto\bm{y}^\top\bm{M}^\top\bm{\Pi}\bm{M}\bm{y}$ is convex, we obtain $\bm{M}_\omega^\top\bm{\Pi}\bm{M}_\omega\preceq\bm{\Pi}-\frac{\tilde\tau}{4}\bm{I}$ for all $\omega\in[0,1]$. Expanding \eqref{Eq: increment dynamics} with $\|\bm{\Pi}\|\le\frac97$ and $\|\bm{M}_\omega\|\le\frac32$, applying Young's inequality, and summing over $(i,j)$ yields
    \begin{equation}
        \label{Eq: tracking descent in increments}
        H_{k+1}-H_k\le-\frac{\tilde\tau}{8}\bigl(\|\bB^k-\bB^{k-1}\|^2+\|\bm{V}^k\|^2\bigr)+73\tilde\tau\|\mathcal{D}(\bx^{k+1})-\mathcal{D}(\bx^k)\|^2.
    \end{equation}
    With $\|\mathcal{D}(\bx^{k+1})-\mathcal{D}(\bx^k)\|^2\le2G_{\mathcal{D}}^2\Delta_{\bx}^{k+1}$ by \Cref{Ass: basic regularity}(ii), the first estimate \eqref{Eq: auxiliary tracking descent} holds.

    Finally, from the relation for $\bB^{k+1}-\bB^k$ and $\tilde\tau<1$, we have $\|\bB^{k+1}-\bB^k\|^2\le3(\|\bB^k-\bB^{k-1}\|^2+\|\bm{V}^k\|^2)+3\tilde\tau^2\|\mathcal{D}(\bx^{k+1})-\mathcal{D}(\bx^k)\|^2$. Bounding the first term by \eqref{Eq: tracking descent in increments} and using $3\tilde\tau^2<1$ gives $\|\bB^{k+1}-\bB^k\|^2\le\frac{24}{\tilde\tau}(H_k-H_{k+1})+1753\|\mathcal{D}(\bx^{k+1})-\mathcal{D}(\bx^k)\|^2$. Since $\Delta_{\bB}^{k+1}=\frac12\|\bB^{k+1}-\bB^k\|^2$, this proves the second estimate \eqref{Eq: tracking of B increments}.
\end{proof}
\endgroup

We next show that the augmented potential of DualHash-StoRM is bounded below.
\begingroup\color{blue}
\begin{lemma}[Lower bound of the augmented $\Psi_r$]
\label{Lemma: lower bound of W}
Under Assumptions~\ref{Ass: basic regularity} and~\ref{Ass: main assumptions} and the parameter conditions \eqref{Eq: the parameters conditions in dualhash-storm}, for every $k\ge1$,
\[
\EE[\Psi_r(\calQ^k)]+\frac{\eta_T}{\rho_T}\EE[\|\be^k\|^2]\ge\Psi_r^*:=\underline{\ell}-c_1,
\]
where \blue{$\underline{\ell}$ is given in Assumption~\ref{Ass: basic regularity}(i), and} $c_1$ is defined in \eqref{eq:lb-common} and is independent of $T$.
\end{lemma}

\begin{proof}
By \eqref{Eq: the update of x in STORM} and \eqref{eq:lb-gradient},
\[
C_3\Delta_{\bx}^{k+1}
=\frac{\tau L_F^2\eta_T^2}{4}\|\nabla_{\bx}F(\bx^k,\bB^k)+\be^k\|^2
\le\tau L_F^3\eta_T^2\bigl(F(\bx^k,\bB^k)-\underline{\ell}\bigr)+\frac{\tau L_F^2\eta_T^2}{2}\|\be^k\|^2 .
\]
Since $\tilde L_F\ge L_F$, $\eta_T\tilde L_F\le\frac12$ and $\rho_T\le\frac12$ give $\eta_T\le\frac{1}{2L_F}$ and $\eta_T\rho_T\le\frac{1}{4L_F}$. Therefore
\[
\tau L_F^3\eta_T^2\le\frac t4<1,
\qquad
\frac{\tau L_F^2\eta_T^2}{2}=\frac{\tau L_F^2\eta_T\rho_T}{2}\cdot\frac{\eta_T}{\rho_T}\le\frac t8\cdot\frac{\eta_T}{\rho_T}<\frac{\eta_T}{\rho_T}.
\]
Using \eqref{eq:lb-common} and $C_2\Delta_{\bB}^k\ge0$, we obtain
\begin{align*}
\EE[\Psi_r(\calQ^k)]+\frac{\eta_T}{\rho_T}\EE[\|\be^k\|^2]
&\ge\underline{\ell}+\left(1-\frac t4\right)\EE[F(\bx^k,\bB^k)-\underline{\ell}]+\left(1-\frac t8\right)\frac{\eta_T}{\rho_T}\EE[\|\be^k\|^2]-c_1\\
&\ge\underline{\ell}-c_1 .
\end{align*}
\end{proof}
\endgroup

We provide a recursive bound on the gradient error vector $\{\be^{k}\}$ with analysis comparable to those in~\cite[Lemma~2\blue{.2}]{xu2023momentum} and~\cite[Lemma~8]{shi2025momentum}. \blue{We provide the proof for completeness.}
\begin{lemma}[Recursive bound on gradient error]
    \label{Lemma: Recursive bound on gradient error} Under \Cref{Ass: main assumptions}\blue{ and \Cref{Ass: the mean squared smoothness of F}},
    \blue{for all $k \geq 1$},
    \begin{align}
        \blue{\EE}[\| \be^{k+1}\|^{2}] \leq \left(1 - \rho_{k}\right)^{2}\blue{\EE}[\| \be^{k}\|^{2}] + \frac{2 \rho_{k}^{2}\sigma^{2}}{\blue{b_{k+1}}}+ \frac{4(1 - \rho_{k})^{2}L_{F}^{2}}{\blue{b_{k+1}}}\blue{\EE}\left[\Delta_{\bx}^{k+1}\blue{{}+\Delta_{\bB}^{k+1}}\right]. \notag
    \end{align}
\end{lemma}
\begin{proof}
    \blue{By \eqref{Eq: the stochastic approximate in STORM}, the definition of $\be^{k}$, and $\EE_{k+1}[\tilde{\nabla}_{\bx}F(\bx^{j},\bB^{j};\calJ_{k+1})]=\nabla_{\bx}F(\bx^{j},\bB^{j})$ for $j\in\{k,k+1\}$, we obtain
    \begin{align}
        \EE_{k+1}[\|\be^{k+1}\|^{2}]
        \leq & \,(1-\rho_k)^{2}\|\be^{k}\|^{2}+\frac{2\rho_k^{2}\sigma^{2}}{b_{k+1}}\notag\\
        & \,+\frac{2(1-\rho_k)^{2}}{b_{k+1}}\cdot\frac1n\sum_{i=1}^{n}\|\nabla_{\bx}F_i(\bx^{k+1},\bb_i^{k+1})-\nabla_{\bx}F_i(\bx^{k},\bb_i^{k})\|^{2}\notag\\
        \leq & \,(1-\rho_k)^{2}\|\be^{k}\|^{2}+\frac{2\rho_k^{2}\sigma^{2}}{b_{k+1}}+\frac{4(1-\rho_k)^{2}L_F^{2}}{b_{k+1}}\left(\Delta_{\bx}^{k+1}+\Delta_{\bB}^{k+1}\right)\quad\text{a.s.},\notag
    \end{align}
    where the first inequality follows from $\|\bm a+\bm b\|^{2}\leq2\|\bm a\|^{2}+2\|\bm b\|^{2}$, \Cref{Ass: main assumptions}, and the mini-batch variance bound \blue{under the sampling rule in Section~\ref{Sec: DualHash framework}}, and the last inequality follows from \Cref{Ass: the mean squared smoothness of F}. Taking the full expectation completes the proof.}
\end{proof}

We now establish a weighted error accumulation bound for the variance-reduced gradient estimator in \Cref{alg:our_algorithm_storm}.
\begin{lemma}[Weighted error accumulation bound]
    \label{Lemma: Weighted error accumulation} Under Assumption\blue{s}~\ref{Ass: basic regularity}, \ref{Ass: main assumptions}, \blue{and~\ref{Ass: the mean squared smoothness of F}} and parameter conditions \blue{\eqref{Eq: the parameters conditions in dualhash-storm}}, for any $T \geq 1$, the following weighted cumulative error bound holds:
    \begin{align}
        \blue{\frac{1}{T}\sumT\EE\Bigl[C_{\bB}\bigl(\Delta_{\bB}^{k+1}+\Delta_{\bB}^{k}\bigr)+\frac{1}{4\eta_T}\Delta_{\bx}^{k+1}+\frac{\eta_T}{2}\|\be^{k}\|^{2}\Bigr]\leq\frac{1}{T}\Bigl(\tilde{\Delta}_{1}+\frac{\eta_T\sigma^{2}}{\rho_Tb_1}\Bigr)+\frac{2\eta_T\rho_T\sigma^{2}}{b},}
        \label{Eq: weighted error accumulation bound}
    \end{align}
    where $\tilde{\Delta}_{1}=\blue{\calL(\bx^{1},\bB^{1},\bLmbd^{1})}-\Psi_r^{*}$.
\end{lemma}
\begin{proof}
    Under \blue{\eqref{Eq: the parameters conditions in dualhash-storm}, $\eta_T\tilde L_F\le\frac12$, so} we have a lower bound for $C_{\bx}^{k}$ with $C_{\bx}^{k}\geq\frac{1}{2\blue{\eta_T}}>0$. Substituting this inequality into \eqref{Eq: the descent property of Lyapunov func in storm} yields
    \begin{align}
        \EE[\Psi_r(\calQ^{k+1})]-\EE[\Psi_r(\calQ^{k})]\leq\blue{-C_{\bB}\EE[\Delta_{\bB}^{k+1}+\Delta_{\bB}^{k}]}-\frac{1}{2\blue{\eta_T}}\EE[\Delta_{\bx}^{k+1}]+\frac{\blue{\eta_T}}{2}\EE[\|\be^{k}\|^{2}].
        \label{Eq: descent of Psi_r in weighted error bound}
    \end{align}
    From \Cref{Lemma: Recursive bound on gradient error} and
    \Cref{lem:auxiliary-tracking}, we obtain
    \begin{align}
        &\blue{\EE[\|\be^{k+1}\|^{2}]+\frac{48L_F^{2}}{b\tilde\tau}\EE[H_{k+1}]}\notag\\
        &\blue{\leq(1-\rho_T)\EE[\|\be^{k}\|^{2}]+\frac{48L_F^{2}}{b\tilde\tau}\EE[H_{k}]+\frac{2\rho_T^{2}\sigma^{2}}{b}+\frac{4L_F^{2}(1+G)}{b}\EE[\Delta_{\bx}^{k+1}],}
        \label{Eq: error recursion with tracking}
    \end{align}
    where the inequality follows from \blue{$(1-\rho_T)^{2}\leq1-\rho_T$}. \blue{Multiplying \eqref{Eq: error recursion with tracking} by $\frac{\eta_T}{\rho_T}$, adding it to \eqref{Eq: descent of Psi_r in weighted error bound}, and applying $\frac{\eta_T}{\rho_T}\cdot\frac{4L_F^{2}(1+G)}{b}\leq\frac{1}{4\eta_T}$, which is equivalent to the lower bound on $\rho$ in \eqref{Eq: the parameters conditions in dualhash-storm},} yields the following expected descent property of the Lyapunov function $\Psi_r$:
    \begin{align}
        &\blue{\EE\Bigl[\Psi_r(\calQ^{k+1})+\frac{\eta_T}{\rho_T}\Bigl(\|\be^{k+1}\|^{2}+\frac{48L_F^{2}}{b\tilde\tau}H_{k+1}\Bigr)\Bigr]-\EE\Bigl[\Psi_r(\calQ^{k})+\frac{\eta_T}{\rho_T}\Bigl(\|\be^{k}\|^{2}+\frac{48L_F^{2}}{b\tilde\tau}H_{k}\Bigr)\Bigr]}\notag\\
        &\blue{\leq-\EE\Bigl[C_{\bB}\bigl(\Delta_{\bB}^{k+1}+\Delta_{\bB}^{k}\bigr)+\frac{1}{4\eta_T}\Delta_{\bx}^{k+1}+\frac{\eta_T}{2}\|\be^{k}\|^{2}\Bigr]+\frac{2\eta_T\rho_T\sigma^{2}}{b}.}
        \label{Eq: expected descent of Psi_r with error}
    \end{align}
    Finally, summing up \eqref{Eq: expected descent of Psi_r with error} from $k=1$ to $T$ and dividing by $T$, \blue{together with $H_1=\Delta_{\bB}^{1}=0$, $\EE[\|\be^{1}\|^{2}]\leq\sigma^{2}/b_1$, $H_{T+1}\geq0$, and \Cref{Lemma: lower bound of W},} yields \eqref{Eq: weighted error accumulation bound} and completes the proof.
\end{proof}

We derive a pathwise upper stationarity error bound for the iterates generated by \Cref{alg:our_algorithm_storm}.
\begin{lemma}[Stationarity error bound]
    \label{Lemma: Stationarity error bound} Under \Cref{Ass: basic regularity}(ii),
    \blue{for all $k \geq 1$},
    \begin{align}
        \mathrm{dist}^{2}(\bm{0}, \mathcal K(\bx^{k}, \bB^{k}, \bm{\Lambda}^{k})) \leq S_{1}\Delta_{\bB}^{k+1}+ S_{2}\Delta_{\bB}^{k}+ S_{3}^{k}\Delta_{\bm{x}}^{k+1}+ 2 \| \be^{k}\|^{2}, \label{Eq: the stationarity error bound in Lemma in storm}
    \end{align}
    where the coefficients are defined as \eqref{Eq: the coefficients of stationarity error bound in SGDM}.
\end{lemma}
\begin{proof}
    This proof follows a similar approach to that of \Cref{Lemma: Stationarity error bound in SGDM} 
in \Cref{app:stom-complexity}. The key difference lies in 
the estimation of $\| \bd_{\bx}^{k}\|^{2}$. To estimate $\| \bd_{\bx}^{k}\|^{2}$, we use 
the optimality condition of $\bx^{k+1}$ from \eqref{Eq: the update of x in STORM} to obtain
    \begin{align}
        \|\bd_{\bx}^{k}\|^{2}= & \,\eta_{k}^{-2}\left\| \eta_{k}\nabla_{\bx}F(\bx^{k}, \bm{B}^{k}) + 0\right\|^{2}\notag                                       \\
        =                      & \, \eta_{k}^{-2}\left\|\eta_{k}\left( \nabla_{\bx}F(\bx^{k}, \bm{B}^{k}) - \blue{\bg_s^{k}}\right) + \eta_{k}\blue{\bg_s^{k}}\right\|^{2}\notag \\
        =                      & \, \eta_{k}^{-2}\left\|\blue{-}\eta_{k}\be^{k}+ (\bx^{k}- \bx^{k+1})\right\|^{2}\notag                                                \\
        \leq                   & \, 4\eta_{k}^{-2}\Delta_{\bx}^{k+1}+ 2 \| \be^{k}\|^{2}.\label{Eq: dx in storm}
    \end{align}
    Then, substituting \eqref{Eq: dx in storm}, \eqref{Eq: dB}, and \eqref{Eq: dLambda}
    into \eqref{Eq: the stationarity error bound in proof} yields \eqref{Eq: the stationarity error bound in Lemma in storm}
    and completes the proof. 
\end{proof}
\subsection{Proof of \Cref{Lemma: convergence rate of storm}}
\label{app:storm-rate-proof}
\begin{proof}
    By taking the full expectation on both sides of \blue{\eqref{Eq: the stationarity error bound in Lemma in storm} in \Cref{Lemma: Stationarity error bound}} and averaging over $T$, we obtain
    \[
        \blue{\frac1T\sumT\EE[\mathrm{dist}^{2}(\bm{0},\mathcal K(\bx^k,\bB^k,\bLmbd^k))]
        \leq\frac1T\sumT\EE[S_1\Delta_{\bB}^{k+1}+S_2\Delta_{\bB}^{k}+S_3^k\Delta_{\bx}^{k+1}+2\|\be^{k}\|^{2}].}
    \]
    \blue{By the definition of $\tilde{L}_{\Delta}$, we have $S_1,S_2\leq\tilde{L}_{\Delta}T^{1/3}C_{\bB}$, $4\eta_TS_3^k\leq\tilde{L}_{\Delta}T^{1/3}$, and $\frac{4}{\eta_T}\leq\tilde{L}_{\Delta}T^{1/3}$.} Then, substituting \blue{\eqref{Eq: weighted error accumulation bound} from \Cref{Lemma: Weighted error accumulation}} into the above inequality yields
    \[
        \blue{\frac1T\sumT\EE[\mathrm{dist}^{2}(\bm{0},\mathcal K(\bx^k,\bB^k,\bLmbd^k))]
        \leq\frac{\tilde{L}_{\Delta}}{T^{2/3}}\Bigl(\tilde\Delta_1+\frac{\eta_T\sigma^2}{\rho_Tb_1}+\frac{2T\eta_T\rho_T\sigma^2}{b}\Bigr).}
    \]
    \blue{Substituting $\frac{\eta_T}{\rho_T}=\frac{T^{1/3}}{8\rho\eta\tilde L_F}$, $2T\eta_T\rho_T=\frac{16\rho\eta^3}{\tilde L_F}$, and $\tilde{C}_{\sigma}=\frac{\eta\tilde{L}_{\Delta}}{2\tilde L_F}$,} we obtain \eqref{Eq: the convergence rate of storm}.
\end{proof}
With this parameter choice, we proceed to analyze the oracle complexity of \Cref{alg:our_algorithm_storm}.
\subsection{Proof of \Cref{Theorem: the oracle complexity in STORM}}
\label{app:storm-oracle-proof}
\begin{proof}
    When $b_1=\lceil c_bT^{1/3}\rceil$, \blue{since $R$ is drawn uniformly from $\{1,\ldots,T\}$, independently of the mini-batches, the law of total expectation and \eqref{Eq: the convergence rate of storm} give}
    \[
        \EE[\mathrm{dist}^{2}(\bm{0}, \mathcal K(\bx^{R}, \bB^{R}, \bm{\Lambda}^{R}))]\le\frac{1}{T^{2/3}}\Bigl(\tilde{L}_{\Delta}\tilde{\Delta}_{1}+\tilde{C}_{\sigma}\Bigl(\frac{\sigma^{2}}{4\rho\eta^{2}c_b}+\frac{32\rho\eta^{2}\sigma^{2}}{b}\Bigr)\Bigr).
    \]
    With the conditions on $T$ in \eqref{Eq: the iterate number upper bound of algorithm VR}, this gives $\EE[\mathrm{dist}^{2}(\bm{0}, \mathcal K(\bx^{R}, \bB^{R}, \bm{\Lambda}^{R}))]\le\varepsilon^2$. Then, the oracle complexity is
    \[
        \text{Complexity}= b_{1}+ \blue{2}(T-1)b = \calO(T^{\frac{1}{3}}) + \calO(T) = \calO(T) = \calO(\varepsilon^{-3}),
    \]
    which completes the proof.
\end{proof}
\subsection{\blue{Proof of \Cref{cor:original-stationarity}}}
\label{app:original-objective}
\begin{proof}[Proof of \Cref{cor:original-stationarity}]
\blue{Apply \Cref{prop:original-certificate} \blue{with $L=L_0$} at $(\bx^R,\bB^R,\bLmbd^R)$ and take expectations, using $\EE[r_{\mathcal{K}}]\le(\EE[r_{\mathcal{K}}^2])^{1/2}$ and $\EE[\sqrt{h}]\le(\EE[h])^{1/2}$ by Jensen's inequality.} \blue{This gives
\begin{align*}
\EE[\varepsilon_R] &\le\Bigl(\sqrt{1+G_{\mathcal{D}}^2}+2\sqrt{2\kappa L_0}\Bigr)\varepsilon+2\sqrt{2L_0\,\EE[h(\mathcal{D}(\bx^R))]},\\
\EE[\delta_R] &\le\sqrt{\frac{2}{L_0}}\Bigl(\sqrt{\kappa}\,\varepsilon+\sqrt{\EE[h(\mathcal{D}(\bx^R))]}\Bigr),
\end{align*}} completing the proof.
\end{proof}

\section{Experimental Details}
\label{app:experiments}
Let $\bA = \{\ba_i \in \RR^{d_a} : i = 1, \ldots, n\}$ be the training set with label vectors $\by_i \in \{0,1\}^C$ (one-hot for single-label datasets and multi-hot for multi-label datasets), where a CNN backbone followed by a fully connected hash layer produces continuous codes $\blue{\calD_i}(\bx)\in\RR^d$ and post-training binary codes $\operatorname{sgn}(\blue{\calD_i}(\bx))\in\{-1,1\}^d$. Following~\cite{xiaSupervisedHashingImage2014,laiSimultaneousFeatureLearning2015}, let $s^o_{ij}=1$ if samples $i,j$ share labels and $0$ otherwise; using $s^h_{ij}=\tfrac{1}{2}\langle \blue{\calD_i}(\bx), \blue{\calD_j}(\bx)\rangle$ (the continuous counterpart of the inner product in the Hamming distance) and the sigmoid likelihood $P(s^o_{ij}\mid s^h_{ij}) = \operatorname{sig}_{\kappa}(s^h_{ij})^{s^o_{ij}}(1-\operatorname{sig}_{\kappa}(s^h_{ij}))^{1-s^o_{ij}}$ with $\operatorname{sig}_{\kappa}(t)=(1+\exp(-\kappa t))^{-1}$ ($\kappa<1$), the negative log-likelihood over sampled pairs $\calS$ gives the pairwise cross-entropy loss
\begin{equation}
\label{Eq: the pairwise CE loss}
\ell_{\calS}(\bx) = \frac{1}{|\calS|}\sum_{s^o_{ij}\in\calS}\Bigl[\log\bigl(1+\exp(\kappa s^h_{ij})\bigr)-\kappa\,s^o_{ij}s^h_{ij}\Bigr].
\end{equation}
Substituting~\eqref{Eq: the pairwise CE loss} and the $\ell_1$ W-regularizer into Problem~\eqref{Prop: block-wise regularized problem} yields the training objective
\begin{align*}
\min_{\bx \in \RR^q, \bB \in \RR^{n\times d}} \;
&\frac{1}{|\calS|}\sum_{s^o_{ij}\in\calS}\Bigl[\log\bigl(1+\exp(\kappa s^h_{ij})\bigr)-\kappa\,s^o_{ij}s^h_{ij}\Bigr] +\frac{\mu}{2}\|\bx\|^2 \\
&+ \frac{\gamma}{2n}\sum_{i=1}^n \| \bb_i - \blue{\calD_i}(\bx)\|_2^2
+ \sum_{i=1}^n \lambda \|\,|\bb_i| - \bm{1}\,\|_1,
\end{align*}
and for an unseen query $\ba^{\mathrm{qry}} \notin \bA$, its binary code is predicted as $\operatorname{sgn}(\blue{\calD_{\mathrm{qry}}(\bx)})$. 

\end{document}